\documentclass[12pt]{article}   
   
\usepackage{latexsym,amsfonts,amsmath,epsfig,tabularx,amsthm,amssymb,mathrsfs,enumitem,bm}
\usepackage{float}
\usepackage{mathtools}
\usepackage{aligned-overset}

\usepackage{microtype}
\usepackage[normalem]{ulem} 
   
\usepackage[draft=false,colorlinks,bookmarksnumbered,urlcolor=black,linkcolor=black, citecolor=blue]{hyperref}

\usepackage{aliascnt}

\newtheorem{theorem}{Theorem}[section]

\newaliascnt{lem}{theorem}
\newtheorem{lemma}[lem]{Lemma}

\aliascntresetthe{lem}

\newaliascnt{ass}{theorem}

\aliascntresetthe{ass}

\newaliascnt{prop}{theorem}
\newtheorem{prop}[prop]{Proposition}

\aliascntresetthe{prop}

\newaliascnt{cor}{theorem}
\newtheorem{corollary}[cor]{Corollary}

\aliascntresetthe{cor}

\newaliascnt{defi}{theorem}
\newtheorem{defi}[defi]{Definition}

\aliascntresetthe{defi}

\theoremstyle{definition}
\newaliascnt{ex}{theorem}

\aliascntresetthe{ex}

\newaliascnt{conv}{theorem}
\newtheorem{convention}[conv]{Convention}

\aliascntresetthe{conv}

\newaliascnt{rem}{theorem}
\newtheorem{remark}[rem]{Remark}

\aliascntresetthe{rem}

\newcommand{\Regd}{\ensuremath{\textup{RD}}}
\renewcommand{\d}{\,\mathrm{d}}											
\renewcommand*{\epsilon}{\varepsilon}                                   
\renewcommand*{\rho}{\varrho}                                   		
\newcommand*{\sep}{\; \vrule \;}                                        
\newcommand*{\sepb}{\; \big| \;}                                        
\newcommand*{\N}{\mathbb{N}}                                            
\newcommand*{\R}{\mathbb{R}}                                            
\newcommand*{\real}{\mathbb{R}}         
\newcommand*{\C}{\mathbb{C}}                                            
\newcommand*{\Z}{\mathbb{Z}}                                            
\newcommand*{\B}{\mathcal{B}}                                           
\renewcommand*{\L}{\mathcal{L}}                                         	
\newcommand*{\D}{\mathcal{D}}                                         	
\newcommand*{\I}{\mathcal{I}}                                         	
\newcommand*{\0}{\mathcal{O}}                                           

\newcommand*{\abs}[1]{\left| #1 \right|}                                
\newcommand*{\norm}[1]{\left\| #1 \right\|}                             
\newcommand*{\normmod}[1]{\left\|\hspace{-1pt}\left| #1 \right|\hspace{-1pt}\right\|}                 
\newcommand*{\floor}[1]{\left\lfloor #1 \right\rfloor}                  
\newcommand*{\ceil}[1]{\left\lceil #1 \right\rceil}                     

\newcommand*{\dist}[2]{\mathrm{dist}\!\left( #1, #2 \right)}            

\renewcommand{\tilde}[1]{ \widetilde{#1} }        						
\DeclareMathOperator{\supp}{supp}										

\newcommand{\loc}{\mathrm{loc}}

\newcommand{\beq}{\begin{equation}}
\newcommand{\eeq}{\end{equation}}
\newcommand{\bit}{\begin{itemize}}
\newcommand{\eit}{\end{itemize}}

\usepackage{tikz}
\usetikzlibrary{arrows,decorations.pathmorphing,decorations.pathreplacing,backgrounds,positioning,fit,petri, patterns}

\setlist{itemsep=2pt, topsep=2pt}

\newcommand{\cont}{\ensuremath{\mathcal{C}}}

\newcommand{\Abs}[1]{\ensuremath{\big\lvert  #1 \big\rvert}}

\newcommand{\grklam}[1]{\big(#1\big)}
\newcommand{\sgrklam}[1]{\Big(#1\Big)}
\newcommand{\ssgrklam}[1]{\bigg(#1\bigg)}

\newcommand{\ggklam}[1]{\big\{#1\big\}}
\newcommand{\sggklam}[1]{\Big\{#1\Big\}}
\newcommand{\ssggklam}[1]{\bigg\{#1\bigg\}}

\newcommand{\eklam}[1]{\left[#1\right]}
\newcommand{\geklam}[1]{\big [#1 \big ]}

\newcommand{\dx}{\mathrm d x}
\newcommand{\dr}{\mathrm d r}

\newcommand{\gdomain}{\ensuremath{G}}   
\newcommand{\domain}{\ensuremath{\mathcal{O}}}   

\DeclareFontFamily{U}{matha}{\hyphenchar\font45}
\DeclareFontShape{U}{matha}{m}{n}{
      <5> <6> <7> <8> <9> <10> gen * matha
      <10.95> matha10 <12> <14.4> <17.28> <20.74> <24.88> matha12
      }{}
\DeclareSymbolFont{matha}{U}{matha}{m}{n}
\DeclareFontSubstitution{U}{matha}{m}{n}
\DeclareFontFamily{U}{mathx}{\hyphenchar\font45}
\DeclareFontShape{U}{mathx}{m}{n}{
      <5> <6> <7> <8> <9> <10>
      <10.95> <12> <14.4> <17.28> <20.74> <24.88>
      mathx10
      }{}
\DeclareSymbolFont{mathx}{U}{mathx}{m}{n}
\DeclareFontSubstitution{U}{mathx}{m}{n}
\DeclareMathDelimiter{\vvvert}{0}{matha}{"7E}{mathx}{"17}

\newcommand{\scco}{\color{black}}

\title{Sobolev spaces with mixed weights and the\\ Poisson equation on angular domains}
\date{}
 
\author{
Petru A. Cioica-Licht\footnote{\emph{Corresponding author}. Institute of Mathematics, University of Kassel, Heinrich-Plett Str.~40, 34132 Kassel, Germany. Email: \href{mailto:cioica-licht@mathematik.uni-kassel.de}{cioica-licht@mathematik.uni-kassel.de}}
\qquad
Cornelia Schneider\footnote{Friedrich-Alexander University Erlangen-Nuremberg, Applied Mathematics III, Cauerstr. 11, 91058 Erlangen, Germany. Email: \href{mailto:cornelia.schneider@math.fau.de}{cornelia.schneider@math.fau.de}. 
The work of this author has been supported by Deutsche Forschungsgemeinschaft (DFG), Grant No. SCHN 1509/1-2.}
\qquad	
Markus Weimar\footnote{Julius-Maximilians-Universit\"at W\"urzburg, Faculty of Mathematics and Computer Science, Mathematics~IX (Scientific Computing), Emil-Fischer-Stra{\ss}e 30,
97074 W\"urzburg, Germany. Email: \href{mailto:markus.weimar@uni-wuerzburg.de}{markus.weimar@uni-wuerzburg.de}
}
}

\begin{document}   
\maketitle   

\vspace*{-2em}

\smallskip

\begin{abstract}
\noindent 
We introduce and analyse a class of weighted Sobolev spaces with mixed weights on angular domains.
The weights are based on both the distance to the boundary and the distance to the one vertex of the domain.
Moreover, we show how the regularity of the Poisson equation can be analysed in the framework of these spaces by means of the Mellin transform, provided the integrability parameter equals two. 
Our main motivation comes from the study of stochastic partial differential equations and associated degenerate deterministic parabolic  equations.

\smallskip


\end{abstract}

\tableofcontents

\section{Introduction}
\label{sec:Introduction}

In this paper we present a thorough analysis of a class of   weighted Sobolev spaces $H^\gamma_{p,\Theta,\theta}(\D)$ involving mixed weights on angular domains
\begin{equation}\label{eq:D}
\D
:=
\D_\kappa
:=
\ggklam{
x=(x_1,x_2)\in\R^2
\sep
x=(r\cos(\phi),r\sin(\phi)),\, 0<r<\infty,\,0<\phi<\kappa
}\subset\R^2
\end{equation}
with arbitrary angle $0<\kappa<2\pi$. Recently, these spaces  have been  used for the regularity analysis of stochastic partial differential equations (SPDEs) and related degenerate partial differential equations (PDEs) on angles and polygons~\cite{Cio20,CioKimLee2019,CioKimLee+2018,KimLeeSeo2021,KimLeeSeo2022b}.
Moreover, we initiate the analysis of the Dirichlet Laplacian on $\D$  in this scale of  Sobolev spaces by proving existence and uniqueness of solutions to the Poisson equation
\begin{equation}\label{eq:Poisson}
\Delta u =f \,\text{ on }\,\D,\qquad u=0\,\text{ on }\,\partial\D,
\end{equation}
within the aforementioned spaces.
To keep the manuscript at a reasonable length, we restrict the analysis of the Poisson equation to the case where the integrability parameter $p$ equals two and postpone the general case to a forthcoming paper. 
Both, the study of the spaces and the analysis of the Dirichlet Laplacian in these spaces, are important steps that are needed for generalizations of~\cite{Cio20,CioKimLee2019,CioKimLee+2018,KimLeeSeo2021,KimLeeSeo2022b} towards a refined $L_p$-theory for SPDEs on non-smooth domains.

The context that motivates our analysis can be roughly summarized as follows: Sobolev spaces provide a natural framework for the regularity analysis of PDEs.  
The usual unweighted Sobolev spaces work particularly well for deterministic, non-degenerate 
PDEs on smooth domains~\cite{AdaFou2003,Eva2002,Kry2008}.
However, they are not very well suited for the analysis of equations that do not satisfy these `classical' assumptions.
In that case, weighted Sobolev spaces turn out to be a  viable alternative,
in particular, in the following situations: 
\medskip
\begin{itemize}
\item \emph{PDEs on non-smooth domains. } Singularities at the boundary of the underlying domain, i.e., corners, edges, cusps, and any other points where the boundary is not sufficiently smooth, are known to lead to singularities of solutions to PDEs.
This results in a breakdown of the  (unweighted) Sobolev regularity of higher order, see e.g.~\cite{CioWei2020,Cos2019,Gri1985,JK95}; see also \cite{Lin2014} for the same effect for SPDEs.
However, the singularities of the solution can often be described accurately by means of Sobolev spaces with weights that involve the distance to the set 
of boundary singularities.
This idea goes back to Kondratiev~\cite{Kon1967,Kon1970}, followed by an abundant number of related papers and monographs.
In this context we only mention~\cite{Dau1988,Gri1985,Gri1992, HarWei2018, KozNaz2014, MazRos2010, RoiSchSei2021, Sol2001} and the references therein. This list is by no means complete.

\item \emph{PDEs that degenerate at the boundary. } If the underlying domain is sufficiently smooth (usually, at least $C^1$ is required) but the equation is degenerate at the boundary in the sense that, for instance, the coefficients are not uniformly elliptic towards the boundary or the forcing terms have blow-ups at the boundary, then weighted Sobolev spaces based on the distance to the entire boundary have proven useful. Again, there is a long list of publications on this topic, of which we mention just a few~\cite{HumLin2021,KimKry2004,Kry1999c,Kuf1980,LinLorRoo+2024+,LinVer2020,LioMag1968,Tri1995}. 

\item \emph{Stochastic PDEs.}
In~\cite{Fla1990,Kry1994} it has been shown that, even if the underlying domain and the coefficients are smooth, solutions to \emph{stochastic} PDEs may fail to have higher order unweighted Sobolev regularity. 
This is due to the roughness of the noise and a resulting  incompatibility between noise and boundary conditions, which leads to blow-ups of the higher order derivatives of the solution along the boundary.
However, in a series of papers~\cite{Kim2004, Kim2008,Kim2014,Kim2017,Kry1994,KryLot1999b,KryLot1999} initiated by N.V.~Krylov it has been demonstrated that, as long as the underlying domain $\domain\subset\R^d$ is of class $C^1$, second order SPDEs with zero Dirichlet boundary conditions can be analysed very accurately by means of certain weighted Sobolev spaces $H^\gamma_{p,\Theta}(\domain)$; see \autoref{sec:WSob:Lototsky} for a definition and the basic properties of these spaces.
\end{itemize}
Thus,
by means of appropriate weighted Sobolev spaces, 
a fairly comprehensive $L_p$-theory can be established for non-degenerate PDEs on non-smooth domains on the one hand and for degenerate PDEs as well as for SPDEs on smooth domains on the other hand. 
However, up to now very little is known about the regularity of \emph{degenerate} PDEs and of \emph{stochastic} PDEs on \emph{non-smooth} domains. 
The main challenge in closing this long persisting gap is to find suitable function spaces that capture both the singular behaviour along the boundary (caused by the noise and/or due to the degeneracy of the equation) and the singularities of the solution caused by the singularities of the boundary.

In~\cite{Cio20,CioKimLee2019,CioKimLee+2018} a research program that aims at narrowing this gap for SPDEs and related degenerate parabolic PDEs has been initiated. 
Therein, the focus lies on the stochastic heat equation with zero Dirichlet boundary condition on the angular domains $\D=\D_\kappa$  introduced above
as well as on polygonal domains.
As has been shown in~\cite{Cio20,CioKimLee2019}, see also~\cite{KimLeeSeo2021,KimLeeSeo2022b}, in this setting the different types of singularities described above and their interplay can be captured accurately by means of certain weighted Sobolev spaces $H^\gamma_{p,\Theta,\theta}(\D)$ with $\gamma\in\N_0$, $1<p<\infty$, and $\Theta,\theta\in\R$, which consist of (equivalence classes of) locally integrable scalar-valued functions $u$ on  $\D$  such that
\begin{equation}\label{eq:norm:intro}
    \sum_{\alpha\in\N_0^2\colon \abs{\alpha}\leq \gamma } \int_\D \abs{\rho_\D^{\abs{\alpha}}\, D^\alpha u}^p  \rho_\circ^{\theta-2}\sgrklam{\frac{\rho_\D}{\rho_\circ}}^{\Theta-2}\,\dx<\infty,
\end{equation}
where $\rho_\circ:=\dist{\cdot}{\{0\}}$ and $\rho_\D:=\dist{\cdot}{\partial\D}$ are the distances to the corner and to the boundary of $\D$, respectively.
In these spaces, existence, uniqueness, and higher order regularity for the stochastic heat equation on $\D$ can be established for sharp ranges of weight parameters $\Theta,\theta\in\R$; cf.~\cite{Cio20,KimLeeSeo2021,KimLeeSeo2022}. Due to the nature of the problem, this is neither possible in unweighted Sobolev spaces nor in the spaces $H^\gamma_{p,\Theta}(\D)$ mentioned above (except for a very restricted range of weight parameters $\Theta\in\R$, see~\cite{Kim2014}).
However, so far, the analysis is limited to non-negative integer smoothness parameters $\gamma\in\N_0$. 
Moreover, an analysis of the space-time regularity, including uncoupling of the integrability parameters in time and space as well as sharp initial conditions, has yet to be done.
Among other things, these extensions require a detailed analysis of the spaces $H^\gamma_{p,\Theta,\theta}(\D)$ and  the behaviour of the Dirichlet Laplacian as well as related (degenerate) PDEs within these spaces. These investigations are the subject of this paper.

We choose the following outline: 
In \autoref{sec:Preliminaries} we present some preliminaries which will be needed throughout the manuscript. 
\autoref{sec:WSobAngle} is dedicated to the detailed study of the weighted Sobolev spaces $H^\gamma_{p,\Theta,\theta}(\D)$ which is inspired by the corresponding analysis of the spaces $H^\gamma_{p,\Theta}(\domain)$ from~\cite{Kry1999c,Lot2000}. 
We first (re)define the spaces  for arbitrary $\gamma\in\R$, $1<p<\infty$, and $\Theta,\theta\in\R$, by means of suitable (approximate) resolutions of unity subordinate to the one vertex of $\D$ at $x=0$ and by means of the weighted Sobolev spaces $H^\gamma_{p,\Theta}(\D)$, see \autoref{def:Hspaces} below. 
Of course, we prove that for non-negative integers $\gamma\in\N_0$ our definition is consistent with the one from \cite[Section~3]{Cio20}.
We then address several properties of these spaces such as    the density of the space of smooth compactly supported functions, interpolation and duality, pointwise multipliers and embeddings---among others. 
\autoref{sec:Poisson} is devoted to the analysis of the Poisson equation within the framework of the spaces $H^\gamma_{2,\Theta,\theta}(\D)$ which mainly relies on a characterization of these spaces in terms of polar coordinates (cf.~\autoref{thm:polar}) and the Mellin transform (cr. Theorem~\ref{thm:Mellinrepresentation}).

\bigskip


\subsubsection*{Notation}\label{sec:Notation}
Before we start, let us fix some notation. 
Let $d\in\N$. Then we let $\abs{\alpha}:=\abs{\alpha}_1=\sum_{j=1}^d \abs{\alpha_j}$ for multi-indices $\alpha:=(\alpha_1,\ldots,\alpha_d)\in\N_0^d$, but $\abs{x}:=\abs{x}_2$ for points $x:=(x_1,\ldots,x_d)\in\R^d$, where $\abs{x}_p:=\left(\sum_{j=1}^d \abs{x_j}^p \right)^{1/p}$ if $0<p<\infty$. We put $\ceil{x}:=\min\{k\in \mathbb{Z}: \ k\geq x\}$ and $\floor{x}:=\max\{k\in \mathbb{Z}: \ k\leq x\}$. Assume $U$ is a set. Whenever we write 
$ 
A(u)\lesssim B(u),\  u\in U,
 \text{ or }  
A\lesssim B \text{ on }U,
$
it means that there is a finite constant $C>0$ that does not depend on $u$, such that $A(u)\leq CB(u)$ for all $u\in U$. 
Sometimes we omit $U$, if it is clear from the context. 
If we want to emphasize that the constant only depends on some parameters $a_1,\ldots,a_n$ for some $n\in\N$, then we write $A(u)\lesssim_{a_1,\ldots,a_n}B(u)$. If $A\lesssim B$ and $B\lesssim A$ on $U$ then we write $A\sim B$ on $U$ or $A(u)\sim B(u)$, $u\in U$. 

For an arbitrary domain $\domain\subsetneq\R^d$ we write $\rho_\domain(x):=\dist{x}{\partial\domain}$ for the distance of a point $x\in\domain$ to the boundary $\partial\domain$ of $\domain$. 
For a (generalised) scalar-valued function $u$ on a domain $\domain\subset\R^d$ and any multi-index $\alpha$ we let $D^\alpha u$ be the $\alpha$ generalized/distributional derivative of $u$ on $\domain$.
We write $\partial^\alpha u$ for the classical derivative -- if it exists. 
$u_{x_i}$ and $u_{x_ix_i}$ is short for the (generalized) first and second order derivative of $u$, respectively, with respect to the variable $x_i$, whereas  $\Delta u = \sum_{i=1}^d u_{x_ix_i}$.
For $k\in\N_0$ we write $D^ku$ for the vector of all $k$-th order partial generalized derivatives of $u$;   similar for   $D:=D^1$ and $\partial$, $\partial^k$.  
 If $u$ is $\C^n$-valued or $\C^{n\times n}$-valued, then $D$ (as well as $D^k$, $D^\alpha$, $\partial$, $\partial^k$, $\partial^\alpha$) is understood component-wise.
 We sometimes need to specify the variable, say $x$, with respect to   which   we differentiate. We do this by writing $D_x$ or $\partial_x$ instead of $D$ or $\partial$. 
For $k\in\N$, $C^k(\domain)$ denotes the space of all $k$-times continuously differentiable scalar-valued functions on $\domain\subset\R^d$, whereas $C^k_0(\domain)$ stands for the spaces of all functions in $C^k(\domain)$ with compact support in $\domain$. 
 Accordingly, $C^\infty(\domain)$ is the space of all infinitely differentiable scalar-valued functions on $\domain\subset\R^d$ and by $C_0^\infty(\domain)\scco$ we denote the spaces of all functions in $C^\infty(\domain)$ with compact support in $\domain$. We write $\mathscr{D}'(\domain)$ for the space of all generalised functions on a domain $\domain\subset\R^d$ and $(u,\varphi):=u(\varphi)$ for $u\in\mathscr{D}'(\domain)$ applied to $\varphi\in C_0^\infty(\domain)$. If $\gdomain\subset\R^d$ is another domain and $\Psi\colon \gdomain\to\domain$ is a $ C^\infty$ diffeomorphism, we write $u\circ\Psi:=\Psi^*u:=\{ C_0^\infty(\gdomain)\ni \psi\mapsto (u,\psi\circ\Psi^{-1}\cdot\abs{\det D\Psi^{-1}}) \}\in \mathscr D'(\gdomain)$  for the pullback of $u\in \mathscr{D}'(\domain)$ w.r.t.\ $\Psi$. 
 
If $(E,\norm{\cdot\sep E})$ is a normed space consisting of (equivalence classes of) scalar-valued functions on $\domain$ and $u=(u_1,\ldots,u_n)^T$ is a vector of scalar-valued functions on $\domain$, then $u\in E$ means that $u_i\in E$ for all $i=1,\ldots,n$, and $\norm{u\sep E}=\sum_{i=1}^n \norm{u_i\sep E}$. 
We use the standard notation $(E',\norm{\,\cdot \sep E'})$ for the topological dual of a normed space $E$, where $\norm{x'\sep E'}:=\sup_{\norm{x\sep E}\leq 1} x'(x)$ for $x'\in E'$. We use the word ``isomorphism'' as follows: Let $(E_1,\norm{\cdot\sep E_1})$ and $(E_2,\norm{\cdot\sep E_2})$ be two normed spaces. We say that a linear mapping $T\colon E_1\to E_2$ is an isomorphism, if $T$ is invertible and $\norm{T(u)\sep E_2}\sim \norm{u\sep E_1}$ for all $u\in E_1$. 
Let $\zeta=(\zeta_\nu)_{\nu\in\Z}$ be a sequence of smooth functions on a domain $\domain  \subset   \R^d$ and let $u\in\mathscr{D}'(\domain)$ be a generalized function. Moreover, let $X$ be a Banach space and $(c_{\nu})_{\nu\in\Z}\subset [0,\infty)$. Whenever we write 
 $
\sum_{\nu\in\Z} c_\nu \norm{\zeta_\nu u \sep X} <\infty,
$ 
we mean that $\zeta_\nu u\in X$ for all $\nu\in\Z$ \emph{and} that the series is finite. Given two quasi-Banach spaces $X$ and $Y$, we write $X\hookrightarrow Y$ if $X\subset Y$ and the natural embedding is bounded.  Moreover, $[X, Y]_{\vartheta}$  denotes the complex interpolation space of exponent $\vartheta$ of the interpolation couple $(X,Y)$, see~\cite[Chapter~4]{BL76}.

If $\gdomain\subset\R^d$, we write $\mathcal B(\gdomain)$ for the Borel $\sigma$-algebra on $\gdomain$. If $\mu$ is a measure on $(\gdomain,\mathcal{B}(\gdomain))$ and $E$ is a Banach space, we write $L_0(\gdomain,\mathcal{B}(\gdomain), \mu;E)$ for the space of all equivalence classes (w.r.t. $\mu$) of Borel-measurable $E$-valued functions. Note that if $\lambda^d$ is   the   Lebesgue measure and $w\colon \gdomain\to (0,\infty)$ is a strictly positive Borel-measurable function, then $L_0(\gdomain):=L_0(\gdomain,\mathcal{B}(\gdomain), \lambda^d;\C)=L_0(\gdomain,\mathcal{B}(\gdomain), w\lambda^d;\C)$.
We write 
\[
L_{1,\loc}(\gdomain,\mathcal{B}(\gdomain),\mu;E)
:=
\sggklam{
f\in L_0(\gdomain,\mathcal{B}(\gdomain), \mu;E)
\colon 
\int_K\norm{f\sep E}\mathrm d\mu<\infty \, \forall   K\subset \gdomain \text{ compact} 
};
\]
$L_{1,\loc}(\gdomain;E):=L_{1,\loc}(\gdomain,\mathcal{B}(\gdomain),\lambda^d;E)$ and $L_{1,\loc}(\gdomain):=L_{1,\loc}(\gdomain,\mathcal{B}(\gdomain),\lambda^d;\C)$.
Note that if $w\colon\gdomain\to (0,\infty)$ is such that 
$
0<\inf_K w\leq\sup_K w<\infty \quad   \text{for all compact } K  \subset   \gdomain, 
$
then $L_{1,\loc}(\gdomain)=L_{1,\loc}(\gdomain,\mathcal{B}(\gdomain),w\lambda^d;\C)$.  
 
Recall that for $\sigma=s+m>0$ with $m\in\N_0$ and $0<s\leq 1$ the H\"older-Zygmund norm of a function $g$ on $G\subset\R^d$ is given by
$$
	\norm{g \sep \mathcal{C}^\sigma(G) } := \norm{g \sep C^m(G) } + \sum_{\abs{\alpha}=m} \big[ \partial^\alpha g \big]_{\mathcal{C}^s(G)}, 
$$
where $\norm{g \sep C^m(G) }:=\sum_{\abs{\alpha}\leq m} \norm{\partial^\alpha g \sep C(G)}$ with $\norm{f \sep C(G)}:=\sup_{x\in G} \abs{f(x)}$ and
$$
   \big[ f \big]_{\mathcal{C}^s(G)} 
    := \begin{cases}
        \displaystyle \sup_{x,y\,\in\,G} \frac{\abs{f(x)-f(y)}}{\abs{x-y}^s} & \quad \text{if} \quad 0<s<1,\\
        \displaystyle \sup_{\substack{x,y\,\in\,G:\\(x+y)/2\,\in\,G}} \frac{\abs{f(x)-2\, f((x+y)/2) + f(y)}}{\abs{x-y}^s} & \quad \text{for} \quad s=1.
    \end{cases}
$$
Moreover,  
$\mathcal{C}^\sigma(\gdomain):=\ggklam{g\colon \norm{g\sep \mathcal{C}^s(\gdomain)}<\infty}$
and
\[
\mathcal{C}^{\sigma}_{\mathrm{loc}}(\gdomain):=\left\{f\colon G\to\C \,\colon f|_{K}\in  \mathcal{C}^{\sigma}(K)\text{ for all } K\subset \gdomain \text{ compact} \right\}.
\]


\section{Preliminaries}
\label{sec:Preliminaries}

In this preliminary section we collect some notation as well as some definitions and facts that we need in order to define and analyse the spaces $H^\gamma_{p,\Theta,\theta}(\D)$ in Section~\ref{sec:WSobAngle}.
In particular, we take a closer look at some basic properties of  the angular domains $\D=\D_\kappa$, (approximate) smooth resolutions of unity, and the weighted Sobolev spaces $H^\gamma_{p,\Theta}(\domain)$ on arbitrary domains $\domain\subsetneq\R^d$.


\subsection{The angular domains \texorpdfstring{$\D=\D_\kappa$}{D}}

Let $0<\kappa<2\pi$ and let $\D=\D_\kappa$ be the angular domain (or \emph{sector} or \emph{two-dimensional cone}) from~\eqref{eq:D}. 
The boundary of $\D$ is smooth everywhere, except at one point: the vertex at $x=0$.
We write
\[
\rho_\circ(x):=\dist{x}{\{0\}}, \qquad x\in\D,
\]
for the distance of a point $x\in\D$ to this vertex.
Moreover, we write $\Phi$ for the transformation 
\[
\Phi \colon (0,\infty)\times [0,2\pi)\to\R^2\setminus\{0\}, \qquad (r,\phi)\mapsto \Phi(r,\phi):=\big( r\,\cos(\phi), r\,\sin(\phi)\big), 
\]
of polar coordinates into Cartesian coordinates.
We let 
$\I:=\I_{\kappa}:=(0,\kappa)$ 
and $\tilde{\D}:=\tilde{\D}_{\kappa}:=(0,\infty)\times \I_{\kappa}$, so that $\D=\Phi(\tilde{\D})$. 
The following simple relationship between $\rho_\circ$ and the distance $\rho_\D$ to the boundary $\partial\D$ turns out to be very useful in the course of this manuscript. 

\begin{lemma}\label{lem:dist-polar}
For $x=\Phi(r,\phi)\in\D$ with $(r,\phi)\in\tilde{\D}$ we have $\rho_\circ(x)=\abs{x} = r$ and $\rho_\D(x)= r \, \sin (\mu(\phi) )$ with
\[
\mu(\phi):=\min\!\left\{ \frac{\pi}{2},\, \phi,\, \kappa-\phi\right\}, \quad\phi\in \I.
\]
 Moreover, 
\[
\rho_{\I}
\sim
\mu
\sim
\sin(\mu)
\sim
\psi_\I
\quad\text{on }\I,	
\]
where 
\[
\psi_\I(\phi)
:=
\sin\sgrklam{\frac{\pi}{\kappa}\phi}, \quad \phi\in\I.
\]

\noindent\begin{minipage}[c]{0.5\textwidth}
\begin{proof}
The equality $\rho_\D(x)=r\sin(\mu(\phi))$ follows simply from the definition of the sine function.
Moreover, $\mu\sim \rho_{\I}$ on $\I$ since for all $\phi\in\I$ it holds that $\mu(\phi) \in(0,\pi/2]$ and
	\begin{align*}
	\mu(\phi) 
    &\leq \min\{\phi, \kappa-\phi\} = \rho_{\I}(\phi)\\
    &= \min\{\pi, \phi, \kappa-\phi\} \\
    &\leq \min\{\pi, 2\phi, 2(\kappa-\phi)\} = 2\, \mu(\phi).    
	\end{align*}
	
The basic inequality
	$$
		\frac{2}{\pi} \alpha \leq \sin(\alpha) \leq \alpha,\qquad \alpha\in [0,\pi/2],
	$$
yields $\sin(\mu)\sim\mu$ on $\I$. Using symmetry, this also gives $\rho_{\I}\sim \psi_\I$ on $\I$.
\end{proof}
	\end{minipage}\hfill
 \raisebox{3ex}{
\begin{minipage}[c]{0.48\textwidth}
\begin{figure}[H]
\includegraphics[width=7cm]{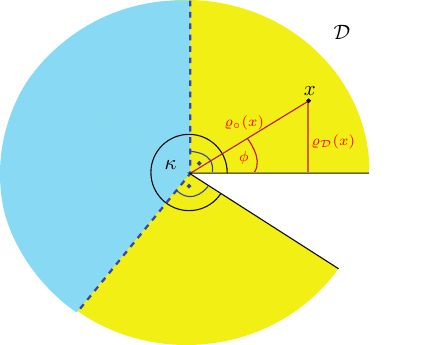}
\caption{Angular domain $\D=\D_{\kappa}$}
\end{figure}

\end{minipage}}
\end{lemma}


\subsection{Smooth resolutions of unity}\label{sec:Resolution}
In this manuscript we deal with different types of weighted Sobolev spaces on various types of domains. For the definitions of these spaces and  the proofs of some of their fundamental properties we use smooth resolutions of unity $\xi=(\xi_{\nu})_{\nu\in\Z}$ on domains $\domain\subset\R^d$ ($d\in\N$), which, for a prescribed closed and non-empty set $M\subset\partial\domain$, some $c>1$ and some $k_0\in\N$, satisfy the following conditions:
\medskip
\begin{itemize}
\item[]{\bfseries{[S$_{k_0}^{c}$]}} For all $\nu\in\Z$ it holds that  $\xi_\nu \in  C^\infty(\overline{\domain}\setminus M)$ and 
\[
    \supp(\xi_\nu) \subset \domain^{[\nu]}_{c,k_0}(M)
    := \ggklam{x\in\overline{\domain}\setminus M  :\  c^{\nu-k_0}<\mathrm{dist}(x,M)<c^{\nu+k_0}}.
\]

\item[]{\bfseries{[D$^c$]}} For all $\alpha\in\N_0^d$ it holds that  
\[
\abs{\partial^\alpha \xi_\nu(x)}\lesssim_\alpha  c^{-\abs{\alpha}\nu},
\quad \nu\in\Z,\quad x\in\R^d.
\]

\item[]{\bfseries{[R]}} For all $x\in\domain$ it holds that  $\xi_\nu(x)\geq 0$ for all $\nu\in\Z$ and $\sum_{\nu\in\Z} \xi_\nu(x) = 1$.
\end{itemize} 
\medskip

\noindent Often we 
shall also 
use only \emph{approximate} resolutions of unity, i.e., instead of {\bfseries{[R]}} we merely assume:
\medskip
\begin{itemize}
    \item[]{\bfseries{[}L{]}} There exists a positive number $\delta>0$ such that $\sum_{\nu\in\Z} \xi_\nu(x) \geq \delta$ for all $x\in\domain$.
\end{itemize}
\medskip
\noindent And, in certain situations, it even suffices to assume that $\xi$ merely satisfies {\bfseries{[S$_{k_0}^{c}$]}} and {\bfseries{[D$^c$]}}.

The domain $\domain$ and the set $M$ vary in the course of the manuscript. In particular, we choose $M:=\partial\domain$ for the definition of $H^\gamma_{p,\Theta}(\domain)$ on arbitrary domains $\domain\subset\R^d$ with non-empty boundary $\partial\domain$, whereas  $M:=\{0\}$ in the definition of the spaces $H^\gamma_{p,\Theta,\theta}(\D)$ on angular domains $\D\subset\R^2$. 
In order to avoid any confusion when switching from one to the other setting, we introduce the following sets.
\begin{defi}
\label{def:resolutions:sets}
Let $\domain\subsetneq\R^d$ be a domain and let $\emptyset\neq M\subset\partial\domain$ be closed. Let $c>1$ and $k_0\in\N$.
Let $i\in \{\textup{L},\textup{R}\}$. We write
\[
    \mathscr{A}_{c,k_0}(\domain,M)
    :=\ggklam{\xi=(\xi_\nu)_{\nu\in\Z} :\ \xi\text{ satisfies } \textup{{\bfseries{[S$_{k_0}^{c}$]}}}, {\textup{\bfseries{[D$^c$]}}} \text{ w.r.t. } \domain \text{ and } M}
\]
and 
\begin{align*}
    \mathscr{A}_{c,k_0}^{[i]}(\domain,M)
    := \ggklam{\xi=(\xi_\nu)_{\nu\in\Z} \in \mathscr{A}_{c,k_0}(\domain,M)  :\  \xi\text{ satisfies } \textup{\bfseries{[$i$]}}}.
\end{align*}
Moreover, 
\[
    \mathscr{A}_{c}(\domain,M)
    := \bigcup_{k_0\in\N} \mathscr{A}_{c,k_0}(\domain,M)
    \quad\text{and}\quad
    \mathscr{A}_{c}^{[i]}(\domain,M)
    := \bigcup_{k_0\in\N} \mathscr{A}_{c,k_0}^{[i]}(\domain,M).
\]
If clear from the context, we omit $\domain$ and $M$ from the notation.
\end{defi}

The following observations will be frequently used in the proofs below. They are verified by straightforward calculations and can be skipped at first reading.

\begin{remark}\label{rem:resolution}
Let $\domain \subsetneq \R^d$ be a domain and let $\emptyset\neq M\subset\partial\domain$ be closed. Moreover, let $c,c_1>1$ and $k_0,k_1\in\N$.

\begin{enumerate}[label=\textup{(\roman*)}]

\item\label{it:resolution:cover}  
\emph{Finite overlapping of level sets. } 
Obviously, 
\[
    \domain 
    \subset  \bigcup_{\nu\in\Z}\domain_{c,k_0}^{[\nu]}(M).
\]
Moreover, there exists $N=N(c,c_1,k_0,k_1)\in\N$ such that, for all $\nu\in\Z$,  
\begin{align*}
    A_\nu
    := A_\nu(c,c_1,k_0,k_1)
    :&= \ggklam{\mu\in\Z :\  \domain_{c_1,k_1}^{[\mu]}(M)\cap \domain_{c,k_0}^{[\nu]}(M)\neq\emptyset}\\
    &= \ssggklam{\mu\in\Z : \  \left\lvert\mu - \frac{\nu}{\log_c(c_1)}\right\rvert < \frac{k_0}{\log_{c}(c_1)}+k_1}\\
    &\subset \ggklam{\alpha(\nu)\pm j :\  j\in \{0,1,\ldots,N\}},    
\end{align*}
where $\alpha(\nu):=\alpha(\nu,c,c_1):=\lfloor \nu/\log_c(c_1)\rfloor$.
In particular,  the cardinality of $A_\nu$ is bounded uniformly in $\nu$. 
Moreover, if  $\xi=(\xi_\nu)_{\nu\in\Z}$ satisfies~{\bfseries{[S$_{k_1}^{c_1}$]}}, then, for all $\nu\in\Z$, 
\[
    \sum_{\mu\in\Z}\xi_{\mu} (x)
    = \sum_{j=-N}^N \xi_{\alpha(\nu)+j}(x),\qquad x\in \domain_{c,k_0}^{[\nu]}(M).
\]
If $c=c_1$, then $\alpha(\nu)=\nu$ and any $N\geq k_0+k_1-1$ is an admissible choice.

\item\label{it:one} Assume $c=c_1$, $k_0=k_1$, and let  $\xi,\zeta\in \mathscr{A}_{c,k_0}^{[\textup{L}]}(\domain,M)$.
Then, by part~\ref{it:resolution:cover}, each
\[
    \eta_\nu
    := \frac{\xi_\nu}{\sum_{j=-2k_0+1}^{2k_0-1}\zeta_{\nu+j}},
\qquad\nu\in\Z,
\]
is well-defined (using the convention ``$\frac{0}{0}=0$'') and we have $\eta=(\eta_\nu)_{\nu\in\Z}\in\mathscr{A}_{c,k_0}^{[\textup{L}]}(\domain,M)$.
The property~{\bfseries{[D$^c$]}} may be verified by using Leibniz's rule.
If, in addition, $\xi=\zeta$, then $\eta\in \mathscr{A}_{c,k_0}^{[\textup{R}]}(\domain,M)$.

\item\label{it:resolution:construction} 
\emph{Construction of $\xi\in\mathscr{A}^{[\textup{L}]}_c(\domain,M)$. } Since $M$ is assumed to be closed, we may construct a regularized distance $\psi$ to $M$ on $\domain$, i.e.,  
an element of the set
\[
    \Regd(\domain,M)
    \!:=\!\ggklam{\psi\in C^\infty(\overline{\domain}\setminus M)  :\  \psi\sim\dist{\cdot}{M},\,  \abs{\partial^\alpha\psi}\lesssim_\alpha{ \big(\dist{\cdot}{M}\big) }^{1-\abs{\alpha}},\,   \alpha\in\N_0^d },
\]
by following the lines of~\cite[Chapter~VI, Section~2.1]{Ste70}.
Then we can use $\psi$ to construct a sequence $\xi=(\xi_\nu)_{\nu\in\Z}\in\mathscr{A}^{[\textup{L}]}_{c}(\domain,M)$ the following way:
Choose an arbitrary non-negative $\eta\in C_0^\infty((0,\infty))$ such that $\eta\equiv 1$ on $[c^{-1},c]$ and set $\xi_\nu(x):=\eta(c^{-\nu}\psi(x))$ for all $x\in \overline{\domain}\setminus M$ and all $\nu\in\Z$.
Then, since $\psi\sim\dist{\cdot}{M}$, there exists $k_1\in\N$ such that $\supp(\xi_\nu)\subset \domain_{c,k_1}^{[\nu]}(M)$ for all $\nu\in\Z$ and, due to the properties of $\psi$ and $\eta$, we even have $\xi\in\mathscr{A}^{[\textup{L}]}_{c,k_1}(\domain,M)$.
In particular, $\mathscr{A}^{[\textup{L}]}_{c}(\domain,M)\neq\emptyset$ and also $\mathscr{A}^{[\textup{R}]}_{c}(\domain,M)\neq\emptyset$, see~\ref{it:one}.

\item\label{it:resolution:constr:D} 
\textit{Construction of $\zeta\in\mathscr{A}^{[\textup{R}]}_{c,1}(\D,\{0\})$. } 
If $\domain=\D$ and $M=\{0\}$, then $\rho_\circ=\dist{\cdot}{\{0\}}=\abs{\cdot}$ belongs to $\Regd(\D,\{0\})$.
Moreover, in this particular case we can construct a resolution of unity $\zeta\in\mathscr{A}^{[\textup{R}]}_{c,1}(\D,\{0\})$ with the additional property that
\begin{equation}\label{eq:dilation}
\zeta_\nu=\zeta_0(c^{-\nu}\cdot),\qquad \nu\in\Z,    
\end{equation}
in the following way: Choose $1<a<b<c$ and set $\zeta_0(x):=\eta(\abs{x})$, $x\in\R^2$, where $\eta:=\varphi - \varphi(c\,\cdot)$ is based on some $\varphi \in  C_0^\infty(\R)$ with
\[
\varphi(r) = \begin{cases}
1 & \quad \text{if } \abs{r} \leq a,\\
0 & \quad \text{if } \abs{r} \geq b,
\end{cases}
\]
and $\varphi(r) \in [0,1]$ otherwise. 
Then $\zeta:=(\zeta_\nu)_{\nu\in\Z}:=(\zeta_0(c^{-\nu}\cdot))_{\nu\in\Z}\in\mathscr{A}^{[\textup{R}]}_{c,1}(\D,\{0\})$.
By setting
\[
    \widetilde{\zeta}_\nu
    :=\sum_{j=-k_0}^{k_0} \zeta_{\nu+j},\qquad\nu\in\Z,
\]
we obtain a new sequence $\widetilde{\zeta}:=(\widetilde{\zeta}_\nu)_{\nu\in\Z}\in \mathscr{A}^{[\textup{L}]}_{c,k_0+1}(\D,\{0\})$
such that
\[
\widetilde{\zeta}_\nu \equiv 1\text{ on }\D_{c,k_0}^{[\nu]}(\{0\}),
\quad
\text{and}
\quad
\widetilde{\zeta}_\nu=\widetilde{\zeta}_0(c^{-\nu}\cdot),
\qquad
\nu\in\Z.
\]
\item\label{it:resolution:r-c}
\emph{Retraction-coretraction pairs. } 
Let $\zeta\in\mathscr{A}^{\textup{[R]}}_{c, 1}(\D,\{0\})$ for some $c>1$ and let $\eta_\nu:=\sum_{j=-1}^1\zeta_{\nu+j}$, $\nu\in\Z$.
In our proofs in \autoref{sec:WSobAngle} we will frequently use the linear mappings
\begin{align*}
    S_c\colon \mathscr{D}'(\D) &\to \mathscr{D}'(\D)^{\Z} && \text{ and }& R_c\colon \mathscr{D}'(\D)^{\Z}&\to \mathscr{D}'(\D) \\
    u &\mapsto \grklam{(\zeta_\nu u)(c^\nu\cdot)}_{\nu\in\Z} && & (f_\nu)_{\nu\in\Z} &\mapsto \sum_{\nu\in\Z}\eta_\nu f_\nu(c^{-\nu}\cdot).
\end{align*}
We also write $S_c$ and $R_c$ for 
their restrictions
to suitable subspaces of the space $\mathscr{D}'(\D)$ of generalized functions on $\D$ and $\mathscr{D}'(\D)^\Z$, respectively.
Obviously, $S_c$ is well-defined since $\zeta_\nu\in C^\infty(\D)$ for all $\nu\in\Z$.
Also $R_c$ is well-defined, as the series $\sum_{\nu\in\Z}\eta_\nu f_\nu(c^{-\nu}\cdot)$ is locally a finite sum of generalized functions on $\D$.
Moreover, note that $R_c$ is a left inverse of $S_c$, since $\eta_\nu\equiv 1$ on $\supp(\zeta_\nu)$, so that for all $u\in\mathscr{D}'(\D)$ there holds
\[
R_cS_cu
=
R_c\grklam{\grklam{(\zeta_\nu u)(c^\nu\cdot)}_{\nu\in\Z}}
=
\sum_{\nu\in\Z} \eta_{\nu}\zeta_\nu u
=
\sum_{\nu\in\Z} \zeta_\nu u
=
u.
\]
In particular, if we can prove that $R_c\in\L(X,Y)$ and $S_c\in\L(Y,X)$ for some (reflexive) Banach space
$X\subset\mathscr{D}'(\D)^\Z$ and some normed space $Y\subset\mathscr{D}'(\D)$, then $Y$ is a (reflexive) Banach space, too,  and $(R_c,S_c)$ is a so-called retraction-coretraction pair for $(X,Y)$.
For details we refer, e.g., to \cite[Section~1.2.4]{Tri1995} and the proof given therein.
\end{enumerate}
\end{remark}

\subsection{The spaces \texorpdfstring{$H^\gamma_{p,\Theta}(\domain)$}{Hgamma,p,Theta(O)}}
\label{sec:WSob:Lototsky}

For the definition and analysis of the spaces $H^\gamma_{p,\Theta,\theta}(\D)$ in \autoref{sec:WSobAngle} we need some knowledge of the weighted Sobolev spaces $H^\gamma_{p,\Theta}(\domain)$ defined in~\cite{Lot2000}, see also~\cite[Section~3.2.3]{Tri1995}.
For the convenience of the reader,  we provide in this section the definition and those properties of $H^\gamma_{p,\Theta}(\domain)$ that are relevant for our analysis.
Note that, at least for special choices of parameters, these spaces had been discussed before in the context of degenerate elliptic PDEs, see e.g.~\cite{Kuf1980,LioMag1968,Tri1995} and the references therein.

\begin{defi}\label{def:Lot_spaces}
Let $\domain\subsetneq\R^d$ be an arbitrary domain with boundary $\partial\domain$. Let $1<p<\infty$ as well as $\gamma,\Theta\in\R$. Moreover, let $\xi=(\xi_\nu)_{\nu\in\Z}\in \mathscr{A}^{[\textup{L}]}_{c}(\domain,\partial\domain)$ for some $c>1$. Then we set
\[
    H^\gamma_{p,\Theta}(\0)
    := H^\gamma_{p,\Theta}(\0)_{c,\xi}
    := \left\{ u \in \mathscr{D}'(\domain)  :\ 
    \norm{u \sep H^\gamma_{p,\Theta}(\domain)}_{c,\xi} < \infty \right\},
\]	
where
\begin{align}\label{eq:Lot_norm}
    \norm{u \sep H^\gamma_{p,\Theta}(\domain)}
    := \norm{u \sep H^\gamma_{p,\Theta}(\domain)}_{c,\xi} 
    := \left( \sum_{\nu\in\Z} c^{\nu \Theta} \norm{(\xi_\nu\, u)(c^\nu \cdot) \sep H^{\gamma}_{p}(\R^d)}^p \right)^{1/p}.
\end{align}
Moreover, we set $L_{p,\Theta}(\domain):=H^0_{p,\Theta}(\domain)$.
\end{defi}

\begin{remark} Concerning the  definition and basic properties of the Bessel potential spaces $H^{\gamma}_{p}(\R^d)$, the reader is referred to \autoref{sec:Bessel}. 
At first sight, $H^\gamma_{p,\Theta}(\domain)_{c,\xi}$ and $\norm{\cdot\sep H^\gamma_{p,\Theta}(\domain)}_{c,\xi}$ seem to depend on the particular choice of $\xi$ and $c$. However, as shown in~\cite[Section~2]{Lot2000}, in the setting of \autoref{def:Lot_spaces}, it holds that
\[
    \sum_{\nu\in\Z} c_1^{\nu \Theta} \norm{({\zeta}_\nu\, u)(c_1^\nu \cdot) \sep H^{\gamma}_{p}(\R^d)}^p
    \lesssim \norm{u \sep H^\gamma_{p,\Theta}(\domain)}_{c,\xi}^p,
    \qquad u \in H^\gamma_{p,\Theta}(\domain)_{c,\xi},
\]
for any  $\zeta=(\zeta_\nu)_{\nu\in\Z}\in\mathscr{A}_{c_1}(\domain,\partial\domain)$ and $c_1>1$.
In particular, if $\zeta\in\mathscr{A}^{[\textup{L}]}_{c_1}(\domain,\partial\domain)$,  then $H^\gamma_{p,\Theta}(\domain)_{c,\xi}=H^\gamma_{p,\Theta}(\domain)_{c_1,\zeta}$ with equivalent norms. 
Therefore, we will omit these indices in the sequel.
Also, we will not mention the dependence of the constants on $\xi$ and $c$.
\end{remark}

In what follows we will use the following properties of the spaces $H^\gamma_{p,\Theta}(\domain)$. 
Throughout, $\psi_\domain\in \Regd(\domain,\partial\domain)$ denotes a regularised distance to the boundary, see \autoref{rem:resolution}\ref{it:resolution:construction} above.

\begin{lemma}\label{lem:LotSpaces:properties}
Let $\domain\subsetneq\R^d$ be a domain,
$1<p<\infty$, and $\gamma,\Theta\in\R$.
\begin{enumerate}[label=\textup{(\roman*)}]
\item\label{it:LotSpaces:Banach} $\grklam{H^\gamma_{p,\Theta}(\domain),\norm{\cdot\sep H^\gamma_{p,\Theta}(\domain)}}$ is a reflexive Banach space.
\item\label{it:LotSpaces:density} $ C_0^\infty(\domain)$ is dense in $H^\gamma_{p,\Theta}(\domain)$.
\item\label{it:LotSpaces:N} If $\gamma\in\N_0$, then 	
\[	
    H^\gamma_{p,\Theta}(\domain) = \left\{ u\in \mathscr{D}'(\domain)  :\   \normmod{u \sep H^\gamma_{p,\Theta}(\domain)}<\infty \right\},
\]	
where
\[
\normmod{u \sep H^\gamma_{p,\Theta}(\domain)}
:= 
\bigg(\sum_{\alpha\in\N_0^d\colon\abs{\alpha}\leq \gamma}
\int_\domain \Abs{\rho_\domain(x)^{\abs{\alpha
}} D^\alpha u(x)}^p \rho_\domain(x)^{\Theta-d}\,\dx
 \bigg)^{1/p}
\]
is an equivalent norm. 
\item\label{it:LotSpaces:interpolation} Let $\gamma_0,\gamma_1,\Theta_0,\Theta_1\in\R$, and $1<p_0,p_1<\infty$. Then for all $0<\vartheta<1$, as well as
\begin{align}\label{eq:iterpol_parameter}
\frac{1}{p} := \frac{1-\vartheta}{p_0} + \frac{\vartheta}{p_1}, 
\quad 
\gamma:=(1-\vartheta) \gamma_0+ \vartheta \gamma_1,
\quad \text{and}\quad
\Theta:=(1-\vartheta) \Theta_0+ \vartheta \Theta_1,
\end{align}
we have
\begin{align}\label{eq:complex_interpol}
\big[ H^{\gamma_0}_{p_0,\Theta_0 p_0}(\domain), H^{\gamma_1}_{p_1,\Theta_1 p_1}(\domain) \big]_\vartheta = H^{\gamma}_{p,\Theta p}(\domain)
\end{align}
isomorphically. 
\item\label{it:LotSpaces:duality} Assume that $1<p,p'<\infty$ and $\gamma,\Theta,\Theta'\in\R$ are such that 
$\displaystyle \tfrac{1}{p}+\tfrac{1}{p'}=1$ and 
$\displaystyle\tfrac{\Theta}{p}+\tfrac{\Theta'}{p'}=d$.
Then 
\[
(\varphi,\psi)
:=
\int_\domain \varphi(x)\,\psi(x)\d x,\qquad \varphi,\psi\in  C_0^\infty(\domain),
\]
can be uniquely extended to a continuous bilinear form on $H^\gamma_{p,\Theta}(\domain)\times H^{-\gamma}_{p',\Theta'}(\domain)$ which provides the isomorphism 
$$
\big(H^\gamma_{p,\Theta}(\domain)\big)' = H^{-\gamma}_{p',\Theta'}(\domain).
$$	
\item\label{it:LotSpaces:multiplier} Let $n\in\N_0$ and let $a\colon\domain\to\R$ satisfy $\abs{a}^{(0)}_n:=\sup_{x\in\domain}\sum_{\abs{\alpha}\leq n}\rho_\domain^{\abs{\alpha}}(x) \abs{D^\alpha a(x)}<\infty$. Then, if $\abs{\gamma}\leq n$,
\[
    \norm{a u \sep H^\gamma_{p,\Theta}(\domain)}
    \leq C(d,p,n) \abs{a}^{(0)}_n \norm{ u \sep H^\gamma_{p,\Theta}(\domain)}.
\]
\item\label{it:LotSpaces:bddDom} If $\domain$ is bounded, then $H^\gamma_{p,\Theta}(\domain)\hookrightarrow H^{\gamma}_{p,\Theta_1}(\domain)$ for all $\Theta_1\geq\Theta$.
\item\label{it:LotSpaces:embeddings} 
    Let $1<p_0 \leq p_1<\infty$, as well as $\gamma_0,\gamma_1,\Theta\in\R$ with $\gamma_0\geq \gamma_1$ such that
    \begin{align}\label{eq:cond:diff-dim}
		\gamma_0 - \frac{d}{p_0} \geq  \gamma_1 - \frac{d}{p_1}.
	\end{align}
	Then $H^{\gamma_0}_{p_0,\Theta p_0}(\domain) \hookrightarrow  H^{\gamma_1}_{p_1,\Theta p_1}(\domain)$.
\item \label{it:LotSpaces:indexshift}
	Let $\nu\in\R$. Then
		$\psi_\domain^\nu u\in H^\gamma_{p,\Theta p}(\domain)$ if, and only if, $u\in H^\gamma_{p,(\Theta+\nu) p}(\domain)$. In this case,
			\begin{align*}
				\norm{\psi_\domain^\nu u \sep H^\gamma_{p,\Theta p}(\domain)} 
				\sim \norm{u \sep H^\gamma_{p,(\Theta+\nu) p}(\domain)}.
			\end{align*}
\item \label{it:LotSpaces:lifting}
			$H^{\gamma+1}_{p,\Theta}(\domain) 
			\!=\! \left\{u \in H^{\gamma}_{p,\Theta}(\domain)  :   \psi_\domain D u \in H^{\gamma}_{p,\Theta}(\domain) \right\}
		    \!=\! \left\{u \in H^{\gamma}_{p,\Theta}(\domain)  :  D(\psi_\domain u) \in H^{\gamma}_{p,\Theta}(\domain) \right\}$ with
			\begin{align*}
				\norm{u \sep H^{\gamma+1}_{p,\Theta}(\domain)} 
				&\sim \norm{u \sep H^\gamma_{p,\Theta}(\domain)} + \norm{\psi_\domain D u \sep H^\gamma_{p,\Theta}(\domain)}\\
				&\sim \norm{u \sep H^\gamma_{p,\Theta}(\domain)} + \norm{ D(\psi_\domain u) \sep H^\gamma_{p,\Theta}(\domain)}.
			\end{align*}
\end{enumerate}
\end{lemma}

\begin{proof}
The statements~\ref{it:LotSpaces:Banach}, \ref{it:LotSpaces:density}, \ref{it:LotSpaces:N}, \ref{it:LotSpaces:duality}, \ref{it:LotSpaces:bddDom}, \ref{it:LotSpaces:indexshift}, and \ref{it:LotSpaces:lifting} have all been proven in~\cite{Lot2000}. 
Also, the statement~\ref{it:LotSpaces:interpolation} on the complex interpolation of weighted Sobolev spaces has been stated and proven in~\cite{Lot2000}, however, with a mistake in the statement which is corrected here (see also the proof of~\autoref{thm:interpolation} below). A proof of~\ref{it:LotSpaces:multiplier} may be found in~\cite[Lemma~3.1]{Kim2006}. Finally, part~\ref{it:LotSpaces:embeddings} is a simple consequence of the classical Sobolev embedding theorem.
\end{proof}

Next we provide a localization result for the spaces $H^\gamma_{p,\Theta}(\domain)$, which can be found in \cite[Theorem~3.4]{Lot2000}. 

\begin{prop}
	\label{prop:Lot_Local}
Let $\domain\subsetneq\R^d$ be a domain, $1<p<\infty$, and $\gamma,\Theta\in\R$.
	Further let $\eta=(\eta_k)_{k\in\N}$ denote a collection of $C^\infty(\domain)$-functions such that
	$$
	\sup_{x\in\domain} \sum_{k\in\N} \rho_{\domain}(x)^{\abs{\alpha}} \abs{D^\alpha \eta_k(x)} \leq C_\alpha, \qquad \alpha\in\N_0^d.
	$$
	Then
	$$
	\sum_{k\in\N} \norm{\eta_k u \sep H^\gamma_{p,\Theta}(\domain)}^p 
	\lesssim \norm{u \sep H^\gamma_{p,\Theta}(\domain)}^p, \qquad u\in H^\gamma_{p,\Theta}(\domain).
	$$
	If, in addition,
	$$
	\inf_{x\in\domain} \sum_{k\in\N} \abs{\eta_k(x)}^p \geq \delta > 0,
	$$
	then 
	$$
	\norm{u \sep H^\gamma_{p,\Theta}(\domain)}^p
	\lesssim \sum_{k\in\N} \norm{\eta_k u \sep H^\gamma_{p,\Theta}(\domain)}^p, \qquad u\in H^\gamma_{p,\Theta}(\domain).
	$$
\end{prop}

\begin{remark}
    Let us briefly mention that assertions in the spirit of \autoref{prop:Lot_Local} play an important role in the theory of so-called \emph{refined localization spaces} $F^{\gamma,\mathrm{rloc}}_{p,q}(\domain)$; see, e.g., \cite[Section~2.2.3]{T08} and the references therein.
    Indeed, for $1<p<\infty$ and $\gamma\in\R$ it is easy to show that $F^{\gamma,\mathrm{rloc}}_{p,2}(\domain)=H^\gamma_{p,d-\gamma p}(\domain)$ under very mild assumptions on $\domain$ which particularly cover bounded Lipschitz domains. However, we will not follow this line of research here but refer to \cite{HSS24} in this context.
\end{remark}

The following result  describes the growth/decay of functions in $H^\gamma_{p,\Theta}(\domain)$ (and some of its derivatives) near and far away from the boundary~$\partial\domain$.  We refer to~\autoref{sec:Notation} for the definition of the Hölder-Zygmund spaces $\cont^s(\domain)$ and associated (semi-)norms.
Recall that $\psi_\domain\in \Regd(\domain,\partial\domain)$ denotes a regularised distance to the boundary, see \autoref{rem:resolution}\ref{it:resolution:construction} above.

\begin{prop}\label{prop:Lot_growth}
    Let $\domain\subsetneq\R^d$ be a domain,  $1<p<\infty$, and $\gamma,\Theta\in\R$ with $\gamma>d/p$ such that $\gamma-d/p=s+m$ with  $m:=\ceil{\gamma-d/p}-1\in\N_0$ and $0<s\leq 1$. 
    Then $u \in H^\gamma_{p,\Theta p}(\domain)$ admits continuous partial derivatives $D^\alpha u$ up to order $m$, and for each $\alpha\in\N_0^d$ with $\abs{\alpha}\leq m$ we have
    \begin{align}\label{eq:C-norm_bound}
        \norm{ \psi_{\domain}^{\Theta+\abs{\alpha}}\, D^\alpha u \sep C(\domain) }
        \lesssim \norm{u \sep H^\gamma_{p,\Theta p}(\domain)}
        \quad\text{and}\quad \left[ \psi_{\domain}^{\Theta+s+\abs{\alpha}}\, D^\alpha u\right]_{\mathcal{C}^s(\domain)} \lesssim \norm{u \sep H^\gamma_{p,\Theta p}(\domain)};
    \end{align}
    in particular, $\psi_{\domain}^{\Theta+s+\abs{\alpha}}\, D^\alpha u \in \mathcal{C}^{s}_{\mathrm{loc}}(\domain)$.
    If, in addition, there holds $\supp(u)\subset S\subset \overline{\domain}$ and $\rho_{\domain}\leq R<\infty$ on $S$, then 
    $\psi_{\domain}^{\Theta+s+\abs{\alpha}}\, D^\alpha u \in \mathcal{C}^{s}(\domain)$ with
    $$
   		 \norm{ \psi_{\domain}^{\Theta+s+\abs{\alpha}}\, D^\alpha u \sep \mathcal{C}^s(\domain)} 
   		 \lesssim R_1^s \norm{u \sep H^\gamma_{p,\Theta p}(\domain)}, \qquad \abs{\alpha} \leq m,
    $$
    where $R_1:=\max\{ 1,\, R\}$.
\end{prop}

\begin{proof} The first part of \autoref{prop:Lot_growth} is taken from \cite[Theorem~4.3]{Lot2000}.
Note that, although not included therein, the case $s=1$ can be obtained by the same considerations.
For the second part, note that if $\supp(u)\subset S\subset \overline{\domain}$ and $\rho_{\domain}\leq R<\infty$ on $S$, then~\eqref{eq:C-norm_bound} implies
\begin{align*}
	\norm{ \psi_\domain^{\Theta+s+\abs{\alpha}}\, D^\alpha u \sep \mathcal{C}^s(\domain)}
	&=\sup_{x\in\domain} \abs{\psi_\domain(x)^s\, \psi_\domain(x)^{\Theta+\abs{\alpha}}\, (D^\alpha u)(x)} 
	+ \Big[ \psi_\domain^{\Theta+s+\abs{\alpha}}\, D^\alpha u \Big]_{\mathcal{C}^s(\domain)} \\
	&\lesssim R^s \norm{ \psi_\domain^{\Theta+\abs{\alpha}}\, D^\alpha u \sep C(\domain)} + \Big[ \psi_\domain^{\Theta+s+\abs{\alpha}}\, D^\alpha u \Big]_{\mathcal{C}^s(\domain)} \\
	&\lesssim R_1^s \norm{u \sep H^\gamma_{p,\Theta p}(\domain)}
\end{align*}
for $\alpha\in\N_0^d$ with $\abs{\alpha}\leq m$, since $\psi_\domain\sim \rho_\domain \leq R \leq \max\{ 1,\, R\} =: R_1$ on $S\supset \supp(D^\alpha u)$.
\end{proof}

For $\domain=\D$ we also have the following assertions, which we will frequently use below.

\begin{lemma}\label{lem:Lot_pw-mult-special}
Let  $1<p<\infty$ as well as $\gamma,\Theta\in\R$.
\begin{enumerate}[label=\textup{(\roman*)}]
\item\label{it:LotSpaces:D:dilation} $\displaystyle \norm{u(s\,\cdot) \sep H^\gamma_{p,\Theta}(\D)} 
		\sim s^{-\Theta/p} \norm{u \sep H^\gamma_{p,\Theta}(\D)}$,\quad $s>0$,\quad $u\in H^\gamma_{p,\Theta}(\D)$.
\item\label{it:LotSpaces:D:localization} Assume that $\zeta:=(\zeta_\nu)_{\nu\in\Z}\in \mathscr{A}_c(\D,\{0\})$ for some $c>1$.  
Then
\begin{equation}\label{eq:Lot_pw-mult-special:1}
    \norm{\zeta_\nu u \sep H^\gamma_{p,\Theta}(\D)} 
    \lesssim \norm{u \sep H^\gamma_{p,\Theta}(\D)},
    \qquad u \in H^\gamma_{p,\Theta}(\D),\quad \nu\in\Z,    
\end{equation}
and for all $\alpha\in\N_0^2$ it holds that
\begin{equation}\label{eq:Lot_pw-mult-special:2}
    \norm{\partial^\alpha\big(\zeta_\nu(c^\nu\cdot)\big) u \sep H^\gamma_{p,\Theta}(\D)} 
    \lesssim \norm{u \sep H^\gamma_{p,\Theta}(\D)},
    \qquad u \in H^\gamma_{p,\Theta}(\D),\quad \nu\in\Z.    
\end{equation}
\end{enumerate}
\end{lemma}

\begin{proof}
We first prove~(i). Due to \autoref{lem:LotSpaces:properties}\ref{it:LotSpaces:interpolation} and~\ref{it:LotSpaces:duality}, it is enough to verify the assertion for $\gamma\in\N_0$. However, in this case the claim follows from the fact that $\rho_\D(sx)=s\rho_\D(x)$, $x\in\D$, simply by using \autoref{lem:LotSpaces:properties}\ref{it:LotSpaces:N} together with the chain rule and Jacobi's transformation theorem. 

To verify~(ii), we first prove~\eqref{eq:Lot_pw-mult-special:2}.
In view of \autoref{lem:LotSpaces:properties}\ref{it:LotSpaces:multiplier} it suffices to show that for all $n\in\N_0$ 
\begin{equation}\label{eq:Lot_pw-mult-special:2:proof}
\sup_{\nu\in\Z}\abs{\partial^\alpha\big(\zeta_\nu(c^\nu\cdot)\big)}_n^{(0)}
= 
\sup_{\nu\in\Z}\sup_{x\in\D}\sum_{\abs{\beta}\leq n}\rho_\D^{\abs{\beta}}(x) \Abs{\partial^{\beta+\alpha} [\zeta_\nu(c^\nu\cdot)](x)}
<\infty.    
\end{equation}
To see this, note that $\zeta\in\mathscr{A}_c(\D,\{0\})$ implies
\[
    \Abs{\partial^{\beta+\alpha} [\zeta_\nu(c^\nu\cdot)](x)}
    =\Abs{c^{\nu\abs{\alpha+\beta}}(\partial^{\beta+\alpha} \zeta_\nu)(c^\nu x)}
    \leq C_{\alpha+\beta},
    \qquad x\in\D,
\]
and $\supp\!\big(\zeta_\nu(c^\nu\cdot)\big)\subset\D_{c,k_0}^{[0]}(\{0\})$ for some $k_0\in\N$ and all $\nu\in\Z$ while 
$\rho_\D\lesssim 1$ on $\D_{c,k_0}^{[0]}(\{0\})$.

Estimate~\eqref{eq:Lot_pw-mult-special:1} follows now from~\eqref{eq:Lot_pw-mult-special:2} with $\alpha=0$ and part~\ref{it:LotSpaces:D:dilation}, which together yield \begin{align*}
        \norm{\zeta_\nu u \sep H^\gamma_{p,\Theta}(\D)} 
        \sim c^{\nu\Theta/p} \norm{(\zeta_\nu u)(c^{\nu}\cdot) \sep H^\gamma_{p,\Theta}(\D)} 
    	\lesssim c^{\nu\Theta/p} \norm{u(c^{\nu}\cdot) \sep H^\gamma_{p,\Theta}(\D)} 
        \sim \norm{u \sep H^\gamma_{p,\Theta}(\D)}
	\end{align*}
with constants that do not depend on $\nu\in\Z$ and $u \in H^\gamma_{p,\Theta}(\D)$.
\end{proof}

\section{The weighted Sobolev spaces \texorpdfstring{$H^\gamma_{p,\Theta,\theta}(\D)$}{Hgamma,p,Theta,theta(D)}}
\label{sec:WSobAngle}
After the preparations in the previous section, we now introduce the spaces $H^\gamma_{p,\Theta,\theta}(\D)$ and prove several properties of the resulting family of spaces.
Throughout this section we fix some arbitrary $0<\kappa<2\pi$ and write $\D:=\D_{\kappa}$, see~\eqref{eq:D}.

\subsection{Definition and basic properties} 
We start with the definition.

\begin{defi}\label{def:Hspaces}
Let $\zeta=(\zeta_\nu)_{\nu\in\Z}\in \mathscr{A}_{c}^{\textup{[L]}}(\D,\{0\})$ for some $c>1$ (see~\autoref{def:resolutions:sets}). 
Let $1<p<\infty$ as well as $\gamma,\Theta,\theta\in\R$. 
Then
\begin{align*}
    H^\gamma_{p,\Theta,\theta}(\D)
    := H^\gamma_{p,\Theta,\theta}(\D)_{c,\zeta}
    := \left\{ u \in \mathscr{D}'(\D)  :\ 
    \norm{u \sep H^\gamma_{p,\Theta,\theta}(\D)}_{c,\zeta}
    < \infty \right\},    
\end{align*}
where 	
\begin{equation}
\label{eq:norm-def}
\norm{u \sep H^\gamma_{p,\Theta,\theta}(\D)}_{c,\zeta} 
:=
\left( 
\sum_{\nu\in\Z} c^{\nu \theta} \norm{(\zeta_\nu\, u)(c^\nu \cdot) \sep H^{\gamma}_{p,\Theta}(\D)}^p 
\right)^{1/p}.
\end{equation}
\end{defi}

For any choice $\zeta$, $p$, $\gamma$, $\Theta$, $\theta$, and $c$ as in \autoref{def:Hspaces}, $H^\gamma_{p,\Theta,\theta}(\D)_{c,\zeta}$ is a vector space and $\norm{\cdot\sep H^\gamma_{p,\Theta,\theta}(\D)}_{c,\zeta}$ defines a norm on it
(to check the definiteness note that $\zeta$ satisfies {\bfseries{[L]}} with $\domain=\D$ and that $\sum_{\nu\in\Z}\zeta_\nu$ is locally finite). 
However, at first sight, $\norm{\cdot\sep H^\gamma_{p,\Theta,\theta}(\D)}_{c,\zeta}$ and therefore $H^\gamma_{p,\Theta,\theta}(\D)_{c,\zeta}$ seem to depend on the concrete choice of~$\zeta$ and $c$. 
The following lemma shows that this is not the case: Replacing $\zeta$ in the definition above by any $\xi\in\mathscr{A}^{[\textup{L}]}_{c_1}(\D,\{0\})$ for some $c_1>1$
will lead to the same spaces with equivalent norms. 

\begin{lemma}\label{lem:norm-change}
Let $\zeta$, $p$, $\gamma$, $\Theta$, $\theta$, and $c$ be as in \autoref{def:Hspaces}. 
Moreover let  
$\xi:=(\xi_\nu)_{\nu\in\Z}\in\mathscr{A}_{c_1}(\D,\{0\})$ for some $c_1>1$. 
Then 
\begin{equation}\label{eq:norm:equiv}
    \sum_{\nu\in\Z} c_1^{\nu \theta} \norm{(\xi_\nu\, u)(c_1^\nu \cdot) \sep H^{\gamma}_{p,\Theta}(\D)}^p
    \lesssim \norm{u \sep H^\gamma_{p,\Theta,\theta}(\D)}_{c,\zeta}^p,
    \qquad u\in H^\gamma_{p,\Theta,\theta}(\D)_{c,\zeta}.    
\end{equation}
In particular, 
if 
$\xi\in\mathscr{A}^{[\textup{L}]}_{c_1}(\D,\{0\})$, 
then
$H^\gamma_{p,\Theta,\theta}(\D)_{c,\zeta}=H^\gamma_{p,\Theta,\theta}(\D)_{c_1,\xi}$ with equivalent norms.
\end{lemma}
\begin{proof}
We first prove~\eqref{eq:norm:equiv} under the additional assumptions that $\zeta\in\mathscr{A}^{[\textup{R}]}_{c}(\D,\{0\})$.
  Then, by  
definition, $\zeta\in\mathscr{A}^{[\textup{R}]}_{c,k_0}(\D,\{0\})$ and $\xi\in \mathscr{A}_{c_1,k_1}(\D,\{0\})$ for some $k_0,k_1\in\N$.
Let 
\[
    \widetilde{\zeta}_\nu
    := \sum_{j=-N}^N \zeta_{\alpha(\nu)+j},
    \qquad \nu\in\Z,
\]
with $\alpha(\nu)=\lfloor\nu/\log_{c_1}(c)\rfloor$ and $N=N(c_1,c,k_1,k_0)$ as in \autoref{rem:resolution}\ref{it:resolution:cover}.
Then, for all $\nu\in\Z$, $\widetilde{\zeta}_\nu\equiv 1$ on $\D^{[\nu]}_{c_1,k_1}(\{0\})\supset \supp(\xi_\nu)$.
Therefore, 
  by  
repeatedly applying \autoref{lem:Lot_pw-mult-special}, we obtain
\begin{align*}
\norm{(\xi_\nu u)(c_1^\nu\cdot) \sep H^\gamma_{p,\Theta}(\D)}
&=
\norm{\xi_\nu(c_1^\nu\cdot)\, (\widetilde{\zeta}_\nu u)(c_1^\nu\cdot) \sep H^\gamma_{p,\Theta}(\D)} \\
&\lesssim
\norm{(\widetilde{\zeta}_\nu u)(c_1^\nu\cdot) \sep H^\gamma_{p,\Theta}(\D)} \\
&\sim
\sgrklam{\frac{c_1}{c}}^{-\nu\Theta/p}\norm{(\widetilde{\zeta}_\nu u)(c^\nu\cdot) \sep H^\gamma_{p,\Theta}(\D)} \\
&\leq
\sgrklam{\frac{c_1}{c}}^{-\nu\Theta/p}\sum_{j=-N}^N\norm{(\zeta_{\alpha(\nu)+j} u)(c^\nu\cdot) \sep H^\gamma_{p,\Theta}(\D)} \\
&\sim
c_1^{-\nu\Theta/p}\sum_{j=-N}^N c^{(\alpha(\nu)+j)\Theta/p}\norm{(\zeta_{\alpha(\nu)+j} u)(c^{\alpha(\nu)+j}\cdot) \sep H^\gamma_{p,\Theta}(\D)}
\end{align*}
with constants that do not depend on $\nu\in\Z$, so that
\begin{align*}
    \sum_{\nu\in\Z} c_1^{\nu\Theta} \norm{(\xi_\nu u)(c_1^\nu\cdot) \sep H^\gamma_{p,\Theta}(\D)}^p
    &\lesssim \sum_{j=-N}^N \sum_{\nu\in\Z}  c^{-(\alpha(\nu)+j)\Theta}\norm{(\zeta_{\alpha(\nu)+j} u)(c^{\alpha(\nu)+j}\cdot) \sep H^\gamma_{p,\Theta}(\D)}^p \\
    &\lesssim \norm{u \sep H^\gamma_{p,\Theta,\theta}(\D)}_{c,\zeta}^p,
\end{align*}
where we used that $N$ does not depend on $\nu$ and that for all $m\in\Z$ the cardinality $\#  \ggklam{\nu\in\Z :\ \alpha(\nu)=m}\leq \lfloor \log_{c_1}(c)\rfloor + 1$.

In order to obtain~\eqref{eq:norm:equiv} for arbitrary $\zeta\in\mathscr{A}^{[\textup{L}]}_{c,k_0}(\D,\{0\})$ it now suffices to prove that
\begin{equation}\label{eq:norm:equiv:1}
    \norm{u \sep H^\gamma_{p,\Theta,\theta}(\D)}_{c,\zeta^*}
    \lesssim \norm{u \sep H^\gamma_{p,\Theta,\theta}(\D)}_{c,\zeta}, 
    \qquad u\in H^\gamma_{p,\Theta,\theta}(\D)_{c,\zeta},    
\end{equation}
with $\zeta^*:=(\zeta_\nu/\eta_\nu)_{\nu\in\Z}$ where $\eta_\nu:=\sum_{j=-2k_0+1}^{2k_0-1}\zeta_{\nu+j}$, $\nu\in\Z$.
As mentioned in \autoref{rem:resolution}\ref{it:one}, it holds that $\zeta^*\in \mathscr{A}^{[\textup{R}]}_{c,k_0}(\D,\{0\})$ and $\zeta^{**}:=(\zeta_\nu^*/\eta_\nu)_{\nu\in\Z}\in\mathscr{A}^{[\textup{L}]}_{c,k_0}(\D,\{0\})$.
Thus, since
\begin{align*}
    \norm{(\zeta_\nu^* u)(c^\nu \cdot ) \sep H^\gamma_{p,\Theta}(\domain) }
    = \norm{\zeta_\nu^{**}(c^\nu\cdot)\,(\eta_\nu u)(c^\nu\cdot) \sep H^\gamma_{p,\Theta}(\domain) },
    \qquad \nu\in\Z,
\end{align*}
very similar arguments as above with $\eta$ instead of $\tilde{\zeta}$ yield~\eqref{eq:norm:equiv:1}.
\end{proof}

\autoref{lem:norm-change} in mind, we make the following convention.
\begin{convention}\label{conv:zeta}
For the remainder of this text, we fix some arbitrary $\zeta=(\zeta_\nu)_{\nu\in\Z}\in\mathscr{A}^{\textup{[R]}}_{e,1}(\D,\{0\})$ as constructed in \autoref{rem:resolution}\ref{it:resolution:constr:D} (note that $\zeta$ then satisfies~\eqref{eq:dilation} with $c=e$). 
If not explicitly stated otherwise, we set $\norm{\cdot \sep H^\gamma_{p,\Theta,\theta}(\D)}:=\norm{\cdot \sep H^\gamma_{p,\Theta,\theta}(\D)}_{e,\zeta}$. 
By \autoref{lem:norm-change} all statements and proofs below hold mutatis mutandis with $e$ replaced by any $c>1$ and with any other $\zeta\in\mathscr{A}^{\textup{[L]}}_{c}(\D,\{0\})$.
\end{convention}

For $1<p<\infty$, $\theta\in\R$, and a normed space $X$, we define the vector-valued sequence space
\[
    \ell_p^\theta(\Z;X)
    := \ggklam{x=(x_\nu)_{\nu\in\Z} :\  x_\nu\in X \text{ for all }\nu\in\Z \text{ and } \norm{x\sep \ell_p^\theta(\Z;X)} <\infty },
\]
where 
\[
\norm{x\sep \ell_p^\theta(\Z;X)}
:=
\ssgrklam{\sum_{\nu\in\Z} e^{\nu\theta}\norm{x_\nu \sep X}^p}^{1/p}.
\]
We write $\ell_p(\Z;X):=\ell_p^0(\Z;X)$ and $\ell_p(\Z):=\ell_p(\Z;\R)$. 
Recall that if $X$ is a (reflexive) Banach space, then also $\grklam{\ell_p^\theta(\Z;X),\norm{\cdot\sep \ell_p^\theta(\Z;X)}}$ is a (reflexive) Banach space, see~\cite[Theorem~VI.2.1.1]{Ama2019}.

\begin{prop}\label{prop:basic}
    Let $1<p<\infty$ and $\gamma,\Theta,\theta\in\R$. 
    Then $\grklam{H^\gamma_{p,\Theta,\theta}(\D),\norm{\cdot \sep H^\gamma_{p,\Theta,\theta}(\D)}}$ is a  reflexive Banach space.
\end{prop}
\begin{proof}
By \autoref{lem:LotSpaces:properties}\ref{it:LotSpaces:Banach}, $H^\gamma_{p,\Theta}(\domain)$ is a reflexive Banach space, thus so is $\ell_p^\theta(\Z;H^\gamma_{p,\Theta}(\D))$.
Take $\zeta$ as in \autoref{conv:zeta} and let $(\eta_\nu)_{\nu\in\Z}$ as well as $S:=S_e$ and $R:=R_e$ be as in \autoref{rem:resolution}\ref{it:resolution:r-c}.
As mentioned therein, our assertion follows once we can prove that 
\[
    S\in\L\big(H^\gamma_{p,\Theta,\theta}(\D),\ell_p^\theta(\Z;H^\gamma_{p,\Theta}(\D))\big)
    \quad\text{and}\quad 
    R\in \L\big(\ell_p^\theta(\Z;H^\gamma_{p,\Theta}(\D)),H^\gamma_{p,\Theta,\theta}(\D)\big).
\]
The former is obvious, where $S$ is even isometric.
To see the latter, note that
by \autoref{rem:resolution}\ref{it:resolution:cover} and~\ref{it:resolution:constr:D}, for all $f=(f_\nu)_{\nu\in \Z}\in \ell_p^\theta(\Z;H^\gamma_{p,\Theta}(\D))$, it holds that
\[
    \zeta_\nu Rf
    = \zeta_\nu \sum_{\mu\in\Z}\eta_\mu f_\mu(e^{-\mu}\cdot)
    = \sum_{j=-2}^2 \zeta_\nu\eta_{\nu+j}f_{\nu+j}(e^{-(\nu+j)}\cdot),
    \qquad \nu\in\Z.
\]
Thus, arguing along the lines of the proof of \autoref{lem:norm-change}, by applying \autoref{lem:Lot_pw-mult-special} we obtain
\begin{align*}
    \norm{\grklam{\zeta_\nu Rf}(e^\nu\cdot) \sep H^\gamma_{p,\Theta}(\D)}
    \lesssim \sum_{j=-2}^2 \norm{f_{\nu+j} \sep H^\gamma_{p,\Theta}(\D)},
    \qquad \nu\in\Z,\quad f\in \ell_p^\theta(\Z;H^\gamma_{p,\Theta}(\D)),
\end{align*}
so that
\begin{align*}
    \norm{Rf \sep H^\gamma_{p,\Theta,\theta}(\D)}^p
    &\lesssim \sum_{j=-2}^2 \sum_{\nu\in\Z} e^{(\nu+j)\theta} \norm{f_{\nu+j} \sep H^\gamma_{p,\Theta}(\D)}^p \\
    &\sim \norm{f \sep \ell_p^\theta(\Z;H^\gamma_{p,\Theta}(\D))}^p,
    \qquad\qquad\quad f=(f_\nu)_{\nu\in\Z}\in \ell_p^\theta(\Z;H^\gamma_{p,\Theta}(\D)).
\end{align*}
Thus, $R\in \L\big(\ell_p^\theta(\Z;H^\gamma_{p,\Theta}(\D)),H^\gamma_{p,\Theta,\theta}(\D)\big)$. 
\end{proof}

The following embedding result includes the monotonicity of the scale  $H^\gamma_{p,\Theta,\theta}(\D)$ w.r.t.\ the smoothness parameter $\gamma$ as well as a Sobolev-type embedding as special cases.
\begin{prop}\label{prop:embeddings} 
Let $1<p_0 \leq p_1<\infty$, as well as $\gamma_0,\gamma_1,\Theta,\theta\in\R$ with $\gamma_0\geq \gamma_1$ such that
	\begin{align*}
		\gamma_0 - \frac{d}{p_0} \geq  \gamma_1 - \frac{d}{p_1}.
	\end{align*}
	Then $H^{\gamma_0}_{p_0,\Theta p_0,\theta p_0}(\D) \hookrightarrow  H^{\gamma_1}_{p_1,\Theta p_1,\theta p_1}(\D)$.
\end{prop}
\begin{proof}
Since $\ell_{p_0}(\Z)\hookrightarrow \ell_{p_1}(\Z)$, the assertion is an immediate consequence of \autoref{def:Hspaces} and \autoref{lem:LotSpaces:properties}\ref{it:LotSpaces:embeddings}.
\end{proof}

We conclude the current subsection with elementary monotonicity assertions related to the weight parameters $\Theta$ and $\theta$.

\begin{lemma}
\label{lem:monotone}
	Let $1<p<\infty$ as well as $\gamma,\Theta_0,\Theta_1,\theta\in\R$ with $\Theta_0\leq \Theta_1$. Then
	$H^{\gamma}_{p,\Theta_0,\theta}(\D) \hookrightarrow H^{\gamma}_{p,\Theta_1,\theta}(\D)$.
\end{lemma}
\begin{proof}
Since $(\zeta_\nu)_{\nu\in\Z}$ satisfies~\eqref{eq:dilation}, 
we have that
$\supp(\zeta_\nu u)(c^\nu\cdot)\subset \D^{[0]}_{e,1}(\{0\})$ for all $\nu\in\Z$. 
Thus, similar to \autoref{lem:LotSpaces:properties}\ref{it:LotSpaces:bddDom}, we obtain
\[
	\norm{(\zeta_\nu u)(e^\nu\cdot) \sep H^{\gamma}_{p,\Theta_1}(\D)} 
    \lesssim \norm{(\zeta_\nu u)(e^\nu\cdot) \sep H^{\gamma}_{p,\Theta_0}(\D)}, 
    \qquad u\in H^\gamma_{p,\Theta_0,\theta}(\D), \quad \nu\in\Z, 
\]	
and the assertion follows.
\end{proof}

\begin{lemma}
\label{lem:monotone2}
Let $1<p<\infty$ as well as $\gamma,\Theta,\theta_0,\theta_1\in\R$ with $\theta_0\leq \theta_1$. Moreover, let	$0<r<R<\infty$.
\begin{enumerate}[label=\textup{(\roman*)}]
    \item\label{it:mon1} If $u\in H^{\gamma}_{p,\Theta,\theta_0}(\D)$ satisfies $\supp(u) \subset  \overline{B_{R}(0)}$, then $u\in H^{\gamma}_{p,\Theta,\theta_1}(\D)$ and
    \[
        \norm{u\sep H^{\gamma}_{p,\Theta,\theta_1}(\D)} 
        \lesssim R^{(\theta_1-\theta_0)/p} \,\norm{u\sep H^{\gamma}_{p,\Theta,\theta_0}(\D)}
    \]
    with constants independent of $u$ and $R$.

    \item\label{it:mon2} If $u\in H^{\gamma}_{p,\Theta,\theta_1}(\D)$ satisfies $\supp(u) \subset  \overline{\D}\setminus B_{r}(0)$, then $u\in H^{\gamma}_{p,\Theta,\theta_0}(\D)$ and
    \[
        \norm{u\sep H^{\gamma}_{p,\Theta,\theta_0}(\D)} 
        \lesssim r^{(\theta_0-\theta_1)/p} \,\norm{u\sep H^{\gamma}_{p,\Theta,\theta_1}(\D)}
    \]
    with constants independent of $u$ and $r$.
\end{enumerate}
In particular, if $u\in\mathscr{D}'(\D)$ with $\supp(u)\subset (\overline{\D}\cap \overline{B_R(0)})\setminus B_r(0)$, then $u\in H^\gamma_{p,\Theta,\theta}(\D)$ if, and only if, $u\in H^\gamma_{p,\Theta}(\D)$. In this case, $\norm{u\sep H^\gamma_{p,\Theta,\theta}(\D)}\sim \norm{u\sep H^\gamma_{p,\Theta}(\D)}$.
\end{lemma}
\begin{proof}
To show~(i) we choose $K\in\Z$ such that $e^{K-1} \leq R < e^{K}$.
Then, since $\supp(u) \subset \overline{B_{R}(0)}$, we have $\zeta_\nu u\equiv 0$ on $\D$ for all $\nu\geq K+1$.
Thus, since $\theta_0\leq \theta_1$, 
\begin{align*}
    \norm{u\sep H^{\gamma}_{p,\Theta,\theta_1}(\D)}^p
    &= \sum_{\nu \leq K} e^{\nu \theta_1} \, \norm{(\zeta_\nu u)(e^{\nu}\cdot) \sep H^\gamma_{p,\Theta}(\D)}^p \\
    &\leq e^{K(\theta_1-\theta_0)} \norm{u \sep H^{\gamma}_{p,\Theta,\theta_0}(\D)}^p \\
    &\lesssim R^{(\theta_1-\theta_0)} \norm{u \sep H^{\gamma}_{p,\Theta,\theta_0}(\D)}^p.
\end{align*}
Part~(ii) follows with very similar arguments.
Moreover, if $u\in\mathscr{D}'(\D)$ is such that $\supp(u)\subset (\overline{\D}\cap \overline{B_R(0)})\setminus B_r(0)$, then there exists some $K\in\N$ such that $\zeta_\nu u \equiv 0$ on $\D$ for all $\nu\in\Z$ with $\abs{\nu}>K$. 
Thus, by \autoref{def:Hspaces} and \autoref{lem:Lot_pw-mult-special}\ref{it:LotSpaces:D:dilation}, we have that
\begin{align*}
    \norm{u\sep H^{\gamma}_{p,\Theta,\theta}(\D)}
    &= \left( \sum_{\nu\in\Z} e^{\nu\theta}\norm{(\zeta_\nu u)(e^\nu\cdot)\sep H^{\gamma}_{p,\Theta}(\D)}^p \right)^{1/p} \\
    &\sim \left( \sum_{\nu\in\Z} e^{\nu(\theta-\Theta)}\norm{\zeta_\nu u\sep H^{\gamma}_{p,\Theta}(\D)}^p \right)^{1/p} 
    \sim \sum_{\abs{\nu}\leq K} \norm{\zeta_\nu u\sep H^{\gamma}_{p,\Theta}(\D)}.
\end{align*}
Now \autoref{lem:Lot_pw-mult-special}\ref{it:LotSpaces:D:localization} yields $\norm{u\sep H^{\gamma}_{p,\Theta,\theta}(\D)} \lesssim \norm{u\sep H^{\gamma}_{p,\Theta}(\D)}$ provided $u\in H^{\gamma}_{p,\Theta}(\D)$.
And, if $u\in H^{\gamma}_{p,\Theta,\theta}(\D)$, then the converse estimate is implied by the triangle inequality in $H^{\gamma}_{p,\Theta}(\D)$, as $\sum_{\nu\in\Z} \zeta_\nu\equiv 1$ on $\D$.
This proves the final assertion.
\end{proof}

\subsection{Density of \texorpdfstring{$C_0^\infty(\D)$}{test functions}, duality, and interpolation}
\label{sec:Density}

For many properties of the spaces $H^\gamma_{p,\Theta,\theta}(\D)$ that we address below, in particular, for the duality and interpolation statements, we   shall   
use the density of $ C^\infty_0(\D)$ in $H^\gamma_{p,\Theta,\theta}(\D)$ which we prove now.

\begin{theorem}\label{thm:density}
Let $1<p<\infty$ and $\gamma,\Theta,\theta\in\R$. Then $ C_0^\infty(\D)$ is dense in $H^\gamma_{p,\Theta,\theta}(\D)$.
\end{theorem}
\begin{proof}
For $k\in\N$ let $h_k:=\sum_{\abs{\nu}\leq k}\zeta_\nu$. Then $\supp(h_k)\subset \D^{[0]}_{e,k+1}(\{0\})$ and $h_k\equiv 1$ on $\D^{[0]}_{e,k}(\{0\})$ for all $k\in\N$.
In particular, applying \autoref{lem:Lot_pw-mult-special} again and using similar arguments as in the proof of \autoref{lem:norm-change}, yields that for all $u\in H^\gamma_{p,\Theta,\theta}(\D)$,
\[
    (1-h_k) u \longrightarrow 0 \quad \text{in}\quad H^\gamma_{p,\Theta,\theta}(\D), \quad\text{as}\quad k\to\infty.
\]
Since $\supp(h_k)\subset \D^{[0]}_{e,k+1}(\{0\})$ for all $k\in\N$, by \autoref{lem:monotone2}, 
it  holds  
\[
    \norm{h_k u \sep H^\gamma_{p,\Theta,\theta}(\D)} \sim \norm{h_k u \sep H^\gamma_{p,\Theta}(\D)}.
\]
Thus, the assertion follows from the fact that $h_k u$ can be approximated by $C_0^\infty(\D)$ functions which are dense in $H^\gamma_{p,\Theta}(\D)$, see  \autoref{lem:LotSpaces:properties}\ref{it:LotSpaces:density}.  
\end{proof}

Next we show that the duals of the spaces introduced above can be characterized by means of spaces from the same scale. 
For an alternative characterization of the duals of spaces with regularity $\gamma\in\N_0$, we refer to \autoref{prop:dual:N} below.

\begin{theorem}
\label{thm:dual}
Assume that $1<p,p'<\infty$ and  $\gamma,\Theta,\Theta',\theta,\theta'\in\R$ satisfy
\begin{equation}\label{eq:dual:parameters}
		\frac{1}{p}+\frac{1}{p'}=1,  
		\qquad \text{and}\qquad  \frac{\Theta}{p}+\frac{\Theta'}{p'}=\frac{\theta}{p}+\frac{\theta'}{p'}=2.  
\end{equation}
Then 
\begin{equation}\label{eq:dualform:D}
    (\varphi,\psi):=\int_\D \varphi(x)\,\psi(x)\d x, 
    \qquad \varphi,\psi\in  C_0^\infty(\D),    
\end{equation}
extends uniquely to a continuous bilinear form on $H^\gamma_{p,\Theta,\theta}(\D)\times H^{-\gamma}_{p',\Theta',\theta'}(\D)$ which provides the isomorphism 
\[
    \big(H^\gamma_{p,\Theta,\theta}(\D)\big)' = H^{-\gamma}_{p',\Theta',\theta'}(\D).
\]	
\end{theorem}

\begin{proof}
\emph{Step 1. } We first show that 
\begin{align}\label{eq:dual:test}
    (\varphi,\psi) 
    \lesssim \norm{\varphi \sep H^\gamma_{p,\Theta,\theta}(\D)} \norm{\psi \sep H^{-\gamma}_{p',\Theta',\theta'}(\D)},
    \qquad \varphi,\psi\in C_0^\infty(\D).
\end{align}
As a consequence, due to \autoref{thm:density}, the form $(\cdot,\cdot)$ defined in~\eqref{eq:dualform:D} uniquely extends  
to a continuous bilinear form on $H^\gamma_{p,\Theta,\theta}(\D)\times H^{-\gamma}_{p',\Theta',\theta'}(\D)$ and  $\psi\mapsto (\cdot,\psi)$ is a continuous embedding from $H^{-\gamma}_{p',\Theta',\theta'}(\D)$ to $(H^\gamma_{p,\Theta,\theta}(\D))'$. 
To verify~\eqref{eq:dual:test} let $\zeta=(\zeta_\nu)_{\nu\in\Z}\in\mathscr{A}^{\textup{[R]}}_{e,1}(\D,\{0\})$ and  $\eta=(\eta_\nu)_{\nu\in\Z}\in\mathscr{A}^{\textup{[L]}}_{e,2}(\D,\{0\})$ be as in \autoref{rem:resolution}\ref{it:resolution:r-c} with $c=e$.
Then, since $\eta_\nu\equiv 1$ on $\supp(\zeta_\nu)$ for all $\nu\in\Z$, by \autoref{lem:LotSpaces:properties}\ref{it:LotSpaces:duality} and Hölder's inequality, we have that indeed
\begin{align*}
(\varphi,\psi)
&=
\sum_{\nu\in\Z}\int_\D \zeta_\nu(x)\eta_\nu(x)\varphi(x)\psi(x)\,\dx 
=
\sum_{\nu\in\Z}e^{2\nu}\int_\D \zeta_\nu(e^\nu x)\eta_\nu(e^\nu x)\varphi(e^\nu x) \psi(e^\nu x)\,\dx \\
&\lesssim
\sum_{\nu\in\Z} e^{\nu\theta/p} \norm{(\zeta_\nu\varphi)(e^\nu\cdot) \sep H^\gamma_{p,\Theta}(\D)} e^{\nu\theta'/p'}\norm{(\eta_\nu\psi)(e^\nu\cdot) \sep H^{-\gamma}_{p',\Theta'}(\D)}\\
&\lesssim
\norm{\varphi \sep H^\gamma_{p,\Theta,\theta}(\D)} \norm{\psi \sep H^{-\gamma}_{p',\Theta',\theta'}(\D)},
\qquad\qquad\qquad\qquad\qquad\qquad\quad \varphi,\psi\in  C_0^\infty(\D).
\end{align*}

\emph{Step 2. } Now we show that for every $L\in (H^{\gamma}_{p,\Theta,\theta}(\D))'$ there is some $\psi\in \mathscr{D}'(\D)$ such that $L(\varphi)=\psi(\varphi)$ for all $\varphi\in  C_0^\infty(\D)$.
To this end, let $h=(h_k)_{k\in\N}$ be as in the proof of \autoref{thm:density}.
Then, by \autoref{lem:monotone2},
\begin{align*}
L(h_k\varphi) 
\leq C_L 
\norm{h_k\varphi \sep H^\gamma_{p,\Theta,\theta}(\D)}
\sim 
\norm{h_k\varphi \sep H^\gamma_{p,\Theta}(\D)}
\lesssim 
\norm{\varphi \sep H^\gamma_{p,\Theta}(\D)}, \qquad \varphi\in  C_0^\infty(\D),
\end{align*}
where we used \autoref{lem:LotSpaces:properties}\ref{it:LotSpaces:multiplier} for the last estimate (the constants here may depend on $k\in\N$).
Due to the density of $ C_0^\infty(\D)$ in $H^\gamma_{p,\Theta}(\D)$, this shows that $L(h_k\cdot) \in \big( H^\gamma_{p,\Theta}(\D) \big)'$ for all $k\in\N$. 
Thus, by \autoref{lem:LotSpaces:properties}\ref{it:LotSpaces:duality}, for each $k\in\N$ there exists $\psi_k \in H^{-\gamma}_{p',\Theta'}(\D)\subset \mathscr{D}'(\D)$ such that 
\[
L(h_k\varphi)=\psi_k(\varphi) \quad \text{for all}\quad \varphi \in  C_0^\infty(\D).
\]	
The choice of $h_k$ implies that 
\[
    \psi_{k+1}(\varphi)=\psi_k(\varphi)=:\psi(\varphi) \qquad \text{for } \varphi\in C_0^\infty(\D) \text{ and } k\in\N \text{ such that } \supp(\varphi)\subset \D^{[0]}_{e, k}(\{0\})
\]	
yields a well-defined generalized function $\psi \in \mathscr{D}'(\D)$. 
It satisfies $\psi= \lim_{k\to\infty} \psi_k$ and
\[
\psi(\varphi)=\psi_k(\varphi)=L(h_k\varphi)=L(\varphi)
\]
for all $\varphi\in C_0^\infty(\D)$ and suitably chosen $k\in\N$.

\emph{Step 3. } To complete the proof we need to show that $\psi\in\mathscr{D}'(\D)$ from Step~2 belongs to $H^{-\gamma}_{p',\Theta',\theta'}(\D)$.
For this purpose, let $b:=(b_\nu)_{\nu\in\Z}$ be given by
\[
b_\nu
:= 
e^{\nu \theta'/p'} \norm{(\zeta_\nu \psi)(e^\nu \cdot) \sep H^{-\gamma}_{p',\Theta'}(\D)}, \qquad \nu\in\Z,
\]
so that 
then it is enough to prove that
\begin{align}\label{eq:norm_psi-new}
\norm{\psi \sep H^{-\gamma}_{p',\Theta',\theta'}(\D)} 
= 
\norm{b \sep \ell_{p'}(\Z)} 
= 
\sup_{\norm{a \sep \ell_p(\Z)}=1} \sum_{\nu\in\Z} a_\nu\, b_\nu 
<\infty.
\end{align}
For every sequence $a:=(a_\nu)_{\nu\in\Z}$ of real numbers with $\norm{a \sep \ell_p(\Z)}=1$ we choose a corresponding sequence $(g_\nu)_{\nu\in\Z}$ in $  C_0^\infty(\D)$ 
such that 
\[
\norm{g_\nu \sep H^\gamma_{p,\Theta}(\D)}= \abs{a_\nu} \, e^{-\nu \theta/p}
\qquad\text{and} \qquad 
a_\nu\, b_\nu \lesssim \psi\big(g_\nu( e^{-\nu} \cdot)\xi_\nu\big), \quad \nu\in\Z,
\]
with constants that are also independent of $a$ and $\psi$.
This is indeed possible, since the duality $\big(H^{\gamma}_{p,\Theta}(\D)\big)'=H^{\gamma'}_{p',\Theta'}(\D)$ and the density of $C_0^\infty(\D)$ in $H^{\gamma}_{p,\Theta}(\D)$ show that
\begin{align*}
a_\nu\, b_\nu 
&\leq 
\abs{a_\nu}\, e^{\nu \theta'/p'} \norm{(\zeta_\nu \psi)(e^\nu \cdot) \sep H^{-\gamma}_{p',\Theta'}(\D)}\\
&\sim 
\abs{a_\nu} \, e^{\nu \theta'/p'}\sup_{\norm{\varphi \sep H^{\gamma}_{p,\Theta}(\D)}=1} \abs{\big((\zeta_\nu \psi)(e^\nu \cdot)\big)(\varphi)} \\
&=
e^{2\nu} \sup_{\substack{\norm{\varphi \sep H^{\gamma}_{p,\Theta}(\D)}=1\\\varphi\in  C_0^\infty(\D)}} \abs{\big((\zeta_\nu \psi)(e^\nu \cdot)\big)(a_\nu \, e^{-\nu \theta/p} \varphi)} \\
&\leq 2\, 
e^{2\nu} \big((\zeta_\nu \psi)(e^\nu \cdot)\big)(g_\nu) \\
&\sim \psi\big(g_\nu( e^{-\nu})\zeta_\nu\big)
\end{align*}
for some $g_\nu \in C_0^\infty(\D)$ with norm $\abs{a_\nu}  e^{-\nu \theta/p}$, $\nu\in\Z$, and constants
independent of $a$, $\nu$, and~$\psi$. 
For $M\in\N$  let
\[
\varphi_M:=\sum_{k\in\Z\colon \abs{k}\leq M} g_k(e^{-k}\cdot)\, \zeta_k \in  C_0^\infty(\D).
\]
Then the support properties of $\zeta=(\zeta_\nu)_{\nu\in\Z}$ and \autoref{lem:Lot_pw-mult-special} yield
\begin{align*}
\norm{(\zeta_\nu \varphi_M)(e^\nu\cdot) \sep H^{\gamma}_{p,\Theta}(\D)}
&= 
\Big\lVert 
\zeta_\nu(e^\nu\cdot) \sum_{\abs{k}\leq M} g_k(e^{\nu-k}\cdot)\, \zeta_k(e^\nu\cdot) \,\big|\, H^{\gamma}_{p,\Theta}(\D)\Big\rVert \\
&\leq 
\sum_{j=-1}^1 \norm{ \zeta_\nu(e^\nu\cdot)\, g_{\nu+j}(e^{-j}\cdot)\, \zeta_{\nu+j}(e^\nu\cdot) \sep H^{\gamma}_{p,\Theta}(\D)}\\
&\lesssim 
\sum_{j=-1}^1 \norm{ g_{\nu+j}\, \zeta_{\nu+j}(e^{\nu+j}\cdot) \sep H^{\gamma}_{p,\Theta}(\D)},\qquad\qquad \nu\in\Z,\, M\in\N,
\end{align*}
  again with constants independent of $a$ and $\psi$.   
Thus, 
\[
e^{\nu\theta} \norm{(\zeta_\nu \varphi_M)(e^\nu\cdot) \sep H^{\gamma}_{p,\Theta}(\D)}^p
\lesssim 
\sum_{j=-1}^1 e^{(\nu+j)\theta} \norm{ g_{\nu+j} \sep H^{\gamma}_{p,\Theta}(\D)}^p 
= 
\sum_{j=-1}^1 \abs{a_{\nu+j}}^p, \qquad \nu\in\Z,
\]
and hence
\[
    \norm{\varphi_M \sep H^{\gamma}_{p,\Theta,\theta}(\D)}^p 
    = \sum_{\nu\in\Z} e^{\nu\theta} \norm{(\zeta_\nu \varphi_M)(e^\nu\cdot) \sep H^{\gamma}_{p,\Theta}(\D)}^p
    \lesssim \norm{a \sep \ell_p(\Z)}^p
    =1, 
    \qquad M\in\N,
\]
with constants that are also independent of $a=(a_\nu)_{\nu\in\Z}$ and $\psi$.
Thus, by Step~2,
\begin{align*}
    \sum_{\substack{\abs{\nu}\leq M}} a_\nu b_\nu 
    \lesssim \sum_{\abs{\nu}\leq M} \psi\big(g_\nu(e^{-\nu}\cdot)\zeta_\nu \big)
    = \psi(\varphi_M)
    = L(\varphi_M)
    \leq \big\lVert L \,|\, \big(H^\gamma_{p,\Theta,\theta}(\D)\big)' \big\rVert \, \norm{\varphi_M \,|\, H^{\gamma}_{p,\Theta,\theta}(\D)}
\end{align*}
is bounded uniformly in $M$ and therefore~\eqref{eq:norm_psi-new} holds.
\end{proof}

As a corollary we obtain that the spaces $H^\gamma_{p,\Theta,\theta}(\D)$ are  continuously embedded in the space $\mathscr{D}'(\D)$ of generalized functions. This in turn implies that  any two of them form an interpolation couple in the sense of~\cite[Section~2.3]{BL76}.

\begin{corollary}\label{cor:ContDistr}
    For all $1<p<\infty$ and $\gamma,\Theta,\theta\in\R$ we have $H^\gamma_{p,\Theta,\theta}(\D)\hookrightarrow \mathscr{D}'(\D)$.
\end{corollary}
\begin{proof}
Since by definition $H^\gamma_{p,\Theta,\theta}(\D)\subset \mathscr{D}'(\D)$, it is enough to show that convergence in $H^\gamma_{p,\Theta,\theta}(\D)$ implies convergence in the sense of distributions.
Let $u,u_k \in H^\gamma_{p,\Theta,\theta}(\D)$, $k\in\N$, be such that $\norm{u_k - u\sep  H^\gamma_{p,\Theta,\theta}(\D)}\longrightarrow 0$ as $k\to\infty$. 
Then \autoref{thm:dual} allows to identify $u_k-u$ as elements in $\big( H^{-\gamma}_{p',\Theta',\theta'}(\D) \big)'$, where $p'$, $\Theta'$, and $\theta'$ are as in \eqref{eq:dual:parameters}, and for all $\varphi \in  C_0^\infty(\D)\subset H^{-\gamma}_{p',\Theta',\theta'}(\D)$ we have
\begin{align*}
\abs{u_k(\varphi)-u(\varphi)} 
&=\abs{(u_k-u)(\varphi)} \\
&\lesssim \norm{u_k-u \sep \big( H^{-\gamma}_{p',\Theta',\theta'}(\D) \big)'} \, \norm{\varphi \sep H^{-\gamma}_{p',\Theta',\theta'}(\D)} \\
&\lesssim \norm{u_k-u \sep H^{\gamma}_{p,\Theta,\theta}(\D) } \, \norm{\varphi \sep H^{-\gamma}_{p',\Theta',\theta'}(\D)}
\longrightarrow 0,
\end{align*}
as $k\to\infty$. 
In other words, $u_k\longrightarrow u$ in $\mathscr{D}'(\D)$, as claimed.
\end{proof}

Concerning complex interpolation in the sense of Calder\'on (denoted by the functor $[\cdot,\cdot]_\vartheta$, see \cite[Chapter~4]{BL76}), we have the following result. 
\begin{theorem}
\label{thm:interpolation}
	Let $\gamma_0,\gamma_1,\Theta_0,\Theta_1,\theta_0,\theta_1\in\R$, as well as $1<p_0,p_1<\infty$. 
	For $0<\vartheta<1$ we let 
    \[
    	\frac{1}{p} := \frac{1-\vartheta}{p_0} + \frac{\vartheta}{p_1},
	\]
	as well as
	\[ 
    	\gamma:=(1-\vartheta) \gamma_0+ \vartheta \gamma_1,
    	\qquad
    	\Theta:=(1-\vartheta) \Theta_0+ \vartheta \Theta_1,
    	\qquad \text{and}\qquad
    	\theta:=(1-\vartheta) \theta_0+ \vartheta \theta_1.
	\]
Then we have
	\begin{equation}\label{eq:interpolation}
	 \big[H^{\gamma_0}_{p_0,\Theta_0 p_0,\theta_0 p_0}(\D), H^{\gamma_1}_{p_1,\Theta_1 p_1,\theta_1 p_1}(\D)\big]_\vartheta = H^{\gamma}_{p,\Theta p, \theta p}(\D)   
	\end{equation}
isomorphically.
\end{theorem}
\begin{proof}
    For $i=0,1$ let 
	$$
		A_i := \ell_{p_i}^{\theta_i p_i}\big(\Z; X_i\big)
    	\quad\text{with}\quad
    	X_i:= H^{\gamma_i}_{p_i,\Theta_i p_i}(\D),
    	\qquad\text{and}\qquad
    	B_i:= H^{\gamma_i}_{p_i,\Theta_i p_i,\theta_i p_i}(\D).
	$$
	Further, set $A := \ell_{p}^{\theta p}\big(\Z; X\big)$ with $X:= H^{\gamma}_{p,\Theta p}(\D)$ and $B:= H^{\gamma}_{p,\Theta p,\theta p}(\D)$.	
	Then \autoref{lem:LotSpaces:properties}\ref{it:LotSpaces:interpolation} shows $[X_0,X_1]_{\vartheta}=X$
	and hence \cite[Theorem~VI.2.3.4(ii)]{Ama2019} yields
	\begin{align*}
		[A_0,A_1]_{\vartheta} 
		&= \big[ \ell_{p_0}^{\theta_0 p_0}\big(\Z; X_0\big), \ell_{p_1}^{\theta_1 p_1}\big(\Z; X_1\big) \big]_\vartheta 
		= \ell_{p}^{\theta p}\big(\Z; [ X_0, X_1 ]_\vartheta\big)
		= \ell_{p}^{\theta p}\big(\Z; X \big) 
		= A.
	\end{align*}
	Recall the definition of the operators $S$ and $R$ from the proof of \autoref{prop:basic} (see also \autoref{rem:resolution}\ref{it:resolution:r-c}). 
	We have seen there that $R$ is a retraction which continuously maps $A_i$ into~$B_i$, $i=0,1$, as well as $A$ into $B$ and that in all cases $S$ is a corresponding coretraction.
	In addition, $\{B_0,B_1\}$ is an interpolation couple, see \autoref{cor:ContDistr}. 
	Thus, by~\cite[Theorem~1.2.4]{Tri1995}, $S$ is an isomorphism from $[B_0,B_1]_\vartheta$ onto the closed subspace $\mathrm{ran}\big(SR\big|_{[A_0,A_1]_\vartheta}\big)$ of $[A_0,A_1]_\vartheta=A$.
	Thus, $u\in \big[H^{\gamma_0}_{p_0,\Theta_0 p_0,\theta_0 p_0}(\D), H^{\gamma_1}_{p_1,\Theta_1 p_1,\theta_1 p_1}(\D)\big]_\vartheta$ 
	implies $Su \in \ell_{p}^{\theta p}\big(\Z; H^{\gamma}_{p,\Theta p}(\D) \big)$ and 
	\begin{align}\label{eq:ipol_norm_new}
		\norm{u \sep \big[H^{\gamma_0}_{p_0,\Theta_0 p_0,\theta_0 p_0}(\D), H^{\gamma_1}_{p_1,\Theta_1 p_1,\theta_1 p_1}(\D)\big]_\vartheta} \sim \norm{Su \sep \ell_{p}^{\theta p}\big(\Z; H^{\gamma}_{p,\Theta p}(\D) \big)}.
	\end{align}
	Therefore,  the proof of \autoref{prop:basic} (see also \autoref{rem:resolution}\ref{it:resolution:r-c}) yields $u\in H^{\gamma}_{p,\Theta p, \theta p}(\D)$. 
	
	Conversely,  $u\in H^{\gamma}_{p,\Theta p, \theta p}(\D)$ implies  $Su\in \ell_{p}^{\theta p}\big(\Z; H^{\gamma}_{p,\Theta p}(\D) \big)=A$ and hence $Su=SR\big|_A Su \in \mathrm{ran}\big(SR\big|_{A}\big)$ proves  that $u\in \big[H^{\gamma_0}_{p_0,\Theta_0 p_0, \theta_0 p_0}(\D), H^{\gamma_1}_{p_1,\Theta_1 p_1, \theta_1 p_1}(\D)\big]_\vartheta$. 
	In conclusion, the interpolation formula~\eqref{eq:interpolation} holds and
the norm equivalence follows from~\eqref{eq:ipol_norm_new} and the fact that $S$ is an isometry.
\end{proof}
 
\subsection{Relation to weighted Sobolev spaces}
\label{sec:characterization}

The next theorem shows that for $\gamma\in\N_0$ the spaces $H^\gamma_{p,\Theta,\theta}(\D)$ can be characterized as weighted Sobolev spaces.
If $\domain\subset\overline{\D}$ is measurable, then for $1<p<\infty$ and $\Theta,\theta\in\R$ 
we write $L_{p,\Theta,\theta}(\domain):=L_p(\domain,\mathcal{B}(\domain),w_{\Theta,\theta}\d x)$ for the  
weighted $L_p(\domain)$-space with weight   
\[
    w_{\Theta,\theta}(x)
    :=\rho_\circ(x)^{\theta-2} \left(\frac{\rho_\D(x)}{\rho_\circ(x)} \right)^{\Theta-2},
    \qquad x\in \domain.
\]

\begin{theorem}
\label{thm:coincidence}
	Let $1<p<\infty$, $\gamma\in\N_0$, and $\Theta,\theta\in\R$.
	Then
	$$
		H^\gamma_{p,\Theta,\theta}(\D) 
		= 
		\left\{ u \in L_{1,\loc}(\D) : \  \normmod{u \sep H^\gamma_{p,\Theta,\theta}(\D)} < 	\infty\right\},
	$$
	where
	$$
	\normmod{u \sep H^\gamma_{p,\Theta,\theta}(\D)} 
	:= 
	\Bigg( \sum_{\alpha\in\N_0^2\colon \abs{\alpha}\leq \gamma} 
	\norm{\rho_\D^{\abs{\alpha}} \, D^\alpha u \sep L_{p,\Theta,\theta}(\D)}^p  \Bigg)^{1/p}
	$$
	is an equivalent norm.
\end{theorem}
\begin{proof}
Take $\zeta$ from \autoref{conv:zeta} and let  
$\eta_\nu:=\zeta_{\nu-1}+\zeta_\nu+\zeta_{\nu+1}$, $\nu\in\Z$.
For all $\nu\in\Z$ we then have $\eta_\nu\equiv 1$ on $\D_\nu:=\D^{[\nu]}_{e,1}(\{0\})\supset \supp(\zeta_\nu)$, cf.~\autoref{rem:resolution}\ref{it:resolution:constr:D}.
Thus, for all $u\in\mathscr{D}'(\D)$ and $\nu\in\Z$, 
\begin{align*}
		\sum_{\abs{\alpha}\leq \gamma} \norm{ \rho_\D^{\abs{\alpha}} \, D^\alpha u \sep L_{p,\Theta,\theta}(\D_\nu)}^p
		&\lesssim \sum_{j=-1}^{1}  \sum_{\abs{\alpha}\leq \gamma} \norm{ \rho_\D^{\abs{\alpha}} \, D^\alpha \big( \zeta_{\nu+j} \, u \big) \sep L_{p,\Theta,\theta}(\D_{\nu+j})}^p.
\end{align*}
	On the other hand, Leibniz' rule implies
	\begin{align*}
		\sum_{\abs{\alpha}\leq \gamma} \norm{ \rho_\D^{\abs{\alpha}} \, D^\alpha (\zeta_\nu\, u) \sep L_{p,\Theta,\theta}(\D_\nu)}^p
		& \lesssim \sum_{\abs{\alpha}\leq \gamma} \sum_{\beta \leq \alpha} \norm{ \rho_\D^{\abs{\alpha}} \, (D^\beta \zeta_\nu) \, (D^{\alpha-\beta} u) \sep L_{p,\Theta,\theta}(\D_\nu)}^p \\
		&\lesssim \sum_{\abs{\alpha}\leq \gamma} \norm{ \rho_\D^{\abs{\alpha}} \, D^\alpha u \sep L_{p,\Theta,\theta}(\D_\nu)}^p, \qquad \nu\in\Z,
	\end{align*}
	since for $\beta \leq\alpha$ it holds that
\begin{align*}
	\abs{\rho_\D(x)^{\abs{\alpha}} \, (D^\beta \zeta_\nu)(x)} 
	\lesssim \rho_\D(x)^{\abs{\alpha}-\abs{\beta}}\, \left( \frac{\rho_\D(x)}{e^{\nu}}\right)^{\abs{\beta}} 
	\lesssim \rho_\D(x)^{\abs{\alpha-\beta}}, \qquad x\in \D_\nu,\quad \nu\in\Z.
\end{align*}
	Furthermore, \autoref{lem:Lot_pw-mult-special}\ref{it:LotSpaces:D:dilation} and \autoref{lem:LotSpaces:properties}\ref{it:LotSpaces:N} yield  
	\begin{align*}
		\sum_{\abs{\alpha}\leq \gamma} \norm{ \rho_\D^{\abs{\alpha}} \, D^\alpha (\zeta_\nu\, u) \sep L_{p,\Theta,\theta}(\D_\nu)}^p
		&\sim e^{\nu(\theta-\Theta)} \sum_{\abs{\alpha}\leq \gamma} \int_{\D_\nu} \abs{\rho_\D(x)^{\abs{\alpha}}\, D^\alpha (\zeta_\nu\, u)(x) }^p \rho_\D(x)^{\Theta-2} \d x \\
		&=e^{\nu(\theta-\Theta)} \norm{\zeta_{\nu}\,u \sep H^{\gamma}_{p,\Theta}(\D)}^p 
		\sim e^{\nu\theta} \norm{(\zeta_{\nu}\,u)(e^\nu\cdot) \sep H^{\gamma}_{p,\Theta}(\D)}^p
	\end{align*}
since $w_{\Theta,\theta}\sim e^{\nu(\theta-\Theta)} \rho_\D^{\Theta-2} $ on $\D_\nu \supset \supp(\zeta_\nu)$  for all $\nu\in\Z$.
	Hence, we can conclude that
	\begin{align*}
		\norm{ u  \sep H^{\gamma}_{p,\Theta,\theta}(\D)}^p
		&\sim \sum_{\nu\in\Z} e^{\nu\theta} \norm{(\zeta_{\nu}\,u)(e^\nu\cdot) \sep H^{\gamma}_{p,\Theta}(\D)}^p \\
		&\sim \sum_{\nu\in\Z} \sum_{\abs{\alpha}\leq \gamma} \norm{ \rho_\D^{\abs{\alpha}} \, D^\alpha (\zeta_\nu\, u) \sep L_{p,\Theta,\theta}(\D_\nu)}^p 
		\sim 
		\normmod{u \sep H^{\gamma}_{p,\Theta,\theta}(\D)}^p,
	\end{align*}
	where we used that according to \autoref{rem:resolution}\ref{it:resolution:cover}   our set $\D$ is covered by $\bigcup_{\nu\in\Z} \D_\nu$, where every $x\in\D$ belongs to at most 2 different $\D_\nu$.
\end{proof}
 
As a consequence we obtain the following relationship between the spaces $H^\gamma_{p,\Theta,\theta}(\D)$ and $H^\gamma_{p,\Theta}(\D)$.

\begin{corollary}
\label{cor:Lot:doubleweight}
    Let $1<p<\infty$ and $\gamma,\Theta,\theta\in\R$. 
      For $u\in\mathscr{D}'(\D)$ it  holds  
    $u\in H^\gamma_{p,\Theta,\theta}(\D)$ if, and only if, $\rho_\circ^{(\theta-\Theta)/p}u \in H^\gamma_{p,\Theta}(\D)$.
    Moreover,
    \begin{equation}
        H^\gamma_{p,\Theta,\theta}(\D)
    \ni
    u \mapsto \norm{\rho_\circ^{(\theta-\Theta)/p}u \sep H^\gamma_{p,\Theta}(\D)}
    \end{equation}
is an equivalent norm in $H^\gamma_{p,\Theta,\theta}(\D)$.
In particular, we have $H^\gamma_{p,\Theta,\Theta}(\D)=H^\gamma_{p,\Theta}(\D)$.
\end{corollary}

\begin{proof}
For $\gamma\in\N_0$ the assertion 
  immediately follows  
from~\autoref{thm:coincidence},  \autoref{lem:LotSpaces:properties}\ref{it:LotSpaces:N}, and Leibniz' rule, since for all $\alpha\in\N_0^2$ and $t\in\R$ it holds that
\begin{equation}\label{eq:rhocirc:deriv}
\abs{(\partial^{\alpha} \rho_\circ^t )(x)}
\lesssim \rho_\circ(x)^{t-\abs{\alpha}}, 
\quad x\in\D,
\end{equation}
by the generalized Fa\`a di Bruno formula~\cite[Corollary~2.10]{ConSav1996}.
Then the assertion for all $\gamma\geq 0$ follows by interpolation, see \autoref{thm:interpolation} and \autoref{lem:LotSpaces:properties}\ref{it:LotSpaces:interpolation}, and extends to $\gamma<0$ by means of the duality statements from \autoref{thm:dual} and \autoref{lem:LotSpaces:properties}\ref{it:LotSpaces:duality} with similar arguments as in the proof of~\autoref{thm:pw-mult_new} below.
\end{proof}

\begin{remark}
Corollary~\ref{cor:Lot:doubleweight} 
  particularly shows  
that for all $1<p<\infty$ and all $\gamma,\Theta,\theta\in\R$ the space $H^\gamma_{p,\Theta,\theta}(\D)$ coincides with the space $K^\gamma_{p,\theta,\Theta}(\D)$ introduced in \cite[Definition~2.4]{KimLeeSeo2022b} (equivalent norms). 
\end{remark}

Let us also mention the following characterization of the duals for $\gamma\in\N_0$ which generalizes a well-known result for classical (unweighted) Sobolev spaces; see, e.g., \cite[Theorem~3.12]{AdaFou2003}. 

\begin{prop}\label{prop:dual:N}
Assume that $\gamma\in\N_0$, $1<p,p'<\infty$, and $\Theta,\Theta',\theta,\theta'\in\R$ satisfy~\eqref{eq:dual:parameters}.
    Then
    \begin{align*}
        H^{-\gamma}_{p',\Theta', \theta'}(\D)
        = \bigg\{ u\in\mathscr{D}'(\D) : \ u=\sum_{\abs{\alpha}\leq \gamma}D^{\alpha} u_{\alpha}  \; \text{ for some } \; u_{\alpha}\in L_{p',\Theta'-\abs{\alpha}p', \theta'-\abs{\alpha}p'}(\D) \bigg\}
    \end{align*}
    and
    \begin{align*}
        \norm{ u \sep H^{-\gamma}_{p',\Theta', \theta'}(\D) }
        \sim \min \bigg( \sum_{\abs{\alpha}\leq \gamma} \norm{ u_{\alpha} \sep L_{p',\Theta'-\abs{\alpha}p', \theta'-\abs{\alpha}p'}(\D)}^{p'} \bigg)^{1/p'},
    \quad u \in H^{-\gamma}_{p',\Theta', \theta'}(\D),
    \end{align*}
    where the minimum is taken over all possible representations of $u\in\mathscr{D}'(\D)$ as $\sum_{\abs{\alpha}\leq \gamma}D^{\alpha}u_{\alpha}$ with $u_{\alpha}\in  L_{p',\Theta'-\abs{\alpha}p', \theta'-\abs{\alpha}p'}(\D)$, $\abs{\alpha}\leq \gamma$. 
\end{prop}
\begin{proof}
In view of \autoref{thm:coincidence} and the density statement in \autoref{thm:density}, the claim for $\gamma=0$ is a simple corollary of \autoref{thm:dual} (see also~\cite[Lemma~B.2]{Cio20}), while for $\gamma=1$ it coincides with~\cite[Lemma~B.4]{Cio20}.
The proof of the latter extends mutatis mutandis to arbitrary $\gamma\in\N$ and is left to the reader.
\end{proof}

\subsection{Pointwise multiplication}
\label{sec:Multipliers}

Using \autoref{cor:Lot:doubleweight} we can transfer some of the properties of the spaces $H^\gamma_{p,\Theta}(\D)$ mutatis mutandis to $H^\gamma_{p,\Theta,\theta}(\D)$. 
Especially, we obtain the following assertions on pointwise multipliers and weight index shifts. 
Recall that for domains $\domain\subset\R^d$ and closed sets $M\subset\partial\domain$, we write $\Regd(\domain,M)$ to denote the set of regularized distances to $M$ on $\domain$, see  \autoref{rem:resolution}\ref{it:resolution:construction}.

\begin{corollary}
\label{cor:pwmult:indexshift}
    Let $1<p<\infty$ and $\gamma,\Theta,\theta\in\R$.
    \begin{enumerate}[label=\textup{(\roman*)}]
    \item\label{it:pwmult:simple}  If $a\colon\D\to\R$ is a pointwise multiplier for $H^\gamma_{p,\Theta}(\D)$, then it is a pointwise multiplier for $H^\gamma_{p,\Theta,\theta}(\D)$, too. 
    In particular, if 
    $\abs{a}^{(0)}_n:=\sup_{x\in\domain}\sum_{\abs{\alpha}\leq n}\rho_\D^{\abs{\alpha}}(x) \abs{D^\alpha a(x)}<\infty$ for some $n\in\N_0$, then for all $\abs{\gamma}\leq n$ we have
    \begin{align*}
        \norm{a \, u\sep H^{\gamma}_{p,\Theta,\theta}(\D)} 
        \leq C(d,p,n) \abs{a}^{(0)}_n \norm{ u\sep H^{\gamma}_{p,\Theta,\theta}(\D)},
        \qquad u\in H^{\gamma}_{p,\Theta,\theta}(\D).
    \end{align*}

    \item\label{it:indexshift} Let $\psi_\D\in \Regd(\D,\partial\D)$ and $\psi_\circ\in\Regd(\D,\{0\})$. 
    Further, let $s,t\in\R$ and $u\in\mathscr{D}'(\D)$. 
    Then $\psi_\D^s \, \psi_\circ^t\, u \in H^\gamma_{p,\Theta p,\theta p}(\D)$ if, and only if, $u \in H^\gamma_{p,(\Theta+s) p,(\theta+s+t) p}(\D)$. 
    In this case, 
    \[
        \norm{\psi_\D^s\, \psi_\circ^t\, u \sep H^\gamma_{p,\Theta p,\theta p}(\D)} 
        \sim \norm{u \sep H^\gamma_{p,(\Theta+s) p, (\theta+s+t) p}(\D)}, 
        \qquad u\in H^\gamma_{p,(\Theta+s) p,(\theta+s+t) p}(\D).
    \]
    \end{enumerate}
\end{corollary}

\begin{proof}
Both assertions immediately follow from \autoref{cor:Lot:doubleweight} together with \autoref{lem:LotSpaces:properties}\ref{it:LotSpaces:multiplier} (for Assertion~\ref{it:pwmult:simple}) and \autoref{lem:LotSpaces:properties}\ref{it:LotSpaces:indexshift} (for Assertion~\ref{it:indexshift}), respectively.
\end{proof}

The following statement complements~\autoref{cor:pwmult:indexshift}.

\begin{theorem}
\label{thm:pw-mult_new}
    Let $1<p_1< p_0<\infty$ and $\gamma,\Theta_0,\Theta_1,\theta_0,\theta_1\in\R$.
    Further assume that $a\in H^{m}_{q,(\Theta_1-\Theta_0)q,(\theta_1-\theta_0)q}(\D)$ with 
    \begin{align}\label{eq:cond_pw}
        m:=\ceil{\abs{\gamma}}
    \qquad\text{and}\qquad
    \frac{1}{q} := \frac{1}{p_1} - \frac{1}{p_0}.
    \end{align}
    Then $M_a \colon u \mapsto M_a(u):= a u$ maps $H^{\gamma}_{p_0,\Theta_0 p_0, \theta_0 p_0}(\D)$ into $H^{\gamma}_{p_1,\Theta_1 p_1,\theta_1 p_1}(\D)$ and
    \begin{align}\label{eq:pw-mult_new2}
        \norm{M_a \sep \mathcal{L}\big( H^{\gamma}_{p_0,\Theta_0 p_0, \theta_0 p_0}(\D), H^{\gamma}_{p_1,\Theta_1 p_1, \theta_1 p_1}(\D)\big)} 
        \lesssim \norm{a \sep H^{m}_{q, (\Theta_1-\Theta_0)q,(\theta_1-\theta_0)q }(\D)}.
    \end{align}
\end{theorem}

\begin{remark}\label{rem:pw} 
Together with \autoref{thm:coincidence} we see that 
\begin{align*}
\norm{a \sep H^{m}_{q, (\Theta_1-\Theta_0)q,(\theta_1-\theta_0)q}(\D)} 
&= \sum_{\abs{\beta}\leq m} \norm{\rho_\D^{\abs{\beta}} \, D^\beta a \sep L_{q, (\Theta_1-\Theta_0)q,(\theta_1-\theta_0)q}(\D)}\\
&= \sum_{\abs{\beta}\leq m} \norm{\rho_{\circ}^{\theta_1-\theta_0 - (\Theta_1-\Theta_0)}\, \rho_\D^{\Theta_1 - \Theta_0 - \frac{2}{q}+\abs{\beta}}  \, D^\beta a \sep L_q(\D)}. 
\end{align*}
If we formally put $q=\infty$, \autoref{thm:pw-mult_new} corresponds to the limiting case
 $p_1=p_0=:p$  in~\eqref{eq:cond_pw}, treated in  \autoref{cor:pwmult:indexshift}. 
\end{remark}

\begin{proof}[Proof of \autoref{thm:pw-mult_new}]
\emph{Step 1. } 
Let $\gamma\in\N_0$. 
Then \autoref{thm:coincidence} and direct computation using Leibniz' rule yields
\begin{align*}
		\norm{M_a(u) \sep H^{\gamma}_{p_1,\Theta_1 p_1, \theta_1 p_1}(\D)} 
		&\sim \sum_{\abs{\alpha}\leq\gamma} \norm{\rho_\D^{\abs{\alpha}}\, D^\alpha (a u) \sep L_{p_1,\Theta_1 p_1,\theta_1 p_1}(\D)}  \\
		&\lesssim \sum_{\abs{\alpha}\leq\gamma} \sum_{\beta\leq \alpha} \norm{\rho_\D^{\abs{\alpha-\beta}}\, \rho_\D^{\abs{\beta}} \, (D^\beta a) \, (D^{\alpha-\beta} u) \sep L_{p_1,\Theta_1 p_1,\theta_1 p_1}(\D)}.
\end{align*}
Next we can estimate each summand using the generalized H\"older inequality to obtain
\begin{align*}
		&\norm{\rho_\D^{\abs{\alpha-\beta}}\, \rho_\D^{\abs{\beta}} \, (D^\beta a) \, (D^{\alpha-\beta} u) \sep L_{p_1,\Theta_1 p_1,\theta_1 p_1}(\D)} \\
		&\qquad = \norm{\rho_\circ^{ \theta_1-\theta_0 - (\Theta_1-\Theta_0)} \, \rho_\D^{ \Theta_1 - \Theta_0 - 2/q+\abs{\beta}} \, (D^\beta a) \, \rho_\circ^{\theta_0-\Theta_0} \, \rho_\D^{\Theta_0-2/p_0+\abs{\alpha-\beta}}\,(D^{\alpha-\beta} u) \sep L_{p_1}(\D)} \\
		&\qquad \leq \norm{\rho_\circ^{\theta_1-\theta_0 - (\Theta_1-\Theta_0)} \, \rho_\D^{ \Theta_1 - \Theta_0 - 2/q +\abs{\beta}} \, D^\beta a \sep L_{q}(\D)} \, \norm{\rho_\circ^{\theta_0-\Theta_0} \, \rho_\D^{\Theta_0-2/p_0+\abs{\alpha-\beta}}\,D^{\alpha-\beta} u \sep L_{p_0}(\D)},
\end{align*}
since $p_0,p_1<\infty$. Therefore,
\begin{align*}
		&\norm{M_a(u) \sep H^{\gamma}_{p_1,\Theta_1 p_1, \theta_1 p_1}(\D)} \\
		&\quad \lesssim \sum_{\abs{\alpha}\leq\gamma} \sum_{\beta\leq \alpha} \norm{\rho_\circ^{ \theta_1-\theta_0 - (\Theta_1-\Theta_0)} \, \rho_\D^{\Theta_1 - \Theta_0 - 2/q +\abs{\beta}} \, D^\beta a \sep L_{q}(\D)} \, \norm{\rho_\D^{\abs{\alpha-\beta}}\, D^{\alpha-\beta} u \sep L_{p_0,\Theta_0 p_0, \theta_0 p_0}(\D)} \\
		&\quad \lesssim \norm{a \sep H_{q, (\Theta_1-\Theta_0)q,(\theta_1-\theta_0)q}^{\gamma}(\D)} \, \norm{u \sep H^{\gamma}_{p_0,\Theta_0 p_0, \theta_0 p_0}(\D)}
\end{align*}
with constants independent of $u$ and $a$.
Thus, \eqref{eq:pw-mult_new2} holds for all $\gamma\in\N_0$.

\emph{Step 2. } 
Let $\gamma \in[0,\infty)\setminus\N_0$ and assume $a\in H^{\ceil{\gamma}}_{q, (\Theta_1-\Theta_0)q,(\theta_1-\theta_0)q }(\D)$. 
Then, by Step~1, $M_a$ belongs to $\L \big(H^{\floor{\gamma}}_{p_0,\Theta_0 p_0,\theta_0 p_0}(\D),H^{\floor{\gamma}}_{p_1,\Theta_1 p_1, \theta_1 p_1}(\D)\big)$ and to 
$\L \big(H^{\ceil{\gamma}}_{p_0,\Theta_0 p_0,\theta_0 p_0}(\D),H^{\ceil{\gamma}}_{p_1,\Theta_1 p_1, \theta_1 p_1}(\D)\big)$ and in both cases its norm is bounded from above by $ \norm{a \sep H^{\ceil{\gamma}}_{q, (\Theta_1-\Theta_0)q,(\theta_1-\theta_0)q}(\D)}$ times a finite constant that does not depend on $a$. 
Due to \autoref{thm:interpolation} we know that $H^{\gamma}_{p_i,\Theta_i p_i, \theta_i p_i}(\D)$ can be written as complex interpolation space of $H^{\floor{\gamma}}_{p_i,\Theta_i p_i, \theta_i p_i}(\D)$ and $H^{\ceil{\gamma}}_{p_i,\Theta_i p_i, \theta_i p_i}(\D)$ for $i=0,1$. 
Thus, due to the interpolation property, \eqref{eq:pw-mult_new2} holds for all $\gamma\in[0,\infty)$.

\emph{Step 3. } 
Let $\gamma < 0$ and $a\in H^{\ceil{\abs{\gamma}}}_{q,(\Theta_1-\Theta_0)q,(\theta_1-\theta_0)q}(\D)$. 
Due to \autoref{thm:dual} we have $H^\gamma_{p_i,\Theta_i p_i,\theta_i p_i}(\D) = \big( H^{-\gamma}_{p_i',(\Theta_i p_i)',(\theta_i p_i)'}(\D) \big)'$, where for $i=0,1$ there holds $1/p_i+1/p_i'=1$ as well as $(\Theta_i p_i) /p_i + (\Theta_i p_i)'/p_i'=2$. That is, $\Theta_i = 2 - (\Theta_i p_i)'/p_i'$, and likewise for $\Theta_i$ replaced by $\theta_i$, $i=0,1$. 
Hence, $1<p_0'< p_1'<\infty$, 
\begin{align*}
\frac{1}{q} = \frac{1}{p_1}-\frac{1}{p_0} 
= \frac{1}{p_0'} - \frac{1}{p_1'}
\quad \text{and} \quad
\Theta_1-\Theta_0
=
\frac{(\Theta_0p_0)'}{p_0'}
-
\frac{(\Theta_1 p_1)'}{p_1'}
\end{align*}
as well as
\begin{align*}
\theta_1-\theta_0-(\Theta_1-\Theta_0) 
=
\frac{(\theta_0p_0)'}{p_0'}
-
\frac{(\theta_1p_1)'}{p_1'}
-
\bigg(
\frac{(\Theta_0p_0)'}{p_0'}
-
\frac{(\Theta_1p_1)'}{p_1'}
\bigg). 
\end{align*}
Thus, our previous steps imply $M_a \in \L\big(H^{-\gamma}_{p_1',(\Theta_1 p_1)',(\theta_1 p_1)'}(\D), H^{-\gamma}_{p_0',(\Theta_0 p_0)',(\theta_0 p_0)'}(\D) \big)$ such that
	\begin{align*}
		&\norm{a u \sep H^\gamma_{p_1,\Theta_1 p_1,\theta_1 p_1}(\D)} \\
		&\quad\sim \norm{a u \sep \big( H^{-\gamma}_{p_1',(\Theta_1 p_1)',(\theta_1 p_1)'}(\D) \big)'} \\
		&\quad= \sup\!\left\{ \abs{au(\varphi)}: \  \norm{\varphi\sep H^{-\gamma}_{p_1',(\Theta_1 p_1)',(\theta_1 p_1)'}(\D)}=1 \right\}\\
		&\quad= \sup\!\left\{ \abs{u(a\varphi)}: \  \norm{\varphi\sep H^{-\gamma}_{p_1',(\Theta_1 p_1)',(\theta_1 p_1)'}(\D)}=1 \right\}\\
		&\quad\leq \sup\!\bigg\{ \norm{u \sep H^\gamma_{p_0,\Theta_0 p_0, \theta_0 p_0}(\D)} 
		\norm{a\varphi\sep H^{-\gamma}_{p_0',(\Theta_0 p_0)',(\theta_0 p_0)'}(\D)}: \ \norm{\varphi\sep H^{-\gamma}_{p_1',(\Theta_1 p_1)',(\theta_1 p_1)'}(\D)}=1 \bigg\}  \\
		&\quad\lesssim \norm{a \sep H^{\ceil{\abs{\gamma}}}_{q, (\Theta_1-\Theta_0)q,(\theta_1-\theta_0)q}(\D)} \, \norm{u \sep H^\gamma_{p_0,\Theta_0 p_0,\theta_0 p_0}(\D)}, 
	\end{align*}
 where in the third step we translated the	multiplication of the two distributions $a$ and $u$ into the dual action between them which is well-defined due to our assumptions (i.e., $u\in H^\gamma_{p_0,\Theta_0 p_0, \theta_0 p_0}(\D)$ and $a\varphi\in H^{-\gamma}_{p_0',(\Theta_0 p_0)',(\theta_0 p_0)'}(\D)$). 	This completes the proof. 
\end{proof}

\subsection{Lifting}\label{sec:Lifting}

The following lifting result is quite useful in many calculations, in particular, in the context of PDEs. It generalizes the well-known lifting properties of derivatives along scales of classical Sobolev and Bessel potential spaces. In particular, it generalizes \autoref{lem:LotSpaces:properties}\ref{it:LotSpaces:lifting} for $\domain=\D$. 
 
\begin{theorem}\label{thm:lifting}
Let $1<p<\infty$ and $\gamma,\Theta,\theta\in\R$ as well as $m\in\N$. Moreover, let $\psi_\D\in \Regd(\D,\partial\D)$, see \autoref{rem:resolution}\ref{it:resolution:construction}. 
Then the following assertions are equivalent for $u\in\mathscr{D}'(\D)$.
\begin{enumerate}[label=\textup{(\alph*)}]
\item\label{it:lift:1} $u\in H^{\gamma}_{p,\Theta,\theta}(\D)$.
\item\label{it:lift:D} $D^{\alpha} u \in H^{\gamma-m}_{p,\Theta+\abs{\alpha}p,\theta+\abs{\alpha}p}(\D)$ for all $\alpha \in \N_0^2$ with $\abs{\alpha}\leq m$.
\item\label{it:lift:PsiD} $\psi_\D^{\abs{\alpha}} D^{\alpha} u \in H^{\gamma-m}_{p,\Theta,\theta}(\D)$ for all $\alpha \in \N_0^2$ with $\abs{\alpha}\leq m$.
\item\label{it:lift:DPsi} $D^{\alpha} \big( \psi_\D^{\abs{\alpha}} u\big) \in H^{\gamma-m}_{p,\Theta,\theta}(\D)$ for all $\alpha \in \N_0^2$ with $\abs{\alpha}\leq m$.
\end{enumerate}
	In this case,
	\begin{align*}  
		\norm{u \sep H^{\gamma}_{p,\Theta,\theta}(\D)} 
		&\sim \sum_{\abs{\alpha}\leq m} \norm{D^\alpha u \sep H^{\gamma-m}_{p,\Theta+\abs{\alpha}p,\theta+\abs{\alpha}p}(\D)} 
		\\
		&\sim \sum_{\abs{\alpha}\leq m} \norm{\psi_\D^{\abs{\alpha}} D^\alpha u \sep H^{\gamma-m}_{p,\Theta,\theta}(\D)} 
		\\
		&\sim \sum_{\abs{\alpha}\leq m} \norm{ D^{\alpha} \big( \psi_\D^{\abs{\alpha}} u\big) \sep H^{\gamma-m}_{p,\Theta,\theta}(\D)} 
	\end{align*}
	with constants independent of $u$.
\end{theorem}

\begin{remark}
    Some comments are in order.
	\begin{enumerate}[label=(\textup{\roman*)}]
		\item The implication ``\ref{it:lift:1} $\Longrightarrow$ \ref{it:lift:D}'' in \autoref{thm:lifting} and the corresponding estimate especially imply that 
$D^\alpha\colon H^\gamma_{p,\Theta,\theta}(\D)\to H^{\gamma-\abs{\alpha}}_{p,\Theta+\abs{\alpha}p,\theta+\abs{\alpha}p}(\D)$ are bounded linear operators for each $\alpha\in\N_0^2$.  
		This has already been proven in~\cite[Lemma~2.5(v)]{KimLeeSeo2022b}. 
		\item 
		  For $m:=\gamma\in\N$,  
		assertion~\ref{it:lift:D} agrees with the representation of $H^{\gamma}_{p,\Theta,\theta}(\D)$ as weighted Sobolev space proven in \autoref{thm:coincidence} above. Furthermore,
		Condition~\ref{it:lift:DPsi} and the corresponding norm equivalence yield a similar representation: $H^m_{p,\Theta,\theta}(\D)$ consists of all $u\in L_{1,\loc}(\D)$ such that
		\[
		\sum_{\abs{\alpha}\leq m} \norm{ D^{\alpha} \big( \psi_\D^{\abs{\alpha}} u\big) \sep L_{p,\Theta,\theta}(\D)}^p<\infty.
		\]
Moreover,
		\[
		\norm{u \sep H^m_{p,\Theta,\theta}(\D)}
		\sim \ssgrklam{\sum_{\abs{\alpha}\leq m} \norm{ D^{\alpha} \big( \psi_\D^{\abs{\alpha}} u\big) \sep L_{p,\Theta,\theta}(\D)}^p }^{1/p},
		\qquad u\in H^m_{p,\Theta,\theta}(\D).
		\]
	\end{enumerate}
\end{remark}

\begin{proof}[Proof of \autoref{thm:lifting}]
Note that ``\ref{it:lift:D} $\Longleftrightarrow$ \ref{it:lift:PsiD}'' and the corresponding norm equivalence hold for all $m\in\N$ by \autoref{cor:pwmult:indexshift}\ref{it:indexshift}. To prove the other equivalences we argue as follows.

\emph{Step 1. }
We first prove ``\ref{it:lift:1} $\Longleftrightarrow$ \ref{it:lift:PsiD}'' and the corresponding norm equivalence by induction on $m\in\N$. We start with the base case $m=1$. Let $\zeta:=\rho_\circ^{(\theta-\Theta)/p}$. Then, by \autoref{cor:Lot:doubleweight}, $u\in H^\gamma_{p,\Theta,\theta}(\D)$ if, and only if, $\zeta u\in H^\gamma_{p,\Theta}(\D)$. 
Moreover, 
\begin{equation}\label{eq:norm:equiv:lifting:1}
    \norm{u \sep H^\gamma_{p,\Theta,\theta}(\D)}\sim \norm{\zeta u \sep H^\gamma_{p,\Theta}(\D)},
    \qquad u\in H^\gamma_{p,\Theta,\theta}(\D).
\end{equation}
By \autoref{lem:LotSpaces:properties}\ref{it:LotSpaces:lifting}, 
  we have  
$\zeta u\in H^\gamma_{p,\Theta}(\D)$ if, and only if, $\zeta u \in H^{\gamma-1}_{p,\Theta}(\D)$ and $\psi_\D D(\zeta u)\in H^{\gamma-1}_{p,\Theta}(\D)$. Moreover, taking into account~\eqref{eq:norm:equiv:lifting:1},
\begin{align*}
\norm{u \sep H^\gamma_{p,\Theta,\theta}(\D)}
&\sim
\norm{\zeta u \sep H^\gamma_{p,\Theta}(\D)}\\
&\sim
\norm{\zeta u \sep H^{\gamma-1}_{p,\Theta}(\D)}
+
\norm{\psi_\D D(\zeta u) \sep H^{\gamma-1}_{p,\Theta}(\D)}, 
\qquad u\in H^\gamma_{p,\Theta,\theta}(\D).
\end{align*}
Now integration   
by parts shows that
\[
\psi_\D D(\zeta u) 
= 
\zeta \psi_\D Du + \psi_\D D(\zeta) u
=
\zeta \psi_\D Du + \psi_\D \zeta^{-1} D(\zeta) \zeta u.
\]
Thus, the base case follows by means of \autoref{cor:Lot:doubleweight}\ref{it:pwmult:simple} from the fact that
\[
\abs{\psi_\D \zeta^{-1} D(\zeta)}_n^{(0)}<\infty \quad\text{for all } n\in\N
\]
which may be verified by using~\eqref{eq:rhocirc:deriv} and the fact that $\psi_\D\in\Regd(\D,\partial\D)$.

We move on to the induction step. Our induction hypothesis is that ``\ref{it:lift:1} $\Longleftrightarrow$ \ref{it:lift:PsiD}'' holds with $m=k$ for some $k\in\N$. This hypothesis and the base case show that $u\in H^\gamma_{p,\Theta,\theta}(\D)$ if, and only if, $\psi_\D^{\abs{\alpha}}D^\alpha u \in H^{\gamma-(k+1)}_{p,\Theta,\theta}(\D)$ and $\psi_\D^{\abs{\alpha}} D^\alpha(\psi_\D Du)\in H^{\gamma-(k+1)}_{p,\Theta,\theta}(\D)$ for all $\alpha\in\N_0^2$ with $\abs{\alpha}\leq m$ (and appropriate norm equivalences hold).
Thus, since by Leibniz' rule,
\[
    \psi_\D^{\abs{\alpha}} D^\alpha(\psi_\D Du)
    =
    \psi^{\abs{\alpha}+1} D^\alpha Du
    +
    \sum_{\beta<\alpha} 
    \binom{\alpha}{\beta}
    \psi^{\abs{\alpha-\beta}-1} D^{\alpha-\beta}(\psi)\cdot \psi^{\abs{\beta}+1}D^\beta Du
\]
and
\begin{equation}\label{eq:lifting:multiplier:2}
    \abs{\psi^{\abs{\alpha-\beta}-1} D^{\alpha-\beta}(\psi)}_n^{(0)}<\infty \quad\text{for all } n\in\N, 
\end{equation}
straightforward calculations show that ``\ref{it:lift:1} $\Longleftrightarrow$ \ref{it:lift:PsiD}'' and the corresponding norm equivalence hold for $m=k+1$. The fact that~\eqref{eq:lifting:multiplier:2} holds 
  is   verified by using $\psi_\D\in\Regd(\D,\partial\D)$ after suitable applications of the Leibniz rule and the generalized Fa\`a di Bruno formula~\cite[Corollary~2.10]{ConSav1996}.

\emph{Step 2. }
Finally, 
note that for all $\alpha\in\N_0^2$, by the Leibniz rule, 
it  holds  
\[
D^\alpha(\psi_\D^{\abs{\alpha}}u)
=
\psi_\D^{\abs{\alpha}} D^\alpha u
+
\sum_{\beta<\alpha} \binom{\alpha}{\beta} \psi^{-\abs{\beta}} D^{\alpha-\beta}(\psi_\D^{\abs{\alpha}})
\psi_\D^{\abs{\beta}}D^\beta u
\]
as well as  
\[
\abs{\psi^{-\abs{\beta}} D^{\alpha-\beta}(\psi_\D^{\abs{\alpha}})}_n^{(0)}<\infty
\quad\text{for all }n\in\N;
\]
the latter may be verified 
using Fa\`a di Bruno's formula  
and the fact that $\psi_\D\in\Regd(\D,\partial\D)$.
Thus, for all $m\in\N$, the equivalence ``\ref{it:lift:DPsi} $\Longleftrightarrow$ \ref{it:lift:PsiD}'' and the corresponding norm equivalence follow by standard arguments and \autoref{cor:pwmult:indexshift}\ref{it:pwmult:simple}.
\end{proof}
	
\subsection{Localization}
Let us now deduce a  
localization result which generalizes \autoref{prop:Lot_Local} for $\domain=\D$. 
\begin{theorem}
	\label{thm:Local}
	Let $1<p<\infty$ and $\gamma,\Theta,\theta\in\R$.
	Further let $\eta=(\eta_k)_{k\in\N}$ denote a collection of $C^\infty(\D)$-functions such that 
	$$
	\sup_{x\in\D} \sum_{k\in\N} \rho_\D(x)^{\abs{\alpha}} \abs{\partial^\alpha \eta_k(x)} 	\leq C_\alpha, \qquad \alpha\in\N_0^2.
	$$
	Then
	$$
	\sum_{k\in\N} \norm{\eta_k u \sep H^\gamma_{p,\Theta,\theta}(\D)}^p 
	\lesssim \norm{u \sep H^\gamma_{p,\Theta,\theta}(\D)}^p, \qquad u\in H^\gamma_{p,\Theta,\theta}(\D).
	$$
	If, in addition,
	$$
	\inf_{x\in\D} \sum_{k\in\N} \abs{\eta_k(x)}^p \geq \delta > 0,
	$$
	then 
	$$
	\norm{u \sep H^\gamma_{p,\Theta,\theta}(\D)}^p
	\lesssim \sum_{k\in\N} \norm{\eta_k u \sep H^\gamma_{p,\Theta,\theta}(\D)}^p, \qquad u\in H^\gamma_{p,\Theta,\theta}(\D).
	$$
\end{theorem}
\begin{remark}
Note that in \autoref{thm:Local} we can replace $\N$ by any countable set. 
Moreover,
  let us stress that  
any $\eta:=(\eta_{k})_{k\in\Z}\in \mathscr{A}_{c}^{\textup{[L]}}(\D,\{0\})$ with  $c>1$ satisfies all assumptions from~\autoref{thm:Local} (with $\N$ replaced by $\Z$) and that
	\begin{align*}
    	\norm{u \sep H^\gamma_{p,\Theta,\theta}(\D)}^p
    	&\sim \sum_{k\in\Z} \norm{\eta_k u \sep H^\gamma_{p,\Theta,\theta}(\D)}^p 
    	\sim \sum_{k\in\Z} c^{k(\theta-\Theta)} \norm{\eta_k u \sep H^\gamma_{p,\Theta}(\D)}^p,  
     \quad u\in H^\gamma_{p,\Theta,\theta}(\D),
	\end{align*}
	see \autoref{rem:resolution}\ref{it:resolution:cover}   and \autoref{cor:Lot:doubleweight}. 
\end{remark}

\begin{proof}[Proof of \autoref{thm:Local}]
The proof is a direct consequence of \autoref{cor:Lot:doubleweight} and the corresponding localization assertion for the spaces $H^{\gamma}_{p,\Theta}(\domain)$ from  \autoref{prop:Lot_Local}. 
In particular,
\begin{align*}
    \sum_{k\in \mathbb{Z}} \norm{\eta_k u \sep H^{\gamma}_{p,\Theta,\theta}(\D)}^p
    &\sim \sum_{k\in \mathbb{Z}}\norm{\rho_\circ^{(\theta-\Theta)/p}\eta_k u \sep H^\gamma_{p,\Theta}(\D)}^p\\
    &\lesssim \norm{\rho_\circ^{(\theta-\Theta)/p} u \sep H^\gamma_{p,\Theta}(\D)}^p \\
    &\sim \norm{ u \sep H^{\gamma}_{p,\Theta,\theta}(\D) }^p
    ,  \qquad u\in H^\gamma_{p,\Theta,\theta}(\D).  
\end{align*}
Similar arguments for the second assertion complete the proof.
\end{proof}

\subsection{Embeddings and Hölder regularity}

In terms of embeddings of $H^\gamma_{p,\Theta,\theta}(\D)$ into spaces with smaller integrability, the following can be shown. 

\begin{theorem}\label{thm:right-embedding}
	Let $1<p_1 \leq p_0 < \infty$ as well as $\gamma,\Theta_0,\Theta_1,\theta_0,\theta_1\in\R$
	such that
	\begin{align}\label{eq:cond-emb}
		\Theta_1 - \frac{1}{p_1} > \Theta_0 - \frac{1}{p_0}
		\qquad \text{and} \qquad \theta_1>\theta_0.
	\end{align}
	Moreover, let $z\in\R^2$ and $R > 0$. 
	Then for all $u\in H^\gamma_{p_0,\Theta_0 p_0,\theta_0 p_0}(\D)$ with $\supp(u)\subset B_R(z)$ we have $u\in H^\gamma_{p_1,\Theta_1 p_1,\theta_1 p_1}(\D)$
	and
	$$
		\norm{u \sep  H^\gamma_{p_1,\Theta_1 p_1,\theta_1 p_1}(\D)} \lesssim \norm{u \sep  H^\gamma_{p_0,\Theta_0 p_0,\theta_0 p_0}(\D)}.
	$$
\end{theorem}
\begin{proof}
	W.l.o.g.\ we can assume that $B_R(z)\cap \D \neq\emptyset$ as well as $p_1<p_0$, as otherwise the assertion directly follows from the monotonicity statements w.r.t.\ the weight parameters proven in \autoref{lem:monotone} and \autoref{lem:monotone2}\ref{it:mon1},
	  respectively.  
	We choose $\epsilon>0$ and $\eta\in C^\infty(\R^d)$ such that $\eta\equiv 1$ on $B_R(z)$ and $\supp(\eta)\subset B_{R+\epsilon}(z)$. 
	If we can show that \autoref{thm:pw-mult_new} applies to $a:=\eta$, then this proves the claim since $M_\eta (u) = \eta u$ equals $u$ in $\mathscr{D}'(\D)$ as $\eta\equiv 1$ on $\supp(u)$.
	For this purpose, select $0\leq R_0 < R_1<\infty$ and $0\leq \Gamma < \Upsilon \leq \kappa_0$ such that
	$$
		B_{R+\epsilon}(z) \cap \D  \subset S:=\left\{x=\Phi(r,\phi)\in\D : \  (r,\phi)\in (R_0,R_1)\times (\Gamma,\Upsilon) \right\}.
	$$
	Due to \autoref{rem:pw} it then	suffices to check that for $1/q := 1/p_1-1/p_0$ and $\sigma := \Theta_1-\Theta_0-2/q$, as well as $\mu := (\theta_1-\theta_0)-(\Theta_1-\Theta_0)$ and arbitrarily fixed $m\in\N_0$ the expression
	$$
	\norm{\eta \sep H^m_{q,(\Theta_1-\Theta_0)q,(\theta_1-\theta_0)q}(\D)} 
    	= \sum_{\abs{\beta}\leq m} \norm{\rho_\circ^{\mu}\, \rho_\D^{\sigma+\abs{\beta}} \, D^\beta \eta \sep L_{q}(\D)} 
	    \lesssim \norm{\rho_\circ^{\mu}\,\rho_\D^{\sigma} \sep L_q(S)} 
	$$
	is finite. To see this, note that
	\autoref{lem:dist-polar} implies
	\begin{align*}
		\rho_\circ(x)^{\mu}\,\rho_\D(x)^{\sigma} 
		&= \rho_\circ(x)^{\theta_1-\theta_0}\, \left(\frac{\rho_\D(x)}{\rho_\circ(x)} \right)^{\Theta_1-\Theta_0} \, \rho_\D(x)^{-2/q} \\
		&\sim r^{\theta_1-\theta_0}\, \rho_\I(\phi)^{\Theta_1-\Theta_0} \, (r\, \rho_\I(\phi))^{-2/q}, \qquad x=\Phi(r,\phi)\in \D,
	\end{align*}
	where as usual $\I:=(0,\kappa_0)$.
	Hence, if $q=\infty$, i.e., $p_0=p_1$, then \eqref{eq:cond-emb} implies that $\rho_\circ^{\mu}\,\rho_\D^{\sigma}$ stays bounded on $S$. 
	If otherwise $q<\infty$, then \eqref{eq:cond-emb} also yields the finiteness of
	\begin{align*}
		\norm{\rho_\circ^{\mu}\,\rho_\D^{\sigma} \sep L_q(S)}^q 
		&\sim \int_{R_0}^{R_1} \int_{\Gamma}^{\Upsilon} \abs{r^{\theta_1-\theta_0-2/q}\, \rho_\I(\phi)^{\Theta_1-\Theta_0-2/q} }^q \, \d\phi\, r \d r  \\
		&= \int_{R_0}^{R_1} r^{(\theta_1-\theta_0)q-1} \d r \, \int_{\Gamma}^{\Upsilon} \rho_\I(\phi)^{(\Theta_1-\Theta_0)q-2} \d \phi,
	\end{align*}
	as then both exponents are strictly larger than $-1$.
\end{proof}

\begin{remark} We add some comments with regard to  condition~\eqref{eq:cond-emb}. 
\begin{itemize}
    \item[(i)] 
	If $p_1=p_0$, the proof shows that we can allow for equalities in condition~\eqref{eq:cond-emb}. 
	Moreover, note that if $p_1\leq p_0$ and $\dist{B_R(z)}{0}>0$, then the condition $\theta_1>\theta_0$ in \eqref{eq:cond-emb} can be dropped as we can choose $R_0>0$ in our proof.
	Finally, if even $\dist{B_R(z)}{\partial \D}>0$, then \eqref{eq:cond-emb} can be dropped completely as then also $\Gamma>0$ and $\Upsilon<\kappa_0$ might be chosen.
	
    \item[(ii)] In general,  however,  both inequalities in \eqref{eq:cond-emb} can not be relaxed (up to possible equality): 
	\begin{itemize}
	\item[$\bullet$] Assuming that $\Theta_1-1/p_1 < \sigma < \Theta_0 - 1/p_0$, it is possible to construct a function $g_{\sigma} \in H^{\gamma}_{p_0,\Theta_0 p_0}(\D)\setminus H^{\gamma}_{p_1,\Theta_1 p_1}(\D)$ with $\supp(g_{\sigma})$ concentrated around some arbitrarily fixed $z_0\in\partial\D$. 
    Choosing $z_0\neq 0$ we may assume that $\supp(g_\sigma) \subset B_R(0)\setminus B_r(0)$ for some $0<r<R<\infty$.
	But then from \autoref{lem:monotone2}  we deduce  $g_\sigma\in H^{\gamma}_{p_0,\Theta_0 p_0, \theta_0 p_0}(\D)\setminus H^{\gamma}_{p_1,\Theta_1 p_1,\theta_1 p_1}(\D)$ for arbitrary $\theta_0,\theta_1\in\R$.
	Therefore, in general $\Theta_1-1/p_1 \geq \Theta_0 - 1/p_0$ is necessary.
	
	\item[$\bullet$] We can proceed as follows to show that necessarily $\theta_1 \geq \theta_0$ if $\dist{B_R(z)}{0}=0$ and $B_R(z)\cap \D \neq \emptyset$: 
	In this case, there exist open cones 
	$$
	    C_i:=\Phi\big((0,R_i)\times (\Gamma_i, \Upsilon_i)\big) \subset B_R(z) \cap \D, \qquad i=0,1,
	$$
	with $0<R_0<R_1<\infty$ and $0<\Gamma_1<\Gamma_0<\Upsilon_0<\Upsilon_1<\kappa_0$, i.e., $C_0\subsetneq C_1$. 
	  Then we can further   
	define a smooth cut-off function $\eta\in\mathscr{D}'(\D)$ for $C_0$ supported on $C_1$ in polar coordinates $x=\Phi(r,\phi)$ as the tensor product of $a,b\in C^\infty(\R)$ with
	$$
	    a(r)= \begin{cases}
	    1, & \abs{r} \leq R_0,\\
	    0, & \abs{r} \geq R_1
	    \end{cases}
	    \qquad \text{and} \qquad
	    b(\phi) = \begin{cases}
	    1, & \phi \in (\Gamma_0,\Upsilon_0),\\
	    0, & \phi \notin (\Gamma_1,\Upsilon_1).
	    \end{cases}
	$$
Carefully estimating the derivatives of   this   $\eta$ shows that for $1< p <\infty$ and $\gamma,\Theta,\theta,\sigma\in\R$ the function $f_\sigma :=\eta \, \rho_\circ^{-\sigma}$ belongs to $H^\gamma_{p,\Theta p, \theta p}(\D)$ whenever $\sigma<\theta$ while it is not contained in this space if $\theta < \sigma$. 
	This shows that $\theta_1<\theta_0$ would give a contradiction in \autoref{thm:right-embedding}.
		\end{itemize}
	\end{itemize}
\end{remark}

The following result describes the growth and decay of functions in $H^\gamma_{p,\Theta,\theta}(\D)$ (and their derivatives) near and far away from the boundary~$\partial\D$,
  respectively.  
It has already been proven in~\cite[Lemma~2.5(vi)]{KimLeeSeo2022b} by means of the characterization from~\autoref{cor:Lot:doubleweight}. We refer to \autoref{sec:Notation} for the precise definition of the Hölder-Zygmund spaces $\mathcal{C}^s(\D)$. 
Recall that for domains $\domain\subset\R^d$ and closed sets $\emptyset\neq M\subset\partial\domain$, we write $\Regd(\domain,M)$ to denote the set of regularized distances to $M$ on $\domain$, see  \autoref{rem:resolution}\ref{it:resolution:construction}. 

\begin{theorem}\label{thm:growth}
    Let $1<p<\infty$ and $\gamma,\Theta,\theta\in\R$  such that $\gamma-2/p\geq m+s$ with  $m\in\N_0$  and $0<s\leq 1$. Moreover, let $\psi_\D\in \Regd(\D,\partial\D)$ and $\psi_\circ\in\Regd(\D,\{0\})$.  
    Then every $u \in H^\gamma_{p,\Theta p, \theta p}(\D)$ admits continuous partial derivatives $\partial^\alpha u$ up to order $m$ and 
    \begin{align*}
        \sum_{\abs{\alpha}\leq m} \norm{ \psi_\circ^{\theta-\Theta} \, \psi_\D^{\Theta+\abs{\alpha}}\, \partial^\alpha u \sep C(\D) }
        +
        \sum_{\abs{\alpha}= m}  &\left[ \psi_\circ^{\theta-\Theta} \, \psi_\D^{\Theta+s+\abs{\alpha}}\, \partial^\alpha u \right]_{\mathcal{C}^s(\D)}\\
        &\qquad\qquad\lesssim 
        \norm{u \sep H^\gamma_{p,\Theta p,\theta p}(\D)},
        \qquad u\in H^\gamma_{p,\Theta p,\theta p}(\D).
    \end{align*}
Moreover, if for some $R>0$ there holds $\supp(u) \subset B_R(0)$, then 
    $\psi_\circ^{\theta-\Theta} \, \psi_\D^{\Theta+s+\abs{\alpha}}\, \partial^\alpha u \in \mathcal{C}^{s}(\D)$
    for all $\alpha\in\N_0^2$ with $\abs{\alpha}=m$  and
    $$
   		 \norm{ \psi_\circ^{\theta-\Theta} \, \psi_\D^{\Theta+s+\abs{\alpha}}\, \partial^\alpha u \sep \mathcal{C}^s(\D)} 
   		 \lesssim R_1^s\, \norm{u \sep H^\gamma_{p,\Theta p,\theta p}(\D)}, \qquad u\in H^\gamma_{p,\Theta p,\theta p}(\D),
    $$
	where $R_1:=\max\{ R,\, 1\}$.
\end{theorem}

\subsection{Characterization via polar coordinates}\label{sec:polar}
We now prove a characterization of the spaces $H^\gamma_{p,\Theta,\theta}(\D)$, $\gamma\in\N_0$, by means of polar coordinates that we shall  
use for the analysis of the Poisson equation in the   subsequent \autoref{sec:Poisson}.  
Recall that by $\Phi \colon (0,\infty)\times [0,2\pi)\to\R^2\setminus\{0\}$ we denote the transformation of polar coordinates into Cartesian coordinates which is a $C^\infty$ diffeomorphism from $\widetilde{\D}:=(0,\infty)\times\I$ onto $\D$, where $\I:=(0,\kappa)$.
In order to obtain a characterization of $H^\gamma_{p,\Theta,\theta}(\D)$, we will have to move from derivatives w.r.t.\  
Cartesian coordinates to derivatives in  
polar coordinates. In this context, the rotation matrices
\begin{equation}\label{matrix-A}
	A:=A(\phi):=\begin{pmatrix}
		\cos\phi & -\sin \phi \\
		\sin\phi & \cos\phi
	\end{pmatrix},
	\qquad \phi\in\I, 
\end{equation}
often appear naturally. 
Auxiliary results concerning switching from one system of coordinates to the other are collected in \autoref{app:polar:coord}.
Note in particular that if $\psi_\I\in \Regd(\I,\partial\I)$, then
\begin{align}\label{eq:reg_distances}
    \psi_\D(x)
    := r\,\psi_\I(\phi),
    \qquad x=\Phi(r,\phi)\in \D,
\end{align}
defines a regularized distance to $\partial\D$ on $\D$, i.e., $\psi_\D\in\Regd(\D,\partial\D)$. This follows by direct computation using \autoref{lem:higherorder}. 

\begin{defi}
    For $\gamma\in\N_0$, $1<p<\infty$, and $\Theta,\theta\in\R$, let
    \begin{align*}
        P^{\gamma}_{p,\Theta,\theta}(\widetilde{\D}) 
        := \left\{ \widetilde{u} \in L_{1,\loc}(\widetilde{\D}) :\  \norm{\widetilde{u} \sepb P^{\gamma}_{p,\Theta,\theta}(\widetilde{\D})} < \infty\right\},
	\end{align*}
	where
    \begin{align*}
        \norm{\widetilde{u} \sepb P^{\gamma}_{p,\Theta,\theta}(\widetilde{\D})} 
        := \left( \int_0^\infty r^{\theta-1} \sum_{j=0}^\gamma \norm{ ( r D_r )^j \widetilde{u}(r,\cdot) \sep H^{\gamma-j}_{p,\Theta-1+jp}(\I)}^p \d r \right)^{1/p}. 
	\end{align*}
\end{defi}
Based on this, our characterization reads as follows:

\begin{theorem}\label{thm:polar}
    Let $\gamma\in\N_0$, $1<p<\infty$, and $\Theta,\theta\in\R$.
    Then for $u\in \mathscr{D}'(\D)$ there holds $u\in H^{\gamma}_{p,\Theta,\theta}(\D)$ if, and only if, $\widetilde{u}=u\circ\Phi \in P^{\gamma}_{p,\Theta,\theta}(\widetilde{\D})$. Moreover,
    $$
        H^{\gamma}_{p,\Theta,\theta}(\D) \ni u \mapsto \norm{\widetilde{u} \sepb P^{\gamma}_{p,\Theta,\theta}(\widetilde{\D})}
    $$
    defines an equivalent norm in $H^{\gamma}_{p,\Theta,\theta}(\D)$. 
In particular,
\begin{align*}
    T_\Phi\colon H^\gamma_{p,\Theta,\theta}(\D)\to P^\gamma_{p,\Theta,\theta}(\widetilde{\D}),\qquad
    u \mapsto T_\Phi u := u\circ \Phi,
\end{align*}
defines an isomorphism.
\end{theorem}
\begin{proof}
    First note that for $u\in \mathscr{D}'(\D)$ \autoref{lem:FirstDerivative} ensures that whenever $\norm{u\sep H^{\gamma}_{p,\Theta,\theta}(\D)}$ is well-defined, so is $\norm{u \circ \Phi \sepb P^{\gamma}_{p,\Theta,\theta}(\widetilde{\D})}$ and vice versa, in the sense that all distributional derivatives involved are regular according to \autoref{thm:coincidence} and \autoref{lem:LotSpaces:properties}\ref{it:LotSpaces:N}.
    It hence suffices to show that $\norm{u\sep H^{\gamma}_{p,\Theta,\theta}(\D)} \sim \norm{u\circ\Phi \sepb P^{\gamma}_{p,\Theta,\theta}(\widetilde{\D})}$ for $u\in L_{1,\loc}(\D)$ with $D^\alpha_x u\in L_{1,\loc}(\D)$ for all $\abs{\alpha}\leq\gamma$ (obviously, $\norm{\cdot \sepb P^{\gamma}_{p,\Theta,\theta}(\widetilde{\D})}$ is a norm).
    To prove this, we use mathematical induction on $\gamma\in\N_0$.
    We start with $\gamma=0$. In this case, by \autoref{thm:coincidence}, with $\mu$ as in~\autoref{lem:dist-polar},
    \begin{align*}
        \norm{u \sep H^0_{p,\Theta,\theta}(\D)}^p
        &\sim \int_\D \abs{u(x)}^p \rho_\circ(x)^{\theta-2}\sgrklam{\frac{\rho_\D(x)}{\rho_\circ(x)}}^{\Theta-2}\,\dx\\
        &= \int_0^\infty \int_\I \abs{ u(\Phi(r,\phi)) }^p r^{\theta-2} \, \sin(\mu(\phi))^{\Theta-2} \d \phi \,r \d r \\ 
        &\sim \int_0^\infty r^{\theta-1} \norm{ u(\Phi(r,\cdot)) \sep H^{0}_{p,\Theta-1}(\I)}^p \d r  
        = \norm{\widetilde{u} \sepb P^{\gamma}_{p,\Theta,\theta}(\widetilde{\D})}^p,
        \quad u\in L_{1,\loc}(\D).
    \end{align*}
    
    To verify the induction step $\gamma\mapsto\gamma+1$, we assume the assertion holds for some $\gamma\in\N_0$. 
    Let $\psi_\D$ and $\psi_\I$ be the regularized distances from \eqref{eq:reg_distances}. 
    Then the induction hypothesis yields
    \begin{align*}
        \norm{\psi_\D\, Du \sep H^{\gamma}_{p,\Theta,\theta}(\D)}^p
        \sim \int_0^\infty r^{\theta-1} \sum_{j=0}^{\gamma} \norm{ (r D_r)^j \big[(\psi_\D\, Du)\circ \Phi\big](r,\cdot) \sep H^{\gamma-j}_{p,\Theta-1+jp}(\I)}^p \d r
	\end{align*}
    with constants that do not depend on $u$. 
    Therein, by \autoref{lem:FirstDerivative} as well as \autoref{lem:orthogonality} (with $M:=A$ from~\eqref{matrix-A}), and \autoref{lem:LotSpaces:properties}\ref{it:LotSpaces:indexshift}, the inner norms satisfy
    \begin{align*}
        &\norm{ (r D_r)^j \big[(\psi_\D\, Du)\circ \Phi\big](r,\cdot) \sep  H^{\gamma-j}_{p,\Theta-1+jp}(\I)} \\
        &\quad = \norm{ (r D_r)^j \left[\psi_\I\, A\, \begin{pmatrix}
			(r D_r)  \\ 
			D_\phi 
		\end{pmatrix} (u\circ\Phi) \right](r,\cdot) \sep H^{\gamma-j}_{p,\Theta-1+jp}(\I)} \\
        &\quad = \norm{ A \left[\begin{pmatrix}
            \psi_\I\,(r D_r)^{j+1}  \\ 
            \psi_\I\,(r D_r)^j D_\phi 
        \end{pmatrix} \widetilde{u}\right] (r,\cdot) \sep H^{\gamma-j}_{p,\Theta-1+jp}(\I)}\\
        &\quad
        \sim \norm{\psi_\I \begin{pmatrix}
             (r D_r)^{j+1} \widetilde{u}\\ 
            D_\phi (r D_r)^j \widetilde{u}
        \end{pmatrix} (r,\cdot) \sep H^{\gamma-j}_{p,\Theta-1+jp}(\I)} \\
        &\quad
        \sim \norm{ \psi_\I \eklam{D_\phi (r D_r)^j \widetilde{u}}(r,\cdot) \sep H^{\gamma-j}_{p,\Theta-1+jp}(\I)} + \norm{ \eklam{(r D_r)^{j+1} \widetilde{u}}(r,\cdot) \sep H^{\gamma+1-(j+1)}_{p,\Theta-1+(j+1)p}(\I)}
    \end{align*}
    with constants that do neither depend on $u$ nor on $r$.
    Thus, by \autoref{thm:lifting} (with $m=1$), the induction hypothesis, and \autoref{lem:LotSpaces:properties}\ref{it:LotSpaces:lifting} we have
    \begin{align*}
        \norm{u \sep H^{\gamma+1}_{p,\Theta,\theta}(\D)}^p
        &\sim \norm{u \sep H^{\gamma}_{p,\Theta,\theta}(\D)}^p + \norm{\psi_\D\, Du \sep H^{\gamma}_{p,\Theta,\theta}(\D)}^p \\
        &\sim \int_0^\infty r^{\theta-1} \sum_{j=0}^{\gamma} \norm{ \eklam{(r D_r)^j \widetilde{u}}(r,\cdot) \sep H^{\gamma-j}_{p,\Theta-1+jp}(\I)}^p \d r \\
        &\qquad + \int_0^\infty r^{\theta-1} \sum_{j=0}^{\gamma} \norm{ \psi_\I\eklam{D_\phi (r D_r)^j \widetilde{u}}(r,\cdot) \sep H^{\gamma-j}_{p,\Theta-1+jp}(\I)} \d r\\		
        &\qquad + \int_0^\infty r^{\theta-1} \sum_{k=1}^{\gamma+1} \norm{\eklam{(r D_r)^{k} \widetilde{u}} (r,\cdot) \sep H^{\gamma+1-k}_{p,\Theta-1+kp}(\I)} \d r \\
        &\sim \int_0^\infty r^{\theta-1} \sum_{j=0}^{\gamma+1} \norm{ \eklam{(r D_r)^j \widetilde{u}}(r,\cdot) \sep H_{p,\Theta-1+jp}^{\gamma+1-j}(\I)}^p \d r 
        = \norm{\widetilde{u} \sepb P^{\gamma+1}_{p,\Theta,\theta}(\widetilde{\D})}
    \end{align*}
    for all $u\in L_{1,\loc}(\D)$ with $D^\alpha_x u\in L_{1,\loc}(\D)$ for all $\abs{\alpha}\leq\gamma$ and the constants that do not depend on $u$. 
\end{proof}

\subsection{Characterization via Mellin transform}

Finally, we present a characterization of the spaces $P^\gamma_{2,\Theta,\theta}(\widetilde{\D})$  by means of Mellin transforms, which, due to \autoref{thm:polar}, will also provide us with a new characterization of $H^\gamma_{2,\Theta,\theta}(\D)$. It will play a central role in our existence and uniqueness statement in \autoref{sec:Poisson}. Recall that for any domain $\domain\subset\real^d$, we write $L_{p,\Theta}(\domain)=L_p(\domain,\B(\domain),\rho_\domain^{\Theta-d}\lambda^d;\C)$, see \autoref{def:Lot_spaces} and \autoref{lem:LotSpaces:properties}.

For test functions $u\in C^\infty_0(\real_+)$, the Mellin transform is defined by
\beq \label{mellin-1}
(\mathcal{M}u)(\lambda)
:=
(\mathcal{M}_{r \rightarrow \lambda}u)(\lambda)
:=
\int_0^{\infty}r^{-\lambda-1}u(r)\d r, \qquad \lambda\in\C.
\eeq 
In the following lemma we list some basic properties of this transform, cf.~\cite[Lemma~3.3.6]{MazRos2010}, see also~\cite{ButJan1999}. For $c\in\real$ we write $\Gamma_c:=\ggklam{z\in \C\,\colon \mathrm{Re}(z)=c}$ and 
\[
Y^2_c:=\ggklam{u\colon \Gamma_c\to\C\,\sep\ \grklam{t\mapsto u(c+it)}\in L_2(\real)},
\]
endowed with the norm
\[
\norm{u\sep Y^2_c}:=\ssgrklam{\frac{1}{2\pi i}\int_{c-i\infty}^{c+i\infty} \abs{u(\lambda)}^2\d \lambda}^{1/2},\qquad u\in Y^2_c.
\]
Moreover, for $\domain\subset \C$, $u\colon\domain \to\C$, and $\alpha\in\real$ we write $z^\alpha u$ for the multiplication of $u$ by $z\mapsto z^{\alpha}$, i.e., 
\[
(z^\alpha u)(z):=z^\alpha u(z),\qquad z\in\domain.
\]

\begin{lemma}[Properties of the Mellin transform]\label{mellin-prop}
The following assertions hold.
\begin{enumerate}[label=\textup{(\roman*)}]
\item The transformation \eqref{mellin-1} realizes a linear 
mapping from $C^{\infty}_0(\real_+)$ into the space of analytic functions on $\mathbb{C}$. 
\item\label{it:Mellin:mult} $(\mathcal{M}_{r\to\lambda}(r\partial_ru))(\lambda)=\lambda (\mathcal{M}_{r\to\lambda}u)(\lambda)$ for all $u\in C^{\infty}_0(\real_+)$.
\item The (left) inverse Mellin transform is given by 
\[
u(r)=\frac{1}{2\pi i}\int_{-\beta-i\infty}^{-\beta+i\infty}r^{\lambda}\mathcal{M}{u}(\lambda)\d\lambda, \qquad r>0, 
\]
where $\beta\in\R$ is arbitrary.
\item\label{it:Mellin:iso} For  $\beta\in\real$ the transform \eqref{mellin-1} extends to an invertible linear isometry 
\[
\mathcal{M}_\beta\colon L_{2,2\beta}(\real_+)= L_2(\real_+,r^{2\beta-1}\dr)\rightarrow Y^2_{-\beta}.
\]
Moreover,  Parseval's identity holds, i.e., 
\[
\int_0^{\infty}r^{2\beta-1}u(r)\overline{v(r)}\dr =\frac{1}{2\pi i}\int_{-\beta-i\infty}^{-\beta+i\infty}\mathcal{M}_\beta{u}(\lambda)\overline{\mathcal{M}_\beta{v}(\lambda)}\d\lambda,\quad u,v\in L_{2,2\beta}(\real_+).
\]
\item If  for $\beta<\beta'$ we have $u\in L_{2,2\beta}(\real_+)\cap L_{2,2\beta'}(\real_+)$, then $\lambda\mapsto\mathcal{M}_{\mathrm{Re}\lambda}{u}(\lambda)$ is holomorphic in the strip $\{\lambda \in \C: \ -\beta'<\mathrm{Re}\lambda<-\beta\}$.
\end{enumerate}
\end{lemma}

We will need the following generalization of \autoref{mellin-prop}\ref{it:Mellin:iso}. For $k\in\mathbb{N}_0$ and $c\in\real$ we write 
\[
Y^{2,k}_c:=\ggklam{u\in Y^2_c\,\colon\,\norm{u\sep Y^{2,k}_c}<\infty},
\]
where
\[
\norm{u\sep Y^{2,k}_c}:=\ssgrklam{\sum_{j=0}^k \norm{\grklam{\lambda\mapsto\lambda^ju(\lambda)}\sep Y^2_c}}^{1/2}.
\]

\begin{lemma}\label{lem:Mellinchar}
    Let $\beta\in\real$ and $\gamma\in\mathbb{N}_0$. Then 
    \[
    \mathcal{M}_{\beta,\gamma}\colon H^\gamma_{2,2\beta}(\real_+)\to Y^{2,\gamma}_{-\beta},
    \quad
    u\mapsto \mathcal{M}_{\beta,\gamma}u:=\mathcal{M}_\beta u,
    \]
    is an invertible linear operator with bounded inverse $\mathcal{M}_{\beta,\gamma}^{-1}:=\mathcal{M}_\beta^{-1}|_{Y^{2,\gamma}_{-\beta}}$. In particular,
    \begin{equation}\label{eq:Mellin:multiplier}
        \grklam{\mathcal{M}_\beta(r\mapsto rD_ru(r))}(\lambda)
        =
        \lambda (\mathcal{M}_\beta u)(\lambda),
        \quad \lambda\in\Gamma_{-\beta},
        \quad u\in H^1_{2,2\beta}(\R_+),
    \end{equation}
    and
   \[ 
   \norm{\mathcal{M}_\beta u \sep Y^{2,\gamma}_{-\beta}}
   \sim
   \norm{u \sep H^\gamma_{2,2\beta}(\real_+)},
   \quad u\in H^\gamma_{2,2\beta}(\real_+).
   \]
\end{lemma}

\begin{proof}
Let $\beta\in\real$.
Note first that the assertion is satisfied for $\gamma=0$ due to \autoref{mellin-prop}\ref{it:Mellin:iso}.
Moreover, due to Lemma~\ref{lem:LotSpaces:properties}\ref{it:LotSpaces:lifting}, for all $\gamma\in\N$, $u\in H^\gamma_{2,2\beta}(\real_+)$ if, and only if, $(rD_r)^k u\in L_{2,2\beta}(\real_+)$ for all $k\in\{0,1,\ldots,\gamma\}$ and
\[
\norm{u\sep H^\gamma_{2,2\beta}(\real_+)}
\sim
\ssgrklam{\sum_{k=0}^\gamma \norm{(rD_r)^k u\sep L_{2,2\beta}(\real_+)}^2}^{1/2},\quad u\in H^\gamma_{2,2\beta}(\real_+).
\]
Thus, for arbitrary $u\in C^\infty_0(\real_+)$, \autoref{mellin-prop}\ref{it:Mellin:iso} yields that
\begin{align*}
    \norm{u\sep H^\gamma_{2,2\beta}(\real_+)}
&\sim
\ssgrklam{\sum_{k=0}^\gamma \norm{(r\partial_r)^k u\sep L_{2,2\beta}(\real_+)}^2}^{1/2}\\
&=
\ssgrklam{\sum_{k=0}^\gamma \norm{\lambda^k \mathcal{M}u\sep Y^2_{-\beta}(\real_+)}^2}^{1/2}
=
\norm{ \mathcal{M}u\sep Y^{2,\gamma}_{-\beta}(\real_+)}.
\end{align*}
Since $C^\infty_0(\real_+)$ is dense in $H^\gamma_{2,2\beta}(\real_+)$ (see \autoref{lem:LotSpaces:properties}\ref{it:LotSpaces:density}), the assertion follows as soon as we prove that $\mathcal{M}_{\beta,\gamma}$ is surjective.
The latter follows by induction over $\gamma\in\N_0$ as soon as we can prove that if $v\in Y^2_{-\beta}$ and $\lambda v\in Y^2_{-\beta}$, then $u:=\mathcal{M}_\beta^{-1}v\in L_{2,2\beta}(\real_+)$ and $rD_r u=\mathcal{M}_\beta^{-1}(\lambda v)\in L_{2,2\beta}(\real_+)$---which also proves~\eqref{eq:Mellin:multiplier}. But this follows from the fact that
\[
(\mathcal{M}_{\beta}^{-1}(\lambda v),\psi)
=
-(u,\partial_r(r\psi)),
\quad
\psi\in C^\infty_0(\real_+),
\]
which may be checked by means of \autoref{mellin-prop}\ref{it:Mellin:mult}.
\end{proof}

As we will see below, the ranges of the spaces $P^\gamma_{2,\Theta,\theta}(\widetilde{\D})$ under the Mellin transform with respect to $r$ are given by the following Hilbert spaces.

\begin{defi}
Let $\Theta,\theta\in\real$. We write $Y^0_{2,\Theta,\theta}(\I)$ for the space of all functions $u\colon \Gamma_{-\theta/2}\times\I\to\C$, such that $(t,\phi)\mapsto u(-\theta/2 + it,\phi)$ is Borel measurable and
\[
\norm{u\sep Y^{0}_{2,\Theta,\theta}(\I)}
:=
\ssgrklam{
\frac{1}{2\pi i}\int_{-\theta/2-i\infty}^{-\theta/2+i\infty} \norm{u(\lambda,\cdot)\sep L_{2,\Theta-1}(\I)}^2\,\mathrm{d}\lambda}^{1/2}
<\infty.
\]
Moreover, for $\gamma\in\N$, we introduce the space 
    \[
    Y^{\gamma}_{2,\Theta,\theta}(\I)
    :=
    \ggklam{
    u \colon \lambda^jD^\alpha_\phi u \in Y^0_{2,\Theta+2j+2\alpha,\theta}(\I)\text{ for all } j\in\{0,\ldots,\gamma\},\,\alpha\in\{0,\ldots,\gamma-j\}},
    \]
    endowed with the norm
\[
\norm{u\sep Y^\gamma_{2,\Theta,\theta}(\I)}
:=
\ssgrklam{
\frac{1}{2\pi i}\sum_{j=0}^\gamma \int_{-\theta/2-i\infty}^{-\theta/2+i\infty} \abs{\lambda}^{2j}\norm{u(\lambda,\cdot)\sep H^{\gamma-j}_{2,\Theta-1+2j}(\I)}^2\,\mathrm{d}\lambda}^{1/2}.
\]
\end{defi}

Using the preparations above we can prove the following characterization of the spaces $P^\gamma_{2,\Theta,\theta}(\widetilde{\D})$ by means of the Mellin transform

\begin{theorem}\label{thm:Mellinrepresentation}
    Let $\gamma\in\N_0$ and let $\Theta,\theta\in\real$. Then the mapping
    \begin{align*}
        \mathcal{M}\colon P^\gamma_{2,\Theta,\theta}(\widetilde{\D}) \to Y^\gamma_{2,\Theta,\theta}(\I),\qquad u\mapsto \mathcal{M}u
    \end{align*}
    with 
    \[
        \grklam{\mathcal{M}u}(\lambda,\phi):=\mathcal{M}_{-\theta/2}
        \grklam{r\mapsto u(r,\phi)}(\lambda),
        \qquad (\lambda,\phi)\in\Gamma_{-\theta/2}\times \I,
    \]
    defines an invertible, bounded, linear operator with bounded inverse $\mathcal{M}^{-1}$. In particular, 
    \[
    \norm{\mathcal{M}u\sep Y^\gamma_{2,\Theta,\theta}(\I)}\sim \norm{u\sep P^\gamma_{2,\Theta,\theta}(\widetilde{\D})},
    \qquad u\in P^{\gamma}_{2,\Theta,\theta}(\widetilde{\D}).
    \]
\end{theorem}

\begin{proof}
Let $u\in P^\gamma_{2,\Theta,\theta}(\widetilde\D)$. 
Then,  $\grklam{r\mapsto D^\alpha_\phi(r,\phi)}\in H^{\gamma-\alpha}_{2,\theta}(\real_+)$ for almost all $\phi\in\I$, for all $\alpha\in\{0,1,\ldots,\gamma\}$.
Thus, applying \autoref{lem:Mellinchar} we get that 
\begin{align*}
    \norm{u\sep P^\gamma_{2,\Theta,\theta}(\widetilde\D)}^2
    &=
    \sum_{\alpha=0}^\gamma \sum_{j=0}^{\gamma-\alpha} \int_0^\kappa \int_0^\infty \Abs{(rD_r)^jD_\phi^\alpha u(r,\phi)}^2 r^{\theta-1}\dr \rho_\I(\phi)^{\Theta-2(1-j-\alpha)}\d\phi\\
    &\sim
    \sum_{\alpha=0}^\gamma \int_0^\kappa \norm{D^\alpha_\phi u(\cdot,\phi) \sep H^{\gamma-\alpha}_{2,\theta}(\R_+)}^2 \rho_\I(\phi)^{\Theta-2(1-j-\alpha)}\d\phi\\
     &\sim
    \sum_{\alpha=0}^\gamma \int_0^\kappa \norm{\lambda\mapsto \grklam{\mathcal{M}_{\theta/2} (D^\alpha_\phi u(\cdot,\phi))}(\lambda) \sep Y^{2,\gamma-\alpha}_{-\theta/2}}^2 \rho_\I(\phi)^{\Theta-2(1-j-\alpha)}\d\phi\\
    &\sim
    \sum_{\alpha=0}^\gamma \sum_{j=0}^{\gamma-\alpha} \int_0^\kappa \frac{1}{2\pi i}\int_{-\theta/2-i\infty}^{-\theta/2+i\infty} \Abs{\lambda^j\mathcal{M}_{\theta/2} (D_\phi^\alpha u(\cdot,\phi))(\lambda)}^2 \d\lambda \rho_\I(\phi)^{\Theta-2(1-j-\alpha)}\d\phi\\
    &\sim
    \norm{\mathcal{M}u\sep Y^\gamma_{2,\Theta,\theta}(\I)}^2.
\end{align*}
In the last step we used the fact that $\mathcal{M}_{\theta/2}D_\phi^\alpha u = D_\phi^\alpha\mathcal{M}_{\theta/2} u$, which can be checked for smooth $u$ by means of Lebesgue dominated convergence theorem and then extended to arbitrary $u\in P^\gamma_{2,\Theta,\theta}(\widetilde\D)$ by means of a density argument.
Thus, linearity, boundedness and injectivity are proven. Surjectivity follows from the invertibility of $\mathcal{M}_{\theta/2}$.
\end{proof}

\section{The Poisson equation in \texorpdfstring{$H^\gamma_{p,\Theta,\theta}(\D)$}{Hgamma,p,Theta,theta(D)}}
\label{sec:Poisson}

In this section we begin the study of the regularity of the Poisson equation~\eqref{eq:Poisson} 
within the scale  $H^\gamma_{p,\Theta,\theta}(\D)$.
On the one hand, we show that these spaces are suitable for establishing higher order regularity for the Poisson equation in the sense that the regularity of the solution within this scale of spaces can be lifted with the regularity of the forcing term. This works for arbitrary $p>1$.
On the other hand, we establish existence and uniqueness for the case $p=2$ and a sharp range of weight parameters $\Theta$ and $\theta$. 
The latter is done by means of the Mellin transform and suitable resolvent estimates for the Dirichlet Laplacian on an interval within the scale $H^\gamma_{2,\Theta}(\D)$ from~\cite{LinVer2020}.
The case $p\neq 2$ is postponed to a forthcoming paper. Our main result reads as follows.

\begin{theorem}\label{thm:Poisson:main}
Let $1<p<\infty$, $\theta,\Theta\in\R$, and $\gamma\in\N_0$. Then the following assertions hold.

\begin{enumerate}[label=\textup{(\roman*)}]
\item\label{it:Poisson:main:lifting} \textup{Lifting.} Let $1<\Theta<p+1$, let $f\in H^{\gamma}_{p,\Theta+p,\theta+p}(\D)$, and let $u\in H^{\gamma+1}_{p,\Theta-p,\theta-p}(\mathcal{D})$ be such that 
    \begin{equation}\label{eq:PoissonD}
        \Delta u = f \quad \text{on } \D.
    \end{equation}
Then $u\in H^{\gamma+2}_{p,\Theta-p,\theta-p}(\D)$ and 
    \[
        \norm{ u \sep H^{\gamma+2}_{p,\Theta-p, \theta-p}(\D) }
        \lesssim \norm{ f \sep H^{\gamma}_{p,\Theta+p, \theta+p}(\D) } + \norm{ u \sep H^{\gamma+1}_{p,\Theta,\theta-p}(\D)}
    \]
    with a constant that does not depend on $f$ and $u$.
\item \label{it:Poisson:main:existence}\textup{Existence.} Let $p=2$ and assume that
\begin{equation}\label{eq:PoissonRange}
    1<\Theta<3\qquad\text{and}\qquad 
    \frac{\theta-2}{2}\notin \sggklam{\pm n\frac{\pi}{\kappa}\colon n\in\N}.
\end{equation}
Then for all $f\in H^\gamma_{2,\Theta+2,\theta+2}(\D)$ there exists a unique $u\in H^{\gamma+2}_{2,\Theta-2,\theta-2}(\D)$ such that \eqref{eq:PoissonD} holds. Moreover, 
\begin{equation}\label{eq:a-priori}
\lVert u|H^{\gamma+2}_{2,\Theta-2, \theta-2}(\D)\rVert\lesssim \lVert f|H^{\gamma}_{2,\Theta+2, \theta+2}(\D)\rVert    
\end{equation}
with a constant that does not depend on $f$ and $u$.
\end{enumerate}
\end{theorem}

\begin{remark}
 The following remarks are in order.
\begin{enumerate}[label=\textup{(\roman*)}]
\item The solutions $u\in H^2_{p,\Theta-p,\theta-p}(\D)$ to~\eqref{eq:PoissonD} in Theorem~\ref{thm:Poisson:main} can be seen as a solution to the Poisson equation~\eqref{eq:Poisson} with zero Dirichlet boundary condition.
This is because for the range of $\Theta$ therein, i.e., for $1<\Theta<p+1$, any $u\in H^2_{p,\Theta-p,\theta-p}(\D)$ has trace zero, since for all $\zeta\in C_0^\infty(\R_+)$, 
\[
\zeta(\abs{\cdot}) u\in H^2_{p,\Theta-p}(\domain)
=
\ggklam{u\colon D^\alpha u\in L_p(\domain,\omega_{\Theta+p-d}^\domain),\alpha\leq 2, \textup{Tr}\,u=0}
\]
on a suitable bounded $C^2$ domain $\domain\subset\D$; the equality above is proven in Lemma~\ref{lem:domain:LinVer} below.

\item The ranges of parameters in Theorem~\ref{thm:Poisson:main} include the ranges obtained so far in the analysis of the (stochastic) heat equation within the spaces $H^\gamma_{p,\Theta,\theta}(\D)$ on angular domains, see~\cite{Cio20,KimLeeSeo2021,KimLeeSeo2022b}.
They are sharp in the following sense: As mentioned in \cite[Remark~2.7]{Kim2004}, the restriction $1<\Theta<p+1$ on the parameter $\Theta$ is necessary in order to obtain  the corresponding result for the (stochastic) heat equation within the scale $H^\gamma_{p,\Theta}(\domain)$ on $C^1$ domains $\domain\subset\R^2$. 
Therefore, since a solution to Equation~\eqref{eq:Poisson} on $\D$ that vanishes near the vertex can be considered as a solution to the corresponding steady-state equation on a suitable $C^1$ domain, the range of $\Theta$ in Theorem~\ref{thm:Poisson:main} is sharp.
The range of $\theta$ coincides with the one obtained for the non-degenerate Poisson equation with Dirichlet boundary condition on $\D$ in~\cite[Theorem~6.1.1]{KozMazRos1997}.
It includes the set of all $\theta \in\R$ satisfying 
\[
2\sgrklam{1-\frac{\pi}{\kappa}}<\theta<2\sgrklam{1+\frac{\pi}{\kappa}},
\]
which is exactly the range obtained for the non-degenerate heat equation with zero Dirichlet boundary conditions on the cone $\D$, see, e.g.,~\cite{Naz2001,KozNaz2014,PruSim2007}.
\end{enumerate}
\end{remark}

We are going to prove the two parts of  Theorem~\ref{thm:Poisson:main} separately. We start with part~\ref{it:Poisson:main:lifting}, the lifting. The proof uses a standard localization argument and corresponding regularity estimates in the scale of spaces $H^\gamma_{p,\Theta}(\domain)$ on $C^1$-domains.

\begin{proof}[Proof of Theorem \ref{thm:Poisson:main}\ref{it:Poisson:main:lifting}]
    Let $\zeta=(\zeta_\nu)_{\nu\in\Z}$ be as in \autoref{conv:zeta}, cf. also~\autoref{rem:resolution}\ref{it:resolution:constr:D}. 
Set $\eta:=\zeta_0$ and recall that $\zeta_\nu =\eta(e^{-\nu}\cdot)$ for all $\nu\in\Z$.
Choose a $C^1$-domain $G\subset\D$ such that
\begin{equation}\label{eq:DtoG}
 \supp(\zeta_0)\cap\D\subset G 
\qquad \text{and} \qquad
\rho_G\sim\rho_\D\text{ on }\supp(\zeta_0);   
\end{equation}
see the proof of \cite[Lemma~3.7]{CioKimLee2019} for a construction of an appropriate $G$. Note that~\eqref{eq:DtoG} guarantees that for all $m\in\N_0$ and all $\vartheta\in\R$, for all $g\in L_{1,\loc}(\D)$ and all smooth $\widetilde{\eta}$ with $\supp(\widetilde{\eta})\subseteq\supp(\eta)$ we have that $\widetilde{\eta} g\in H^m_{p,\vartheta}(\D)$ if, and only if, $\widetilde{\eta} g\in H^m_{p,\vartheta}(G)$ and that in this case
\begin{equation}\label{eq:normequiv:D:G}
\norm{\widetilde{\eta} g \sep H^m_{p,\vartheta}(\D)}
\sim
\norm{\widetilde{\eta} g \sep H^m_{p,\vartheta}(G)}    
\end{equation}
with constants that do not depend on $g$ and $\widetilde{\eta}$.
Moreover, since $\Delta u = f$ on $\D$, we obtain for all $\nu\in\Z$ that
\begin{align*}
     \Delta \left(\eta u(e^{\nu}\cdot)\right)=
     e^{2\nu}\eta f(e^{\nu}\cdot)+2\sum_{i=1}^2(\eta_{x_i}u(e^{\nu}\cdot))_{x_i}- 
     \Delta\eta\cdot  u(e^{\nu}\cdot)=:\widetilde{f}_\nu \qquad \text{on}\quad G.
\end{align*}
Thus, if we can show that $\eta u(e^\nu\cdot)\in H^{\gamma+1}_{p,\Theta-p}(G)$ and $\widetilde{f}_\nu\in H^{\gamma}_{p,\Theta+p}(G)$ for all $\nu\in\Z$, then \cite[Theorem~2.11]{KimKry2004} yields
\[
\norm{\eta u(e^\nu\cdot) \sep H^{\gamma+2}_{p,\Theta+p}(G)}
\lesssim
\norm{\widetilde{f}_\nu \sep H^{\gamma}_{p,\Theta-p}(G)},\qquad \nu\in\Z,
\]
which, in turn, due to Lemma~\ref{lem:norm-change} and the fact that both $(\eta_{x_i}(e^{-\nu}\cdot))_{\nu\in\Z}$ as well as $(\eta_{x_ix_i}(e^{-\nu}\cdot))_{\nu\in\Z}$ belong to $\mathscr{A}_{e}(\D,\{0\})$ (see~\autoref{def:resolutions:sets}),
yields
\begin{align*}
    \norm{u\sep H^{\gamma+2}_{p,\Theta-p,\theta-p}(\D)}^p
    &=
    \sum_{\nu\in \Z}e^{\nu(\theta-p)}\norm{(\zeta_\nu u)(e^{\nu}\cdot) \sep H^{\gamma+2}_{p,\Theta-p}(\D)}^p\\
    &\sim
    \sum_{\nu\in \Z}e^{\nu(\theta-p)}\norm{\eta u(e^\nu\cdot) \sep H^{\gamma+2}_{p,\Theta-p}(G)}^p\\
    &\lesssim
    \sum_{\substack{\nu\in \Z\\ i=1,2}}e^{\nu(\theta-p)}\norm{e^{2\nu}\eta f(e^\nu\cdot) + (\eta_{x_i} u(e^\nu\cdot))_{x_i} +\eta_{x_ix_i} u(e^\nu\cdot)\sep H^{\gamma}_{p,\Theta+p}(G)}^p\\
    &\lesssim
    \sum_{\nu\in \Z}e^{\nu(\theta+p)}\norm{(\zeta_\nu f)(e^\nu\cdot) \sep H^{\gamma}_{p,\Theta+p}(\D)}^p\\
    &\qquad + \sum_{i=1}^2\sum_{\nu\in \Z}e^{\nu(\theta-p)}\norm{(\eta_{x_i}(e^{-\nu}\cdot) u)(e^\nu\cdot)) \sep H^{\gamma+1}_{p,\Theta}(\D)}^p\\
    &\qquad + \sum_{i=1}^2\sum_{\nu\in \Z}e^{\nu(\theta-p)}\norm{(\eta_{x_ix_i}(e^{-\nu}\cdot) u)(e^\nu\cdot)\sep H^{\gamma}_{p,\Theta+p}(\D)}^p\\
    &\lesssim
    \norm{f \sep H^{\gamma}_{p,\Theta+p,\theta+p}(\D)}
    +
    \norm{u \sep H^{\gamma+1}_{p,\Theta,\theta-p}(\D)};
\end{align*}
note that in the second but last step we also used Lemma~\ref{lem:LotSpaces:properties}, parts~\ref{it:LotSpaces:indexshift} and~\ref{it:LotSpaces:lifting}.
Since a very similar calculation yields that, indeed, $\eta u(e^\nu\cdot)\in H^{\gamma+1}_{p,\Theta-p}(G)$ and $\widetilde{f}_\nu\in H^{\gamma}_{p,\Theta+p}(G)$ for all $\nu\in\Z$, the assertion follows.
\end{proof}

Now we move towards proving the existence part of Theorem~\ref{thm:Poisson:main}. Since the lifting part is already proven, it is obvious that it is enough to check existence for $\gamma=0$.
To this end, we use the fact that the Laplacian
\[
\Delta_x\colon H^2_{2,\Theta-2,\theta-2}(\D)\to L_{2,\Theta+2,\theta+2}(\D)
\]
is a bounded linear operator and that proving Theorem~\ref{thm:Poisson:main}\ref{it:Poisson:main:existence} means showing that this operator is invertible. Estimate~\eqref{eq:a-priori} follows then from the boundedness of the inverse of $\Delta_x$, which is a consequence of the open mapping theorem, as the spaces involved are Banach spaces. A close look at $\Delta_x$ for fixed $\theta,\Theta\in\R$ shows that
\[
\Delta_x = T_{\Phi,0}^{-1} \circ N_2^{-1} \circ \mathcal{M}_{\frac{2-\theta}{2},0}^{-1} \circ B \circ \mathcal{M}_{\frac{2-\theta}{2},2} \circ T_{\Phi,2},
\]
where 
\[
T_{\Phi,2}\colon H^2_{2,\Theta-2,\theta-2}(\D)\to P^2_{2,\Theta-2,\theta-2}(\widetilde{\D})
\quad\text{and}\quad
T_{\Phi,0}\colon H^0_{2,\Theta+2,\theta+2}(\D)\to P^0_{2,\Theta+2,\theta+2}(\widetilde{\D})
\]
are transformations from polar to Cartesian coordinates as introduced in Theorem~\ref{thm:polar},
\[
\mathcal{M}_{\frac{2-\theta}{2},2} \colon P^2_{2,\Theta-2,\theta-2}(\widetilde{\D}) \to Y^2_{2,\Theta-2,\theta-2}(\I)
\quad\text{and}\quad
\mathcal{M}_{\frac{2-\theta}{2},0} \colon P^0_{2,\Theta+2,\theta-2}(\widetilde{\D}) \to Y^0_{2,\Theta+2,\theta-2}(\I)
\]
are Mellin transforms as introduced in Theorem~\ref{thm:Mellinrepresentation},
\[
N_2\colon P^0_{2,\Theta+2,\theta+2}(\widetilde{\D})\to P^0_{2,\Theta+2,\theta-2}(\widetilde{\D}),\quad
\widetilde{u}\mapsto N_2\widetilde{u}:=\{(r,\phi)\mapsto r^2\widetilde{u}(r,\phi)\}
\]
is an isomorphism (see Corollary~\ref{cor:pwmult:indexshift} together with Theorem~\ref{thm:polar})
and
\begin{equation}\label{eq:ResolventY}
B\colon Y^2_{2,\Theta-2,\theta-2}(\I)\to Y^0_{2,\Theta+2,\theta-2}(\I),
\quad
v\mapsto Bv:=\{(\lambda,\phi)\mapsto (\lambda^2+D_\phi^2)v(\lambda,\phi)\}.
\end{equation}
As demonstrated above, except for $B$, all these operators are known to be isomorphisms. Thus, if we can prove that $B$ is also an isomorphism, so is $\Delta_x$ and the existence part of Theorem~\ref{thm:Poisson:main} is proven. Clearly, $B$ is linear and bounded. To obtain its invertibility we rely on results from~\cite{LinVer2020}. 
Therein, among others, for bounded $C^2$-domains $\domain\subset\R^d$, the Dirichlet Laplacian $\Delta_\textup{Dir}^\domain$ is analysed as an unbounded operator in the weighted $L_p$-spaces $L_p(\domain,\omega_\nu^\domain):=L_{p,\nu+d}(\domain)$, cf. Definition~\ref{def:Lot_spaces} and Lemma~\ref{lem:LotSpaces:properties}\ref{it:LotSpaces:N}. 
In~\cite{LinVer2020}, for $p-1<\nu<2p-1$, the Dirichlet Laplacian $\Delta_\textup{Dir}^\domain$ in $L_p(\domain,\omega_\nu^\domain)$ is given by 
\[
D(\Delta_\textup{Dir}^\domain)
:=
W^{2,p}_{\mathrm{Dir}}(\mathcal{O}, \omega_{\nu}^{\mathcal{O}})
:=
\{
u\in W^{2,p}(\mathcal{O}, \omega_{\nu}^{\mathcal{O}}): \ \mathrm{Tr}\,u=0 
\},
\qquad 
\Delta_\textup{Dir}^\domain u :=\Delta u,
\,\, u\in D(\Delta_\textup{Dir}^\domain),
\]
where for $\nu\in\R$, $k\in\N_0$, and $1<p<\infty$,
\[
W^{k,p}(\mathcal{O}, \omega_{\nu}^{\mathcal{O}}):=\{
u: \mathrm{D}^{\alpha}u \in L_p(\domain,\omega_\nu^\domain), \abs{\alpha}\leq k
\},
\]
endowed with the norm
\[
\norm{u\sep W^{k,p}(\mathcal{O}, \omega_{\nu}^{\mathcal{O}})}
:=
\ssgrklam{\sum_{\abs{\alpha}\leq k}
\int_\domain \Abs{D^\alpha u}^p\rho_\domain^{\nu}\,\dx}^{1/p},
\quad
u\in W^{k,p}(\mathcal{O}, \omega_{\nu}^{\mathcal{O}});
\]
$W^{2,p}_{\mathrm{Dir}}(\mathcal{O}, \omega_{\nu}^{\mathcal{O}})$ is endowed with the norm inherited from $W^{2,p}(\mathcal{O}, \omega_{\nu}^{\mathcal{O}})$.
Note that, other than for the spaces $H^\gamma_{p,\Theta}(\domain)$ with $\gamma\in\N_0$, see Lemma~\ref{lem:LotSpaces:properties}\ref{it:LotSpaces:N}, here the weight does not depend on the order of the derivatives. However, due to Hardy's inequality, the following holds.
\begin{lemma}\label{lem:domain:LinVer}
    Let $d\in\N$, let $\domain\subset\R^d$ be a bounded $C^2$-domain, let $1<p<\infty$, and let $d-1<\Theta<d+p-1$. Then
    \[
    W^{2,p}_\textup{Dir}(\domain,\omega_{\Theta+p-d}^\domain)
    =
    H^2_{p,\Theta-p}(\domain)
    \qquad\text{(equivalent norms)}.
    \]
\end{lemma}
\begin{proof}
    Since, on the one hand, by Lemma~\ref{lem:LotSpaces:properties}\ref{it:LotSpaces:density}, $C^2_0(\domain)\subset H^2_{p,\Theta-p}(\domain)$ dense, and, on the other hand, by \cite[Proposition~3.8]{LinVer2020},  $C^2_0(\domain)$ is dense in $W^{2,p}_\textup{Dir}(\domain,\omega_{\Theta+p-d}^\domain)$ for $d-1<\Theta<d+p-1$, it is enough to prove the norm equivalence 
    \begin{equation}
        \norm{u\sep W^{2,p}(\domain,\omega_{\Theta+p-d}^\domain)}
        \sim
        \norm{u\sep H^2_{p,\Theta-p}(\domain)},
        \quad u\in C^2_0(\domain).
    \end{equation}
    However, since $\domain$ is assumed to be bounded, it is enough to check that ``$\gtrsim$'' holds (the other direction is an immediate consequence of Lemma~\ref{lem:LotSpaces:properties}\ref{it:LotSpaces:N} and~\ref{it:LotSpaces:bddDom}). 
    Since the seminorms involving the second order derivatives in the two norms coincide, we merely have to prove that
    \[
    \int_\domain \abs{u}^p\rho_\domain^{\Theta-p-d}\,\dx
    +
    \int_\domain \abs{\partial_x u}^p\rho_\domain^{\Theta-d}\,\dx
    \lesssim
    \norm{u\sep W^{2,p}(\domain,\omega_{\Theta+p-d}^\domain)}^p,
        \quad u\in C^2_0(\domain).
    \]
    But this is a consequence of Hardy's inequality, which guarantees that for $\Theta<d+p-1$,
    \[
    \norm{u\sep W^{0,p}(\domain,\omega_{\Theta-p-d}^\domain)}
    \lesssim
    \norm{u\sep W^{1,p}(\domain,\omega_{\Theta-d}^\domain)},
    \quad
    u\in C^2_0(\domain),
    \]
    and
    \[
    \norm{u\sep W^{1,p}(\domain,\omega_{\Theta-d}^\domain)}
    \lesssim
    \norm{u\sep W^{2,p}(\domain,\omega_{\Theta+p-d}^\domain)},
    \quad
    u\in W^{2,p}(\domain,\omega_{\Theta+p-d}^\domain),
    \]
see \cite[Corollary~3.4]{LinVer2020} (or \cite[Theorems~8.2 and~8.4]{Kuf1980}).
\end{proof}

Recall that our goal is to prove the existence part in Theorem~\ref{thm:Poisson:main}. To this end we aim to prove that the operator $B$ from~\eqref{eq:ResolventY} is invertible for the range of weight parameters $\Theta$ and $\theta$ from~\eqref{eq:PoissonRange} .
For such $\Theta,\theta\in\R$, if $Bv=F$ for some $F\in Y^0_{2,\Theta+2,\theta-2}(\I)$ and some $v\in Y^2_{2,\Theta-2,\theta-2}(\I)$, then for almost all $\lambda\in \Gamma_{\frac{2-\theta}{2}}=\frac{2-\theta}{2}+i\R$ it holds that $v(\lambda,\cdot)\in H^2_{2,\Theta-3}(\I)$, $F(\lambda,\cdot)\in L_{2,\Theta+1}(\I)$ and
\[
(\lambda^2+D_\phi^2)v(\lambda,\cdot)=F(\lambda,\cdot)\quad \text{in } L_{2,\Theta+1}(\I).
\]
In view of Lemma~\ref{lem:domain:LinVer} this is the same as saying that 
\[
u(\lambda,\cdot)=R(\lambda^2,-\Delta_\textup{Dir}^\I)F(\lambda,\cdot),
\]
where $R(\mu,-\Delta_\textup{Dir}^\I)$ is the resolvent of the unbounded operator $(-\Delta_\textup{Dir}^\I,H^2_{2,\Theta-3}(\I))$ in $L_{2,\Theta+1}(\I)$ at $\mu\in\rho(-\Delta_\textup{Dir}^\I)$, where $\rho(-\Delta_\textup{Dir}^\I)$ is the resolvent set of $-\Delta_\textup{Dir}^\I$ (we refer to~\cite[Chapter~10 and Appendix~G]{HytNeeVer+2017} for notions from the theory of unbounded operators and operator semigroups).
The following lemma collects some properties of $-\Delta_\textup{Dir}^\I$ and its resolvent. 
It is a slight alteration of parts of \cite[Corollary~6.2]{LinVer2020} applied to the one-dimensional domain $\I=(0,\kappa)$. For $0<\sigma<\pi$ we write
\[
\Sigma_\sigma:=\ggklam{z\in\C\setminus\{0\}\colon \abs{\arg(z)}<\sigma}.
\]

\begin{lemma}\label{lem:cor-dir}
Let $1<p<\infty$, let $0<\Theta<p$, and let 
    \[
    (\Delta_\textup{Dir}^\I,D(\Delta_\textup{Dir}^\I)):=(\Delta,H^2_{p,\Theta-p}(\I))
    \]
    be the Dirichlet Laplacian in $L_{p,\Theta+p}(\I)$, $\I=(0,\kappa)$; cf.~Lemma~\ref{lem:domain:LinVer}. Then the following assertions hold. 
    \begin{enumerate}[label=\textup{(\roman*)}]
        \item\label{it:cor-dir:spectrum} The spectrum of   $-\Delta^\I_{\mathrm{Dir}}$   is given by $\sigma(  -\Delta^\I_{\mathrm{Dir}} )=\sggklam{\grklam{n\frac{\pi}{\kappa}}^2: \ n\in \mathbb{N}}$.
        
        \item\label{it:cor-dir:sectorial}
        $-\Delta_{\mathrm{Dir}}^\I$ is sectorial with   angle of sectoriality   $\omega(-\Delta_{\mathrm{Dir}}^\I)=0$. 
            
        \item\label{it:cor-dir:normequiv} $(\Delta^\I_{\mathrm{Dir}},D(\Delta^\I_{\mathrm{Dir}}))$ is a closed and densely defined operator in $L_{p,\Theta+p}(\mathcal{I})$. Moreover,
\[ 
   \norm{\Delta u\sep L_{p,\Theta+p}(\I)} 
            \sim 
            \norm{u \sep H^2_{p,\Theta-p}(\I)},
            \quad 
            u\in H^2_{p,\Theta-p}(\I). 
        \] 
        \item\label{it:cor-dir:resolventEst} 
        Let $s\in\R\setminus\ggklam{\pm n\frac{\pi}{\kappa}\colon n\in\N}$, let $\lambda\in\Gamma_s=s+i\R$, let $f\in L_{p,\Theta+p}(\I)$, and let 
        \[
        u_\lambda:=u_{\lambda,f}:=R(\lambda^2,-\Delta_\textup{Dir}^\I)f \in L_{p,\Theta+p}(\I).
        \]
        Then $u_{\lambda,f}$ is the unique $u\in H^2_{p,\Theta-p}(\D)$ such that $(\lambda^2+\Delta)u=f$. Moreover,
\begin{equation}\label{eq:cor-dir:resolventEst}
            \sum_{j=0}^2 \abs{\lambda}^j \norm{u_{\lambda,f}\sep H^{2-j}_{p,\Theta-p+jp}(\I)}      \lesssim 
            \norm{f\sep L_{p,\Theta+p}(\I)}, 
             \quad\lambda\in\Gamma_s,\,\, f\in L_{p,\Theta+p}(\I).
        \end{equation}
    \end{enumerate}    
\end{lemma}

\begin{proof}
      It is well-known that~\ref{it:cor-dir:spectrum} holds for $p=2$ and $\Theta=1$, see, e.g., \cite[pp.~49-50]{Gri1992}. 
      The general case is thus a consequence of \cite[Corollary~6.2(1)]{LinVer2020}, which states the independence of the spectrum on $p$ and $\Theta$. 
      Assertion~\ref{it:cor-dir:sectorial} is an immediate consequence of~\cite[Corollary~6.2(2)]{LinVer2020}, whereas~\ref{it:cor-dir:normequiv} follows from~\cite[Corollary~6.2(3)]{LinVer2020} together with Lemma~\ref{lem:domain:LinVer} and the fact that $0\in\rho(-\Delta^\I_\textup{Dir})$, as follows from~\ref{it:cor-dir:spectrum}. 
      Thus, we only have to prove~\ref{it:cor-dir:resolventEst}, which mainly follows from the sectoriality of $-\Delta_\textup{Dir}^\I$, its spectral properties and the interpolation properties of the scale $H^\gamma_{p,\Theta}(\I)$. 
      We argue as follows: 
      First note that, due to~\ref{it:cor-dir:spectrum}, $\lambda^2\in \rho(-\Delta^\I_\textup{Dir})$ for all $\lambda\notin\ggklam{\pm n\frac{\pi}{\kappa}\colon n\in\N}$.
      Thus, by the definition of the resolvent, for all $\lambda\in\C\setminus\ggklam{\pm n\frac{\pi}{\kappa}\colon n\in\N}$ and for all $f\in L_{p,\Theta+p}(\I)$, $u_{\lambda,f}$ is the unique $u\in H^2_{p,\Theta-p}(\I)$ such that $\grklam{\lambda^2+\Delta}u=f$.
      It thus remains to prove~\eqref{eq:cor-dir:resolventEst}. 
      To this end, fix $s\in\R\setminus\ggklam{\pm n\frac{\pi}{\kappa}\colon n\in\N}$ and fix $0<\sigma<\pi$.
      The sectoriality of $-\Delta^\I_\textup{Dir}$ yields that there is a finite constant $C_\sigma>0$ such that
      \[
      \abs{\lambda}^2 \norm{u_{\lambda,f}\sep L_{p,\Theta+p}(\I)}
      \leq
      C_\sigma \norm{f\sep L_{p,\Theta+p}(\I)},
      \quad 
      \lambda\in\C,\,\, \lambda^2\in\C\setminus\overline{\Sigma_\sigma}, 
      \,\, f\in L_{p,\Theta+p}(\I).
      \]
      Since $ \Delta u_{\lambda,f} = \lambda^2u_{\lambda,f}+\Delta u_{\lambda,f}-\lambda^2u_{\lambda,f} = f-\lambda^2 u_{\lambda,f}$, the last estimate also yields that
      \[
      \norm{\Delta u_{\lambda,f}\sep L_{p,\Theta+p}(\I)}
      \leq
      (C_\sigma+1) \norm{f\sep L_{p,\Theta+p}(\I)},
      \quad 
      \lambda\in\C,\,\, \lambda^2\in\C\setminus\overline{\Sigma_\sigma}, 
      \,\, f\in L_{p,\Theta+p}(\I),
      \]
      so that, by~\ref{it:cor-dir:normequiv}, we obtain 
      \[
      \abs{\lambda}^2 \norm{u_{\lambda,f}\sep L_{p,\Theta+p}(\I)}
      +
      \norm{u_{\lambda,f}\sep H^2_{p,\Theta-p}(\I)}
      \leq
      \grklam{2C_\sigma+1} \norm{f\sep L_{p,\Theta+p}(\I)},
      \]
      for all $\lambda\in\C$ with $\lambda^2\in\C\setminus\overline{\Sigma_\sigma}$ and all $f\in L_{p,\Theta+p}(\I)$ with the same $C_\sigma$ as above.
      Moreover, using the fact that, by the interpolation statement from Lemma~\ref{lem:LotSpaces:properties}\ref{it:LotSpaces:interpolation},
      \[
      \geklam{H^2_{p,\Theta-p}(\I),\abs{\lambda}^2 L_{p,\Theta+p}(\I)}_{1/2}=\abs{\lambda}H^1_{p,\Theta}(\I),
      \]
      we obtain that
\begin{equation}\label{eq:cor-dir:resolventEst:a}            
\sum_{j=0}^2 \abs{\lambda}^j \norm{u_{\lambda,f}\sep H^{2-j}_{p,\Theta-p+jp}(\I)}
\lesssim_\sigma 
\norm{f\sep L_{p,\Theta+p}(\I)}, 
\quad\lambda\in\Gamma_s\cap S_\sigma,\,\, f\in L_{p,\Theta+p}(\I),
\end{equation}
with $S_\sigma:=\C\setminus\grklam{\overline{\Sigma_{\sigma/2}}\cup \Sigma^c_{\pi-\sigma/2}}$ (note that $\lambda^2\in\C\setminus\overline{\Sigma_\sigma}$ if, and only if, $\lambda\in S_\sigma$). 
Then $\Gamma_s\cap S_\sigma^c\subset \rho(-\Delta_\textup{Dir}^\I)$, so that $\Gamma_s\cap S_\sigma^c\ni \lambda\mapsto R(\lambda^2,-\Delta_\textup{Dir}^\I)$ is continuous. Thus, due to the compactness of $\Gamma_s\cap S_\sigma^c$, the set $\ggklam{R(\lambda^2,-\Delta_\textup{Dir}^\I)\colon\lambda\in \Gamma_s\cap S_\sigma^c}$ is bounded in the operator norm on $L_{p,\Theta+p}(\I)$. Therefore,
\[
\norm{u_{\lambda,f}\sep L_{p,\Theta+p}(\I)}
\lesssim
\norm{f\sep L_{p,\Theta+p}(\I)},
\quad
\lambda\in\Gamma_s\cap S_\sigma^c,\,\, f\in L_{p,\Theta+p}(\I).
\]
With similar arguments as above (and since $\I$ is bounded and Lemma~\ref{lem:LotSpaces:properties}\ref{it:LotSpaces:bddDom} holds), this yields that 
\begin{align*}
\sum_{j=0}^2 \abs{\lambda}^j \norm{u_{\lambda,f}\sep H^{2-j}_{p,\Theta-p+jp}(\I)}
&\lesssim
\norm{u_{\lambda,f}\sep H^2_{p,\Theta-p}(\I)}\\
&\lesssim
\norm{\Delta u_{\lambda,f}\sep L_{p,\Theta+p}(\I)}\\
&\lesssim
(1+\abs{\lambda}^2)\norm{f\sep L_{p,\Theta+p}(\I)}\\
&\lesssim
\norm{f\sep L_{p,\Theta+p}(\I)},
\qquad\qquad
\lambda\in\Gamma_s\cap S_\sigma^c,\,\, f\in L_{p,\Theta+p}(\I).    
\end{align*}
Together with~\eqref{eq:cor-dir:resolventEst:a}, this proves the assertion.
\end{proof}

Now we are finally ready to prove the existence part of Theorem~\ref{thm:Poisson:main}.

\begin{proof}[Proof of Theorem~\ref{thm:Poisson:main}\ref{it:Poisson:main:existence}] 
As outlined above, it is sufficient to prove that the operator $B$ introduced in~\eqref{eq:ResolventY} is invertible. 
To this end, let $F\in Y^0_{2,\Theta+2,\theta-2}(\I)$ with $\Theta$ and $\theta$ satisfying~\eqref{eq:PoissonRange}. 
We need to show that there exists a unique $v\in Y^2_{2,\Theta-2,\theta-2}(\I)$ satisfying
\begin{equation}\label{eq:Poisson:Mellin}
    Bv=F.
\end{equation}
Since $F\in Y^0_{2,\Theta+2,\theta-2}(\I)$, $F$ has a version (also denoted by $F$), such that $F(\lambda,\cdot)\in L_{2,\Theta+1}(\I)$ for all $\lambda\in \Gamma_{\frac{2-\theta}{2}}$. 
Thus, by Lemma~\ref{lem:cor-dir} and since $(\theta-2)/2\in\R\setminus\ggklam{\pm n\frac{\pi}{\kappa}\colon n\in\N}$ and $0<\Theta-1<2$ as $\theta$ and $\Theta$ satisfy~\eqref{eq:PoissonRange}, for all $\lambda\in\Gamma_{\frac{2-\theta}{2}}$, there exists a unique $u_\lambda\in H^2_{2,\Theta-3}(\I)$ such that $(\lambda^2+ D_\phi^2) u_\lambda=F(\lambda,\cdot)$ in $L_{2,\Theta+1}(\I)$ and
\begin{equation}\label{eq:Poisson:resolventEst}
\sum_{j=0}^2 \abs{\lambda}^j\norm{u_\lambda\sep H^{2-j}_{2,\Theta-3+2j}(\I)}
\lesssim
\norm{F(\lambda,\cdot)\sep L_{2,\Theta+1}(\I)},
\quad \lambda\in\Gamma_{\frac{2-\theta}{2}}.
\end{equation}
Moreover, there exists an essentially unique $v\colon\Gamma_{\frac{2-\theta}{2}}\times\I\to \C$, such that $(t,\phi)\mapsto v(\frac{2-\theta}{2}+it,\phi)$ is Borel measurable and $v(\lambda,\cdot)=u_\lambda$ in $L_{p,\Theta-3}(\I)$ for almost all $\lambda\in\Gamma_{\frac{2-\theta}{2}}$ (this may be verified by using the continuity of the resolvent together with basic arguments from measure theory, see, e.g.,~\cite[Proposition~1.12.25]{HytNeeVer+2016}).
Due to~\eqref{eq:Poisson:resolventEst},
\[
\frac{1}{2\pi i} 
\sum_{j=0}^2 
\int_{\frac{2-\theta}{2}-i\infty}^{\frac{2-\theta}{2}+i\infty}
\abs{\lambda}^{2j}\norm{v(\lambda,\cdot)\sep H^{2-j}_{2,\Theta-3+2j}(\I)}^2\,\mathrm{d}\lambda
\lesssim
\frac{1}{2\pi i} 
\int_{\frac{2-\theta}{2}-i\infty}^{\frac{2-\theta}{2}+i\infty}
\norm{F(\lambda,\cdot)\sep L_{2,\Theta+1}(\I)}^2\,\mathrm{d}\lambda.
\]
Thus $v\in Y^2_{2,\Theta-2,\theta-2}(\I)$, $Bv=F$, and
\[
\norm{v\sep Y^2_{2,\Theta-2,\theta-2}(\I)}
\lesssim
\norm{F\sep Y^0_{2,\Theta+2,\theta-2}(\I)}.
\]
Moreover, due to the uniqueness of $u_\lambda$, $\lambda\in\Gamma_{\frac{2-\theta}{2}}$, $v$ is the unique element in $Y^2_{2,\Theta-2,\theta-2}(\I)$ satisfying~\eqref{eq:Poisson:Mellin}. Thus, $B$ is surjective and injective and the theorem is proven.
\end{proof}

\begin{appendix}

\section{Differential calculus in polar coordinates}\label{app:polar:coord}

In this appendix we collect some fundamental results, which we need when switching from Cartesian to polar coordinates, in particular, when ``translating'' derivatives with respect to one coordinate system into derivatives with respect to the other.
Recall that for $\I:=(0,\kappa)$ by $\Phi\colon\widetilde{\D}:=(0,\infty)\times\I\to \D$ we denote the polar coordinate transform which is a $ C^\infty$ diffeomorphism.
Also recall the definition of the rotation matrices $A=A(\phi)$, $\phi\in\I$, from~\eqref{matrix-A}.
Finally, let $\tilde{\rho}_\circ:=\rho_\circ\circ\Phi$.

\begin{lemma}\label{lem:FirstDerivative}
If $g\in \mathscr{D}'(\D)$, then $g\circ\Phi\in \mathscr{D}'(\widetilde{\D})$ and
\begin{equation}\label{eq:FristDerivative}
(D_x g)\circ\Phi = A \begin{pmatrix}
D_r\\
\tilde{\rho}_\circ^{-1} D_\phi
\end{pmatrix}
(g\circ\Phi).    
\end{equation}
In particular, $D^\alpha_x g \in L_{1,\loc}(\D)$ for all $\abs{\alpha}\leq \gamma$ if, and only if, $D^{\alpha}_{(r,\phi)} (g\circ\Phi)\in L_{1,\loc}(\widetilde{\D})$ for all $\abs{\alpha}\leq \gamma$.
\end{lemma}
\begin{proof}
For $g\in C_0^\infty(\D)$ the assertion follows simply by using the chain rule. The assertion for arbitrary $g\in\mathscr{D}'(\D)$ then follows from  the usual rules of distributional calculus.
\end{proof}

Formula~\eqref{eq:FristDerivative} generalizes to higher order derivatives the following way.

\begin{lemma}[{\cite[p.~1556]{DahWei2015}}]\label{lem:higherorder}
For $\alpha\in\N_0^2\setminus\{0\}$ let 
$$
    \Lambda_\alpha := \left\{ \beta=(\beta_1,\beta_2)\in\N_0^2\setminus\{0\} :\ \abs{\beta}\leq\abs{\alpha}\right\}.
$$
Then, for all $\beta\in\Lambda_{\alpha}$, there exist trigonometric polynomials $T_{\alpha,\beta}$ on $\I$ such that 
\begin{align}\label{eq:higherorder}
(D_x^\alpha g)\circ\Phi 
= 
\sum_{\beta\in\Lambda_{\alpha}} 
T_{\alpha,\beta} \, \tilde{\rho}_\circ^{\beta_1-\abs{\alpha}}\, D_r^{\beta_1}\,D_\phi^{\beta_2} (g\circ\Phi), \qquad g\in\mathscr{D}'(\D).
	\end{align}
\end{lemma}
\begin{proof}
We argue by mathematical induction on $\gamma:=\abs{\alpha}$.
The base case $\gamma=1$ immediately follows from \autoref{lem:FirstDerivative}, which yields (with $e_1:=(1,0)^T$ and $e_2:=(0,1)^T$)
\[
(D_xg)\circ\Phi
=
\begin{pmatrix}
(D_x^{e_1} g)\circ\Phi\\
(D_x^{e_2} g)\circ\Phi
\end{pmatrix}
=
\begin{pmatrix}
T_{e_1,e_1} \, D_r + T_{e_1,e_2}\, \tilde{\rho}_\circ^{-1}\,D_\phi\\
T_{e_2,e_1} \, D_r + T_{e_2,e_1}\, \tilde{\rho}_\circ^{-1}\,D_\phi
\end{pmatrix}
(g\circ\Phi)
\]
with $T_{e_1,e_1}(\phi):=T_{e_2,e_2}(\phi):=\cos(\phi)$, $T_{e_1,e_2}(\phi):=-\sin(\phi)$, and $T_{e_2,e_1}(\phi)=\sin(\phi)$, $\phi\in \I$.
For the induction step $\gamma\mapsto\gamma+1$ we assume that for some $\gamma\in\N$ the assertion holds for all $\alpha\in\N_0^2\setminus\{0\}$ with $\abs{\alpha}\leq\gamma$. 
Take $\widetilde{\alpha}\in\N_0^2$ with $\abs{\widetilde{\alpha}}=\gamma+1$. Then there exists $\alpha\in\N_0^2$ with $\abs{\alpha}=\gamma$ such that $\widetilde{\alpha}=\alpha + e_i$ for some $i\in\{1,2\}$. Thus, using the base step and the induction hypothesis, we obtain
\begin{align*}
(D_x^{\widetilde{\alpha}}g)\circ\Phi
&=
(D_x^{e_i}(D_x^\alpha)g)\circ\Phi\\
&=
T_{e_i,e_1} D_r((D_x^\alpha g)\circ\Phi)
+
T_{e_i,e_2} \tilde{\rho}_\circ^{-1} D_\phi((D_x^\alpha g)\circ\Phi)\\
&=
T_{e_i,e_1} D_r\ssgrklam{\sum_{\beta\in\Lambda_\alpha} 
T_{\alpha,\beta} \, \tilde{\rho}_\circ^{\beta_1-\abs{\alpha}}\, D_r^{\beta_1}\,D_\phi^{\beta_2} (g\circ\Phi)}\\
&\quad +
T_{e_i,e_2} \tilde{\rho}_\circ^{-1}D_\phi \ssgrklam{\sum_{\beta\in\Lambda_\alpha} 
T_{\alpha,\beta} \, \tilde{\rho}_\circ^{\beta_1-\abs{\alpha}}\, D_r^{\beta_1}\,D_\phi^{\beta_2} (g\circ\Phi)}
=:I+I\!I.
\end{align*}
Using the product rule, the first summand becomes
\begin{align*}
I
&=
\sum_{\beta\in\Lambda_\alpha} T_{e_i,e_1}  T_{\alpha,\beta} 
\sgrklam{ (\beta_1-\abs{\alpha}) \tilde{\rho}_\circ^{\beta_1-1-\abs{\alpha}}\, D_r^{\beta_1}\,D_\phi^{\beta_2} (g\circ\Phi)
+
\tilde{\rho}_\circ^{\beta_1-\abs{\alpha}}\, D_r^{\beta_1+1}\,D_\phi^{\beta_2} (g\circ\Phi)}\\
&=
\sum_{\beta\in\Lambda_\alpha} T_{e_i,e_1}  T_{\alpha,\beta} 
\sgrklam{ (\beta_1-\abs{\alpha}) \tilde{\rho}_\circ^{\beta_1-\abs{\widetilde{\alpha}}}\, D_r^{\beta_1}\,D_\phi^{\beta_2} (g\circ\Phi)
+
\tilde{\rho}_\circ^{\beta_1+1-\abs{\tilde{\alpha}}}\, D_r^{\beta_1+1}\,D_\phi^{\beta_2} (g\circ\Phi)}\\
&=
\sum_{\beta\in\Lambda_{\tilde{\alpha}}}
T_{\tilde{\alpha},\beta}^I
\tilde{\rho}_\circ^{\beta_1-\abs{\tilde{\alpha}}}
D_r^{\beta_1} D_\phi^{\beta_2}(g\circ\Phi)
\end{align*}
with suitable trigonometric polynomials $T_{\tilde{\alpha},\beta}^I$ for $\beta\in\Lambda_{\tilde{\alpha}}$.
Similarly,
\begin{align*}
I\!I
&=
\sum_{\beta\in\Lambda_\alpha} T_{e_i,e_2} \tilde{\rho}_\circ^{\beta_1-\abs{\tilde{\alpha}}}\,  D_r^{\beta_1} \sgrklam{
D_\phi(T_{\alpha,\beta}) \,D_\phi^{\beta_2} (g\circ\Phi)
+
T_{\alpha,\beta} \, D_\phi^{\beta_2+1}(g\circ\Phi)}\\
&=
\sum_{\beta\in\Lambda_\alpha} T_{e_i,e_2} D_\phi(T_{\alpha,\beta}) \tilde{\rho}_\circ^{\beta_1-\abs{\tilde{\alpha}}}\,  D_r^{\beta_1} \,D_\phi^{\beta_2} (g\circ\Phi)
+
T_{e_i,e_2} T_{\alpha,\beta} \tilde{\rho}_\circ^{\beta_1-\abs{\tilde{\alpha}}} \, D_\phi^{\beta_2+1}(g\circ\Phi)\\
&=
\sum_{\beta\in\Lambda_{\tilde{\alpha}}}
T^{I\!I}_{\tilde{\alpha},\beta}
\tilde{\rho}_\circ^{\beta_1-\abs{\tilde{\alpha}}} D_r^{\beta_1}D_{\phi}^{\beta_2}(g\circ\Phi)
\end{align*}
with suitable trigonometric polynomials  $T^{I\!I}_{\tilde{\alpha},\beta}$ for $\beta\in\Lambda_{\tilde{\alpha}}$.
Thus, 
\[
(D_x^{\tilde{\alpha}} g)\circ\Phi 
= 
\sum_{\beta\in\Lambda_{\tilde{\alpha}}} 
T_{\tilde{\alpha},\beta} \, \tilde{\rho}_\circ^{\beta_1-\abs{\tilde{\alpha}}}\, D_r^{\beta_1}\,D_\phi^{\beta_2} (g\circ\Phi)
\]
with $T_{\tilde{\alpha},\beta}:=T^{I}_{\tilde{\alpha},\beta}+T^{I\!I}_{\tilde{\alpha},\beta}$ for $\beta\in\Lambda_{\tilde{\alpha}}$.
\end{proof}

Since the matrix $A=A(\phi)$ is orthogonal, multiplication by  $A(\phi)\in\R^{2\times 2}$ preserves norms on $\R^2$, i.e., for $1<p<\infty$ it holds that
\begin{align}
\label{eq:orth-pnorm}
	\abs{A(\phi)\, x}_p \sim \abs{x}_p, \qquad \phi\in\I, \quad x\in\R^2;
\end{align}
the constants in the equivalence depend solely on $p$.
It is easily seen that the same is true for its (component-wise) classical derivatives $(\partial_\phi^n A)(\phi)$, $n\in\N$.
In particular, for all $n\in\N_0$,
\begin{align*}
    \partial_\phi^n A \in \mathcal{M}
    := \ggklam{M\colon \I \to \R^{2\times 2} :\  & M \text{ Borel-measurable, } M(\phi) \text{ orthogonal for all } \phi\in\I}.
\end{align*}
The following lemma extends~\eqref{eq:orth-pnorm} to $\norm{\,\cdot \sep H^\gamma_{p,\Theta}(\I)}$-norms and is used in the proof of \autoref{thm:polar}. 

\begin{lemma}\label{lem:orthogonality}
	Let $\gamma\in\N_0$, $1<p<\infty$, and $\Theta\in\R$.
	Furthermore, let $M\colon \I\to\R^{2\times 2}$ be such that $\partial_\phi^n M \in \mathcal{M}$ for all $n\in\N_0$.  
	Then 
	\begin{align}\label{eq:Aphi:estim}
		\norm{Mv \sep H^\gamma_{p,\Theta}(\I)} 
		\sim \norm{v \sep H^\gamma_{p,\Theta}(\I)},
		\qquad v=(v_1,v_2)\in L_{1,\loc}(\I;\R^2).
	\end{align}
\end{lemma}
\begin{proof}
	We proceed by mathematical induction on $\gamma$.
	For $\gamma=0$ we have to consider $H^0_{p,\Theta}(\I)=L_{p,\Theta}(\I)$. 
	Obviously, $v$ is Borel-measurable if, and only if, $Mv$ is. 
	Further, since $M\in\mathcal{M}$, 
	\begin{align*}
		\norm{v \sep H^0_{p,\Theta}(\I)} 
		&\sim \bigg( \int_\I \abs{v(\phi)}_p^p  w_\Theta(\phi)\d\phi \bigg)^{1/p} \\
		&\sim \bigg( \int_\I 	\abs{M(\phi)\,v(\phi)}_p^p   w_\Theta(\phi)\d\phi \bigg)^{1/p} 
		\sim \norm{M\,v \sep H^0_{p,\Theta}(\I)}.
	\end{align*}
Now suppose that \eqref{eq:Aphi:estim} holds for some $\gamma\in\N_0$.
	Then \autoref{lem:LotSpaces:properties}\ref{it:LotSpaces:lifting} and~\ref{it:LotSpaces:indexshift} together with the chain rule show
	\begin{align*}
		\norm{M v \sep H^{\gamma+1}_{p,\Theta}(\I)}
		&\sim \norm{M v \sep H^\gamma_{p,\Theta}(\I)} 
		+ 
 		\norm{ D_\phi(M v) \sep H^\gamma_{p,\Theta+p}(\I)}\\
 		&\lesssim 
		\norm{M\, v \sep H^\gamma_{p,\Theta}(\I)}
		+
		\norm{ (\partial_\phi M)\, v \sep H^\gamma_{p,\Theta+p}(\I)} 
		+ 
		\norm{ M D_\phi v \sep H^\gamma_{p,\Theta+p}(\I)}.
	\end{align*}
Thus, applying three times the induction hypothesis, as well as \autoref{lem:LotSpaces:properties}\ref{it:LotSpaces:bddDom}, \ref{it:LotSpaces:indexshift}, and~\ref{it:LotSpaces:lifting},
	we obtain one of the asserted estimates:
	\begin{align*}
		\norm{M v \sep H^{\gamma+1}_{p,\Theta}(\I)}
		&\lesssim \norm{v \sep H^\gamma_{p,\Theta}(\I)} 
		+ \norm{ v \sep H^\gamma_{p,\Theta+p}(\I)} + \norm{ D_\phi v \sep H^\gamma_{p,\Theta+p}(\I)} \\
		&\lesssim \norm{v \sep H^\gamma_{p,\Theta}(\I)} + \norm{\psi\, D_\phi v \sep H^\gamma_{p,\Theta}(\I)} \\
		&\lesssim \norm{v \sep H^{\gamma+1}_{p,\Theta}(\I)}.
	\end{align*}
The reverse estimate follows from this estimate, too, since
  $M(\cdot)^{T}=M(\cdot)^{-1}$ and obviously $\partial_\phi^n M(\cdot)^T\in\mathcal{M}$ for all $n\in\N_0$. 
\end{proof}

\section{Bessel potential spaces}\label{sec:Bessel}
In this final appendix we recall the definition of Bessel potential spaces $H^\gamma_p(\R^d)$ and gather some of their well-known properties which are frequently used throughout the manuscript.
To do so, we use a Fourier analytical approach similar to~\cite[Section~1.3.2]{Tri1992}.

For $s\in\R$, let 
\[
    (1-\Delta)^s\colon\mathscr{S}'(\R^d)\to\mathscr{S}'(\R^d),\qquad f\mapsto \mathscr{F}^{-1}\grklam{\ggklam{\xi\mapsto (1+\abs{\xi})^s(\mathscr{F}f)(\xi)}},
\]
where $\mathscr{S}'(\R^d)\subset\mathscr{D}'(\R^d)$ denotes the space of tempered distributions (defined as the topological dual of the Schwartz space of rapidly decreasing functions), $\mathscr{F}\colon \mathscr{S}'(\R^d)\to\mathscr{S}'(\R^d)$ is the Fourier transform and $\mathscr{F}^{-1}$ its inverse.
Then for $d\in\N$, $\gamma\in\R$, and $1<p<\infty$,
\[
    H^\gamma_p(\R^d)
    := \ggklam{u\in\mathscr{S}'(\R^d)\colon (1-\Delta)^{\gamma/2}f\in L_p(\R^d)},
\]
denotes the space of Bessel potentials
endowed with the norm
\[
    \norm{f \sep H^\gamma_p(\R^d)} := \norm{(1-\Delta)^{\gamma/2}f \sep L_p(\R^d)}, \qquad f\in H^\gamma_p(\R^d).
\]

\begin{remark}\label{rem:Bessel:distribution}
In some texts, like, for instance,~\cite{Kry2008}, the space $H^\gamma_p(\R^d)$ is defined as the space of all distributions $f\in \mathscr{D}'(\R^d)$ (not necessarily tempered!) for which there exists $h\in L_p(\R^d)$ with $f=(1-\Delta)^{-\gamma/2}h$. 
However, note that such an $f$ can always be extended to become a tempered distribution, see~\cite[Theorem~13.1.2(i) together with Remark~13.3.3]{Kry2008}.
\end{remark}

For the convenience of the reader, the next lemma collects the properties of $H^\gamma_p(\R^d)$, as needed in our arguments. 
Therein, for $1<p<\infty$, we use $W^k_p(\R^d)$ to denote the classical $L_p$-Sobolev space of order $k\in\N_0$.

\begin{lemma}[Properties of Bessel potential spaces]\label{lem:bessel}
Let $d\in\N$, $1<p,p_0,p_1<\infty$, as well as $\gamma,\gamma_0,\gamma_1\in\R$. Then the following assertions hold.
\begin{enumerate}[label=\textup{(\roman*)}]
\item\label{it:bessel:BS} $H^{\gamma}_{p}(\R^d)$ is a reflexive Banach space.

\item\label{it:bessel:test} $ C_0^\infty(\R^d)$ is dense in $H^{\gamma}_{p}(\R^d)$.

\item\label{it:bessel:sobolev} If $\gamma\in\N_0$, then $H^{\gamma}_{p}(\R^d) = W_p^\gamma(\R^d)$ with equivalent norms.

\item\label{it:bessel:dual} $\big( H^{\gamma}_{p}(\R^d) \big)' = H^{-\gamma}_{p'}(\R^d)$, $\displaystyle \frac{1}{p}+\frac{1}{p'}=1$, with equivalent norms.

\item\label{it:bessel:interpolation} If for $0<\vartheta<1$ there holds
\[
    \frac{1}{p} = \frac{1-\vartheta}{p_0} + \frac{\vartheta}{p_1} 
    \qquad \text{and}\qquad		
    \gamma = (1-\vartheta) \gamma_0+ \vartheta \gamma_1,
\]
then
\[
    \big[ H^{\gamma_0}_{p_0}(\R^d), H^{\gamma_1}_{p_1}(\R^d)\big]_{\vartheta}
    = H^{\gamma}_{p}(\R^d)
\]
with equivalent norms.

\item\label{it:bessel:multiplier} 
Let $(\zeta_\nu)_{\nu\in\Z}$ denote a collection of $ C^\infty(\R^d)$-functions
such that for some $c>1$
\[
    \abs{\partial^\alpha \zeta_\nu(x)}
    \lesssim_\alpha c^{-\abs{\alpha}\nu},
    \qquad \nu\in\Z,\quad \alpha\in\N_0^d, \quad x\in\R^d,
\]
then
\[
    \norm{\zeta_\nu(c^\nu\cdot) f \sep H^{\gamma}_{p}(\R^d)}
    \lesssim \norm{f \sep H^{\gamma}_{p}(\R^d)}, \qquad f \in H^{\gamma}_{p}(\R^d),\quad \nu\in\Z.
\]

\item\label{it:bessel:differomorph} For $k\in\Z$ and $c>1$ we have
\[
    \norm{f(c^k\cdot) \sep H^{\gamma}_{p}(\R^d)} 
    \lesssim \norm{f \sep H^{\gamma}_{p}(\R^d)}, \qquad f \in H^{\gamma}_{p}(\R^d).
\]

\item\label{it:bessel:localization} Let $(\zeta_k)_{k\in\N_0}$ denote a collection of $ C^\infty(\R^d)$-functions such that
\[
\sup_{x\in\R^d} \sum_{k\in\N_0} \abs{\partial^\alpha \zeta_k(x)} \leq C_\alpha, \qquad \alpha\in\N_0^d.
\]
Then, with some constant depending on $d$, $\gamma$, and $C_\alpha$, we have
\[
    \sum_{k\in\N_0} \norm{\zeta_k f \sep H^\gamma_p(\R^d)}^p
    \lesssim \norm{f \sep H^\gamma_p(\R^d)}^p, \qquad f\in H^\gamma_p(\R^d).
\]
If, in addition,
\[
    \inf_{x\in\R^d} \sum_{k\in\N_0} \abs{\zeta_k(x)}^p \geq \delta>0,
\]
then, with some constant depending on $\delta$, $d$, $\gamma$, and $C_\alpha$, we have
\[
    \norm{f \sep H^\gamma_p(\R^d)}^p
    \lesssim \sum_{k\in\N_0} \norm{\zeta_k f \sep H^\gamma_p(\R^d)}^p, \qquad f\in H^\gamma_p(\R^d).
\]
\end{enumerate}
\end{lemma}

\begin{proof}
For proofs of~\ref{it:bessel:BS}-\ref{it:bessel:sobolev} see~\cite[Theorems~13.3.7 and 13.3.12]{Kry2008}. 
For the other assertions one may use the coincidence of $H^\gamma_p(\R^d)$ with so-called Triebel-Lizorkin spaces $F^\gamma_{p,2}(\R^d)$ for all $\gamma\in\R$ and $1<p<\infty$, see, e.g.,~\cite[Definition~2.3.2 and Theorem~2.5.6]{Tri1983}.
Then, the duality statement~\ref{it:bessel:dual} follows from~\cite[Theorem~2.11.2]{Tri1983} and the complex interpolation formula in~\ref{it:bessel:interpolation} is a consequence of~\cite[Theorem~2.4.7]{Tri1983}.
Assertion~\ref{it:bessel:multiplier} follows, for instance, from the following multiplier assertion, proven, e.g., in~\cite[Theorem~4.2.2]{Tri1992}:
If $m\in\N$ is sufficiently large (compared to $\abs{\gamma}$ and $p$), then 
\[
    \norm{af \sep H^{\gamma}_{p}(\R^d)} 
    \lesssim_m \norm{f \sep H^{\gamma}_{p}(\R^d)} \sum_{\abs{\alpha}\leq m} \norm{\partial^\alpha a \sep L_\infty(\R^d)},
    \qquad a\in C^m_b(\R^d),\quad f \in H^{\gamma}_{p}(\R^d),
\]
where $C^m_b(\R^d)$ consists of all $C^m(\R^d)$-functions with bounded derivatives up to order $m$.
Thus,~\ref{it:bessel:multiplier} follows if we choose $a:=a_\nu:=\zeta_\nu(c^\nu\cdot)$, since, by assumption, for all $\alpha\in\N_0^d$,
\[
    \norm{\partial^\alpha a_\nu \sep L_\infty(\R^d)}
    = \norm{c^{\abs{\alpha}\nu} \, (\partial^\alpha \xi_\nu)(c^\nu \cdot) \sep L_\infty(\R^d)} 
    \leq C_\alpha, \qquad \nu\in\Z.
\]
The statement~\ref{it:bessel:differomorph} follows from the fact that if $\Psi\colon\R^d\to\R^d$ is an $m$-diffeomorphism with $m\in\N$ large enough (again  depending on $\abs{\gamma}$ and $p$), then (see, e.g.,~\cite[Theorem~4.3.2]{Tri1992})
\[
    \norm{u\circ \Psi \sep H^{\gamma}_{p}} 
    \lesssim \norm{u \sep H^{\gamma}_{p}}, \qquad u \in H^{\gamma}_{p}(\R^d).
\]
Finally, the localization result~\ref{it:bessel:localization} can been found in~\cite[Lemma~6.7]{Kry1999} and~\cite[Theorem~2.1]{Kry1994b}.
\end{proof}

\end{appendix}

\phantomsection
\addcontentsline{toc}{section}{References}
\bibliographystyle{amsplain}

\footnotesize

\providecommand{\bysame}{\leavevmode\hbox to3em{\hrulefill}\thinspace}
\providecommand{\MR}{\relax\ifhmode\unskip\space\fi MR }
\providecommand{\MRhref}[2]{%
  \href{http://www.ams.org/mathscinet-getitem?mr=#1}{#2}
}
\providecommand{\href}[2]{#2}


\end{document}